\documentclass{amsart}

\usepackage{amssymb}
\usepackage{amscd}
\usepackage[all]{xy}

\theoremstyle{plain}
\newtheorem{prop}{Proposition}[subsection]

\let\a\alpha
\let\b\beta
\let\g\gamma
\let\d\delta
\let\e\varepsilon
\let\z\zeta

\let\l\lambda
\let\m\mu

\let\s\sigma

\let\f\varphi

\let\L\Lambda

\let\Om\Omega

\def\A{\mathcal A}
\def\B{\mathcal B}
\def\C{\mathbb C}
\def\CC{\mathcal C}
\def\E{\mathcal E}
\def\F{\mathcal F}
\def\G{\mathcal G}
\def\I{\mathcal I}
\def\J{\mathcal J}
\def\K{\mathcal K}
\def\P{\mathbb{P}}
\def\S{\mathcal S}
\def\O{\mathcal O}
\def\U{\mathcal U}
\def\W{\mathbb{W}}

\def\Ker{{\mathcal Ker}}
\def\Coker{{\mathcal Coker}}
\def\Im{{\mathcal Im}}

\def\D{{\scriptscriptstyle \operatorname{D}}}
\def\H{\operatorname{H}}
\def\M{\operatorname{M}_{\P^2}}
\def\N{\operatorname{N}}
\def\T{{\scriptscriptstyle \operatorname{T}}}
\def\Hom{\operatorname{Hom}}
\def\Aut{\operatorname{Aut}}
\def\Ext{\operatorname{Ext}}
\def\GL{\operatorname{GL}}
\def\rank{\operatorname{rank}}
\def\spann{\operatorname{span}}

\def\Grass{\operatorname{Grass}}

\def\tensor{\otimes}
\def\isom{\simeq}
\def\lra{\longrightarrow}
\def\egal{\ar@{=}}

\def\ba{\begin{array}}
\def\ea{\end{array}}

\begin{document}


\title[moduli of plane sheaves of dimension one and multiplicity five]
{on the moduli spaces of semi-stable plane sheaves
of dimension one and multiplicity five}

\author{mario maican}

\address{Mario Maican \\
Institute of Mathematics of the Romanian Academy \\
Calea Grivi\c tei 21\\
010702 Bucharest \\
Romania}

\email{mario.maican@imar.ro}

\begin{abstract}
We find locally free resolutions of length one for all semi-stable sheaves
supported on curves of multiplicity five in the complex projective plane.
In some cases we also find geometric descriptions of these sheaves
by means of extensions.
We give natural stratifications for their moduli spaces
and we describe the strata as certain quotients modulo linear algebraic groups.
In most cases we give concrete descriptions of these quotients as fibre bundles.
\end{abstract}

\maketitle


\section{Introduction}

Let $\M(r,\chi)$ denote the moduli space of semi-stable sheaves $\F$ on the
complex projective plane $\P^2$ with support of dimension $1$, multiplicity $r$
and Euler characteristic $\chi$.
The Hilbert polynomial of $\F$ is $P_{\F}(t)=rt+\chi$ and the ratio $p(\F)=\chi/r$
is the slope of $\F$. We recall that $\F$ is semi-stable, respectively stable,
if $\F$ is pure (meaning that there are no proper subsheaves with support of
dimension zero) and any proper subsheaf $\F' \subset \F$ satisfies
$p(\F') \le p(\F)$, respectively $p(\F') < p(\F)$.
The spaces $\M(r,\chi)$ for $r \le 3$ are completely understood from the work
of Le Potier \cite{lepotier}, and others.
In \cite{drezet-maican} Dr\'ezet and the author studied the spaces $\M(4,\chi)$.
This paper is concerned with the geometry of the spaces $\M(5,\chi)$.
In view of the obvious isomorphism $\M(r,\chi) \isom \M(r,\chi+r)$ sending
the stable-equivalence class of a sheaf $\F$ to the stable-equivalence class of the twisted
sheaf $\F \tensor \O(1)$, it is enough to assume that $0 \le \chi \le 4$.
According to \cite{lepotier}, the spaces $\M(5,\chi)$ are projective, irreducible,
locally factorial, of dimension $26$ and smooth at all points given by stable sheaves.
In particular, $\M(5,\chi)$, $1\le \chi \le 4$, are smooth.

In this paper we shall carry out the same program as in \cite{drezet-maican}.
We shall decompose each moduli space into locally closed subvarieties,
called \emph{strata}, by means of cohomological conditions.
Given a stratum $X \subset \M(5,\chi)$, we shall find locally free sheaves $\A$ and $\B$
on $\P^2$ such that each sheaf $\F$ giving a point in $X$ admits a presentation
$$
0 \lra \A \stackrel{\f}{\lra} \B \lra \F \lra 0.
$$
The linear algebraic group $G=(\Aut(\A) \times \Aut(\B))/\C^*$ acts by conjugation on the
finite dimensional vector space $\W=\Hom(\A,\B)$.
Here $\C^*$ is embedded as the subgroup of homotheties.
The set of morphisms $\f$ appearing above is a locally closed subset $W \subset \W$,
which is invariant under the action of $G$.
We shall prove that a good or a categorical quotient of $W$ by $G$ exists and is isomorphic to $X$.
The existence of the good quotient does not follow from the geometric invariant
theory if $G$ is non-reductive, which, most of the time, will be our case.
In some cases we shall describe the sheaves in the strata by means of extensions.

Throughout this paper we keep the notations and conventions from \cite{drezet-maican}.
We work over the complex numbers. We fix a vector space $V$ over $\C$ of dimension $3$
and we identify $\P^2$ with the space $\P(V)$ of lines in $V$.
We fix a basis $\{ X, Y, Z \}$ of $V^*$.
If $\A$ and $\B$ are direct sums of line bundles on $\P^2$,
we identify $\Hom(\A,\B)$ with the space of matrices with entries in appropriate
symmetric powers of $V^*$, i.e. matrices with entries homogeneous polynomials
in $X, Y, Z$. We especially refer to the section of preliminaries in \cite{drezet-maican},
which contains most of the techniques that we shall use.

According to \cite{maican-duality},
there is a duality isomorphism $\M(r,\chi) \isom \M(r,-\chi)$ sending the stable-equivalence class
of a sheaf $\F$ to the stable-equivalence class of the dual sheaf
$\F^{\D}={\mathcal Ext}^1(\F,\omega_{\P^2})$.
This allows us to study the spaces $\M(5,\chi)$ in pairs.
Thus $\M(5,3)$ and $\M(5,2)$ are isomorphic and will be studied in section 2.
The spaces $\M(5,1)$ and $\M(5,4)$ are, likewise, isomorphic and will be treated in section 3.
The last section deals with $\M(5,0)$.
In the remaining part of this introduction we shall make a summary of results.


\subsection{The moduli spaces $\M(5,3)$ and $\M(5,2)$}

We shall decompose the moduli space $\M(5,3)$ into four strata:
an open stratum $X_0$, two locally closed strata $X_1, X_2$ and a closed stratum $X_3$.
$X_1$ is a proper open subset inside a fibre bundle over $\P^2 \times \N(3,2,3)$ with fibre $\P^{16}$.
Here $\N(3,2,3)$ is the moduli space of semi-stable Kronecker modules
$\tau \colon \C^2 \tensor V \to \C^3$.
$X_2$ is a proper open subset inside a fibre bundle over $\N(3,3,2)$ with fibre $\P^{17}$.
$X_3$ is isomorphic to the Hilbert flag scheme
of quintic curves in $\P^2$ containing zero-dimensional subschemes of length $2$.

A sheaf $\F$ from $\M(5,3)$ gives a point in $X_0$ if and only if the following cohomological conditions are
satisfied:
\[
h^0(\F(-1))=0, \qquad \quad h^1(\F)=0, \qquad \quad h^0(\F \tensor \Om^1(1))=1.
\]
Each semi-stable sheaf whose stable-equivalence class is in $X_0$ has a resolution of the form
\[
0 \lra 2\O(-2) \oplus \O(-1) \stackrel{\f}{\lra} 3\O \lra \F \lra 0.
\]
We consider the vector space $\W=\Hom(2\O(-2) \oplus \O(-1), 3\O)$
and the linear algebraic group
\[
G= (\Aut(2\O(-2) \oplus \O(-1)) \times \Aut(3\O))/\C^*
\]
acting on $\W$ by conjugation.
$\C^*$ is embedded as the subgroup of homotheties here.
The set of morphisms $\f$ occuring above forms an open $G$-invariant subset
$W \subset \W$ given by the following conditions: $\f$ is injective and $\f$ is not in the
orbit of a morphism represented by a matrix of the form
\[
\left[
\begin{array}{ccc}
\star & \star & \star \\
\star & \star & 0 \\
\star & \star & 0
\end{array}
\right] \qquad \text{or} \qquad \left[
\begin{array}{ccc}
\star & \star & \star \\
\star & \star & \star \\
\star & 0 & 0
\end{array}
\right].
\]
$W$ admits a geometric quotient $W/G$ modulo $G$ and $W/G \isom X_0$.
The information about $X_0$ is summarised in the second row of the following table.
The other rows of the table contain the analogous information about the remaining strata of $\M(5,3)$.
The last column gives the codimension of each stratum. For each $W$ there is a geometric
quotient $W/G$ modulo the canonical group $G$ acting by conjugation on the ambient vector space $\W$
of homomorphisms of sheaves and $W/G$ is isomorphic to the corresponding
stratum of $\M(5,3)$.

\begin{table}[ht]{Table 1. Summary for $\M(5,3)$.}
\begin{center}
\begin{tabular}{|c|c|c|c|}
\hline
\!\!\! {\Tiny stratum} \!\!\!
&
{\Tiny cohomological conditions}
&
{\Tiny subset $W \subset \W$ of morphisms $\f$}
&
\!\!\! {\Tiny codim.} \!\!\!
\\
\hline
$X_0$
&
\begin{tabular}{r}
$h^0(\F(-1))=0$ \\
$h^1(\F)=0$\\
$h^0(\F \tensor \Om^1(1))=1$
\end{tabular}
&
\begin{tabular}{c}
$2\O(-2) \oplus \O(-1) \stackrel{\vphantom{I} \f}{\lra} 3\O$ \\
$\f$ is injective \\
$\f$ is not equivalent to \\
${\displaystyle \left[
\ba{ccc}
\star & \star & \star \\
\star & \star & 0 \\
\star & \star & 0
\ea
\right]}$ or
${\displaystyle \left[ \ba{ccc}
\star & \star & \star \\
\star & \star & \star \\
\star & 0 & 0
\ea \right]_{{}_{\vphantom{I}}}}$
\end{tabular}
&
$0$ \\
\hline
$X_1$
&
\begin{tabular}{r}
$h^0(\F(-1))=0$ \\
$h^1(\F)=0$\\
$h^0(\F \tensor \Om^1(1))=2$
\end{tabular}
&
\begin{tabular}{c}
$2\O(-2) \oplus 2\O(-1) \stackrel{\vphantom{I} \f}{\lra} \O(-1) \oplus 3\O$ \\
$\f$ is injective \\
$\f_{12}=0$ \\
$\f_{11}$ has \\
linearly independent entries \\
$\f_{22}$ has linearly independent \\
maximal minors
\end{tabular}
&
$2$ \\
\hline
$X_2$
&
\begin{tabular}{r}
$h^0(\F(-1))=1$ \\
$h^1(\F)=0$\\
$h^0(\F \tensor \Om^1(1))=3$ \\
\end{tabular}
&
\begin{tabular}{c}
$3\O(-2) \stackrel{\vphantom{I} \f}{\lra} 2\O(-1) \oplus \O(1)$ \\
$\f$ is injective \\
$\f_{11}$ has linearly independent \\
maximal minors
\end{tabular}
&
$3$ \\
\hline
$X_{3}$
&
\begin{tabular}{r}
$h^0(\F(-1))=1$ \\
$h^1(\F)=1$\\
$h^0(\F \tensor \Om^1(1))=4$ \\
\end{tabular}
&
\begin{tabular}{c}
$\O(-3) \oplus \O(-1) \stackrel{\vphantom{I} \f}{\lra} \O \oplus \O(1)$ \\
$\f$ is injective \\
$\f_{12} \neq 0$ \\
$\f_{12} \nmid \f_{22}$
\end{tabular}
&
$4$ \\
\hline 
\end{tabular}
\end{center}
\end{table}

Applying to $X_i$ the duality isomorphism $\M(5,3) \to \M(5,2)$ of \cite{maican-duality} defined by
\[
\F \lra \F^{\D}(1) ={\mathcal Ext}^1(\F,\omega_{\P^2}) \tensor \O(1),
\]
we get a dual stratum $X_i^{\D} \subset \M(5,2)$
given by the cohomological conditions derived from Serre duality
(see 2.1.2 \cite{drezet-maican}). For instance, $X_0^{\D}$ consists of those sheaves
$\G$ in $\M(5,2)$ satisfying the conditions
\[
h^1(\G)=0, \qquad \quad h^0(\G(-1))=0, \qquad \quad h^1(\G \tensor \Om^1(1))=1.
\]
According to \cite{maican-duality}, lemma 3, taking the dual of each term in a locally free
resolution of length 1 for $\F$ gives a resolution for $\F^\D$.
Thus every sheaf $\G$ in $X_0^{\D}$ has a resolution of the form
\[
0 \lra 3\O(-2) \stackrel{\psi}{\lra} \O(-1) \oplus 2\O \lra \G \lra 0.
\]
The conditions on $\psi$ are the transposed conditions on the morphism $\f$ from above.
In this fashion we get a ``dual table'' for $\M(5,2)$. We omit the details.

Inside $X_0$ there is an open dense subset of sheaves that have a presentation of the form
\[
0 \lra 2\O(-2) \lra \Om^1(2) \lra \F \lra 0.
\]
The complement in $X_0$ of this subset, denoted $X_{01}$, has codimension $1$.
The generic sheaves giving points in $X_{01}$ have the form
\[
\O_C(1)(P_1+P_2+P_3+P_4-P_5),
\]
where $C \subset \P^2$ is a smooth
quintic curve, $P_1, \ldots, P_5$ are distinct points on $C$ and $P_1, P_2, P_3, P_4$ are in
general linear position.

The sheaves giving points in $X_1$ are precisely the non-split extension sheaves of the form
\[
0 \lra \G \lra \F \lra \C_x \lra 0,
\]
where $\G$ varies in $X_2^{\D}$ and $\C_x$ is the structure sheaf of a closed point
in the support of $\G$.

The sheaves $\G$ in $X_2^{\D}$ are either of the form $\J_Z(2)$, where $\J_Z \subset \O_C$
is the ideal sheaf of a zero-dimensional subscheme of length $3$ contained in a quintic curve $C$,
$Z$ not contained in a line, or they are extension sheaves of the form
\[
0 \lra \O_L(-1) \lra \G \lra \O_D(1) \lra 0,
\]
where $L$ is a line and $D$ is a quartic curve, that are not in the kernel
of the canonical map
\[
\Ext^1(\O_D(1),\O_L(-1)) \lra \Ext^1(\O(1),\O_L(-1)).
\]

The sheaves in $X_3$ are the twisted ideal sheaves $\J_Z(2) \subset \O_C(2)$ of zero-dimensional
subschemes $Z$ of length $2$ contained in quintic curves $C$.


\subsection{The moduli spaces $\M(5,1)$ and $\M(5,4)$}

We shall decompose the moduli space $\M(5,1)$ into four strata:
an open stratum $X_0$, two locally closed strata $X_1, X_2$ and a closed stratum $X_3$.
$X_0$ is a proper open subset inside a fibre bundle with base $\N(3,4,3)$ and fibre $\P^{14}$.
$X_1$ is a proper open subset inside a fibre bundle with base $\Grass(2,\C^6)$ and fibre $\P^{16}$.
$X_2$ is a proper open subset inside a fibre bundle with fibre $\P^{17}$ and base $Y \times \P^2$,
where $Y$ is the Hilbert scheme of zero-dimensional subschemes of $\P^2$ of length $2$.
$X_3$ is the universal quintic in $\P^2 \times \P(S^5 V^*)$.

The information about the cohomological conditions defining each stratum in $\M(5,1)$
and resolutions for semi-stable sheaves can be found in table 2 below.
This is organised as table 1, so we refer to the
previous subsection for the meaning of the different items.
Again, each $X_i$ is isomorphic to the corresponding geometric quotient $W/G$.
By duality, from table 2 can be obtained a table for $\M(5,4)$, which we do not include here.

\begin{table}[ht]{Table 2. Summary for $\M(5,1)$.}
\begin{center}
{
\begin{tabular}{|c|c|c|c|}
\hline
&
{\tiny cohomological conditions}
&
{\tiny subset $W \subset \W$ of morphisms $\f$}
&
{} \\
\hline
$X_0$
&
\strut
\begin{tabular}{r}
$h^0(\F(-1)) = 0$ \\
$h^1(\F)=0$\\
$h^0(\F \tensor \Om^1(1))=0$
\end{tabular}
&
\begin{tabular}{c}
$4\O(-2) \stackrel{\f}{\lra} 3\O(-1) \oplus \O$ \\
$\f$ is injective \\
$\f_{11}$ is not equivalent \\
to ${\displaystyle \left[
\ba{cccc}
\star & \star & \star & 0\\
\star & \star & \star & 0 \\
\star & \star & \star & 0
\ea
\right]}$ or \\
${\displaystyle \left[
\ba{cccc}
\star & \star & 0 & 0 \\
\star & \star & 0 & 0 \\
\star & \star & \star & \star
\ea
\right]}$ or
${\displaystyle \left[
\ba{cccc}
\star & 0 & 0 & 0 \\
\star & \star & \star & \star \\
\star & \star & \star & \star
\ea
\right]_{{}_{\vphantom{I}}}}$
\end{tabular}
&
$0$ \\
\hline
$X_1$
&
\begin{tabular}{r}
$h^0(\F(-1)) = 0$ \\
$h^1(\F)=1$\\
$h^0(\F \tensor \Om^1(1))=0$
\end{tabular}
&
\begin{tabular}{c}
$\O(-3) \oplus \O(-2) \stackrel{\vphantom{I} \f}{\lra} 2\O$ \\
$\f$ is injective \\
$\f_{12}$ and $\f_{22}$ are \\
linearly independent two-forms
\end{tabular}
&
$2$ \\
\hline
$X_2$
&
\begin{tabular}{r}
$h^0(\F(-1))=0$ \\
$h^1(\F)=1$\\
$h^0(\F \tensor \Om^1(1))=1$
\end{tabular}
&
\begin{tabular}{c}
$\O(-3) \! \oplus \! \O(-2) \! \oplus \! \O(-1) \stackrel{\f}{\lra} \O(-1) \! \oplus \! 2\O$ \\
$\f$ is injective \\
$\f_{13}=0$ \\
$\f_{12} \neq 0$, $\f_{12} \nmid \f_{11}$ \\
$\f_{23}$ has linearly independent entries
\end{tabular}
&
$3$ \\
\hline
$X_3$
&
\begin{tabular}{r}
$h^0(\F(-1))=1$ \\
$h^1(\F)=2$\\
$h^0(\F \tensor \Om^1(1))=3$
\end{tabular}
&
\begin{tabular}{c}
$2\O(-3) \stackrel{\vphantom{I} \f}{\lra} \O(-2) \oplus \O(1)$ \\
$\f$ is injective \\
$\f_{11}$ has linearly independent entries
\end{tabular}
&
$5$ \\
\hline
\end{tabular}
}
\end{center}
\end{table}

Inside $X_0$ there is an open dense subset consisting of sheaves of the form
$\J_Z(2)^{\D}$, where $Z \subset \P^2$ is a zero-dimensional scheme of length $6$
not contained in a conic curve, contained in a quintic curve $C$, and $\J_Z \subset \O_C$
is its ideal sheaf. The complement in $X_0$ of this open subset is the disjoint
union of two sets $X_{01}$ and $X_{02}$. The sheaves in $X_{01}$ occur as non-split
extensions of one of the following three kinds:
\begin{align*}
0 \lra \G \lra & \F \lra \O_L \lra 0, \\
0 \lra \E \lra & \F \lra \O_Y \lra 0, \\
0 \lra \O_L(-1) \lra & \F \lra \J_Z(1)^{\D} \lra 0.
\end{align*}
Here $L \subset \P^2$ is a line, $\G$ is in the exceptional divisor of $\M(4,0)$,
$\E$ is the twist by $-1$ of a sheaf in the stratum $X_3 \subset \M(5,3)$,
$Y \subset \P^2$ is a zero-dimensional scheme of length $3$ not contained in a line,
contained in the support of $\E$,
$Z \subset \P^2$ is a zero-dimensional scheme of length $3$ not contained in a line,
contained in a quartic curve $C$, and $\J_Z \subset \O_C$ is its ideal sheaf.
Not all of the above extension sheaves are in $X_{01}$, namely there are certain conditions
that must be satisfied for which we refer to subsection 3.3.
For $X_{02}$ we can be more specific.
A sheaf $\F$ gives a point in $X_{02}$ precisely if it is an extension of the form
\[
0 \lra \O_{C'} \lra \F \lra \O_C \lra 0
\]
and satisfies $\H^1(\F)=0$.
Here $C'$ is a cubic curve, $C$ is a conic curve in $\P^2$.

The sheaves $\F$ in $X_1$ are either of the form $\J_Z(2)$, where $Z \subset \P^2$
is the intersection of two conic curves without common component,
$Z$ is contained in a quintic curve $C$ and $\J_Z \subset \O_C$
is its ideal sheaf, or they are extension sheaves of the form
\[
0 \lra \O_L(-1) \lra \F \lra \J_x(1) \lra 0,
\]
satisfying the relation $h^0(\F \tensor \Om^1(1))=0$. Here $L \subset \P^2$ is a line
and $\J_x \subset \O_{C'}$ is the ideal sheaf of a closed point $x$ on a quartic curve $C' \subset \P^2$.

The generic sheaves from $X_2$ are of the form $\O_C(1)(-P_1+P_2+P_3)$, where $C \subset \P^2$
is a smooth quintic curve and $P_1, P_2, P_3$ are distinct points on $C$.

The sheaves giving points in $X_3$ are precisely the non-split extension sheaves of the form
\[
0 \lra \O_C(1) \lra \F \lra \C_x \lra 0.
\]
Here $C \subset \P^2$ is a quintic curve and $\C_x$ is the structure sheaf of a closed point.


\subsection{The moduli space $\M(5,0)$}

This moduli space can be decomposed into four strata: an open stratum $X_0$,
two locally closed strata $X_1, X_2$ and a closed stratum $X_3$.
$X_0$ is a proper open subset inside $\N(3,5,5)$.
$X_2$ is a proper open subset inside a fibre bundle over $\P^2 \times \P^2$ with fibre $\P^{18}$.
$X_3$ consists of sheaves of the form $\O_C(1)$, where $C \subset \P^2$ is a quintic curve,
and is isomorphic to $\P(S^5V^*)$.
All strata are invariant under the duality isomorphism.

The information about the cohomological conditions defining each stratum
and resolutions for semi-stable sheaves can be found in table 3 below,
which is organised as table 1.
$X_0$ is a good quotient, $X_1$ is a categorical quotient and $X_2$ is a geometric quotient
of $W$ by $G$.

\begin{table}[ht]{{Table 3. Summary for $\M$(5,0).}}
\begin{center}
\begin{tabular}{|c|c|c|c|}
\hline
\!\!\! {\Tiny stratum} \!\!\!
&
{\Tiny cohomological conditions}
&
{\Tiny subset $W \subset \W$ of morphism $\f$}
&
\!\!\! {\Tiny codim.} \!\!\! \\
\hline
$X_0$
&
\begin{tabular}{r}
$h^0(\F(-1))=0$ \\
$h^1(\F)=0$\\
$h^0(\F \tensor \Om^1(1))=0$
\end{tabular}
&
\begin{tabular}{c}
$5\O(-2) \stackrel{\f}{\lra} 5\O(-1)$ \\
$\f$ is injective
\end{tabular}
&
$0$ \\
\hline
$X_1$
&
\begin{tabular}{r}
$h^0(\F(-1))=0$ \\
$h^1(\F)=1$\\
$h^0(\F \tensor \Om^1(1))=0$
\end{tabular}
&
\begin{tabular}{c}
$\O(-3) \oplus 2\O(-2) \stackrel{\vphantom{I} \f}{\lra} 2\O(-1) \oplus \O$ \\
$\f$ is injective \\
$\f_{12}$ is injective
\end{tabular}
&
$1$ \\
\hline
$X_2$
&
\begin{tabular}{r}
$h^0(\F(-1))=0$ \\
$h^1(\F)=2$\\
$h^0(\F \tensor \Om^1(1))=1$
\end{tabular}
&
\begin{tabular}{c}
$2\O(-3) \oplus \O(-1) \stackrel{\vphantom{I} \f}{\lra} \O(-2) \oplus 2\O$ \\
$\f$ is injective \\
$\f_{11}$ has linearly independent entries \\
$\f_{22}$ has linearly independent entries
\end{tabular}
&
$4$ \\
\hline
$X_{3}$
&
\begin{tabular}{r}
$h^0(\F(-1))=1$ \\
$h^1(\F)=3$\\
$h^0(\F \tensor \Om^1(1))=3$
\end{tabular}
&
\begin{tabular}{c}
$\O(-4) \stackrel{\f}{\lra} \O(1)$ \\
$\f \neq 0$
\end{tabular}
&
$6$ \\
\hline 
\end{tabular}
\end{center}
\end{table}

The generic sheaves in $X_0$ are of the form $\J_Z(3)$, where $Z \subset \P^2$
is a zero-dimensional scheme of length $10$ not contained in a cubic curve, contained
in a quintic curve $C$, and $\J_Z \subset \O_C$ is its ideal sheaf.

The sheaves giving points in $X_2$ are precisely the non-split extension sheaves of the form
\[
0 \lra \J_x(1) \lra \F \lra \C_z \lra 0,
\]
where $\J_x \subset \O_C$ is the ideal sheaf of a closed point $x$ on a quintic curve $C \subset \P^2$
and $\C_z$ is the structure sheaf of a closed point $z \in C$.
When $x=z$ we exclude the possibility $\F \isom \O_C(1)$.

\section*{Acknowledgements} 

The author was supported
by the BitDefender Scholarship at the Institute of Mathematics of the Romanian Academy
and by the Consiliul Na\c tional al Cercet\u arii \c Stiin\c tifice,
grant PN II--RU 169/2010 PD--219.
The author is grateful to J.-M. Dr\'ezet for many insightful remarks and helpful suggestions.
Important changes were made after the referee report.
The referee suggested improvements to the proofs of \ref{2.1.4} and \ref{2.1.6},
noticed gaps in our original construction of quotients as fibre bundles,
and pointed out the argument used to show that the generic sheaves at \ref{2.3.2} and \ref{3.3.4}
are certain line bundles supported on smooth quintics



\section{Euler characteristic two or three}


\subsection{Locally free resolutions for semi-stable sheaves}

\begin{prop}
\label{2.1.1}
There are no sheaves $\F$ giving points in $\M(5,3)$ and satisfying the conditions
$h^0(\F(-1))=0$ and $h^1(\F) \neq 0$.
\end{prop}

\begin{proof}
According to 6.4 \cite{maican}, there are no sheaves $\G$ in $\M(5,2)$
satisfying the conditions $h^0(\G(-1))\neq 0$ and $h^1(\G)=0$. The result follows by duality.
\end{proof}

\noindent
From this and from 4.3 \cite{maican} we obtain:

\begin{prop}
\label{2.1.2}
Let $\F$ be a sheaf in $\M(5,3)$ satisfying the condition $h^0(\F(-1))=0$. Then $h^1(\F)=0$
and $h^0(\F \tensor \Om^1(1))=1$ or $2$. The sheaves from the first case are precisely the sheaves
that have a resolution of the form
\begin{align*}
\tag{i}
0 \lra 2\O(-2) \oplus \O(-1) \stackrel{\f}{\lra} 3\O \lra \F \lra 0
\end{align*}
with $\f$ not equivalent, modulo the action of the natural group of automorphisms,
to a morphism represented by a matrix of the form
\[
\left[
\begin{array}{ccc}
\star & \star & \star \\
\star & \star & 0 \\
\star & \star & 0
\end{array}
\right] \qquad \text{or} \qquad \left[
\begin{array}{ccc}
\star & \star & \star \\
\star & \star & \star \\
\star & 0 & 0
\end{array}
\right].
\]
The sheaves in the second case are precisely the sheaves that have a resolution of the form
\begin{align*}
\tag{ii}
0 \lra 2\O(-2) \oplus 2\O(-1) \stackrel{\f}{\lra} \O(-1) \oplus 3\O \lra \F \lra 0,
\end{align*}
\[
\f = \left[
\begin{array}{cc}
\f_{11} & 0 \\
\f_{21} & \f_{22}
\end{array}
\right], \qquad \f_{11}= \left[
\begin{array}{cc}
\ell_1 & \ell_2
\end{array}
\right],
\]
where $\ell_1, \ell_2$ are linearly independent one-forms and the
maximal minors of $\f_{22}$ are linearly independent two-forms.
\end{prop}

\begin{prop}
\label{2.1.3}
Let $\F$ be a sheaf giving a point in $\M(5,3)$ and satisfying the conditions
$h^1(\F)=0$ and $h^0(\F(-1))\neq 0$.
Then $h^0(\F(-1)) =1$. These sheaves are precisely the sheaves with resolution of the form
\[
0 \lra 3\O(-2) \stackrel{\f}{\lra} 2\O(-1) \oplus \O(1) \lra \F \lra 0,
\]
where $\f_{11}$ has linearly independent maximal minors.
\end{prop}

\begin{proof}
The first conclusion follows from 6.6 \cite{maican}.
According to 5.3 \cite{maican}, every sheaf $\G$ in $\M(5,2)$ satisfying
$h^0(\G(-1))=0$ and $h^1(\G)=1$ has a resolution
\[
0 \lra \O(-3) \oplus 2\O(-1) \stackrel{\psi}{\lra} 3\O \lra \G \lra 0
\]
in which $\psi_{12}$ has linearly independent maximal minors.
The second conclusion follows by duality.
\end{proof}

\begin{prop}
\label{2.1.4}
Let $\F$ be a sheaf giving a point in $\M(5,3)$
and satisfying the conditions $h^0(\F(-1))=1$ and $h^1(\F)=1$.
Then $h^0(\F \tensor \Om^1(1))=4$ and $\F$ has a resolution of the form
\[
0 \lra \O(-3) \oplus \O(-1) \stackrel{\f}{\lra} \O \oplus \O(1) \lra \F \lra 0
\]
with $\f_{12} \neq 0$ and $\f_{22}$ not divisible by $\f_{12}$. Conversely,
every $\F$ having such a resolution is semi-stable.
\end{prop}

\begin{proof}
Let $\F$ give a point in $\M(5,3)$ and satisfy the cohomological conditions from the proposition.
Write $m=h^0(\F \tensor \Om^1(1))$.
The Beilinson free monad (2.2.1) \cite{drezet-maican} for $\F$ reads
\[
0 \lra \O(-2) \lra 3\O(-2) \oplus m\O(-1) \lra (m-1)\O(-1) \oplus 4\O \lra \O \lra 0
\]
and gives the resolution
\[
0 \lra \O(-2) \lra 3\O(-2) \oplus m\O(-1) \lra \Om^1 \oplus (m-4)\O(-1) \oplus 4\O \lra \F
\lra 0.
\]
We see from the above that $m \ge 4$. Combining with the Euler sequence
we obtain the resolution
\begin{multline*}
0 \lra \O(-2) \stackrel{\psi}{\lra} \O(-3) \oplus 3\O(-2) \oplus m\O(-1) \stackrel{\f}{\lra} \\
3\O(-2) \oplus (m-4)\O(-1) \oplus 4\O \lra \F \lra 0,
\end{multline*}
\[
\psi = \left[
\begin{array}{c}
0 \\ 0 \\ \psi_{31}
\end{array}
\right], \qquad \f= \left[
\begin{array}{ccc}
\eta & \f_{12} & 0 \\
0 & \f_{22} & 0 \\
0 & \f_{32} & \f_{33}
\end{array}
\right]. \qquad \text{Here} \qquad \eta = \left[
\ba{c}
X \\ Y \\ Z
\ea
\right].
\]
We have a commutative diagram in which the vertical maps are projections onto direct summands:
\[
\xymatrix
{
\O(-3) \oplus 3\O(-2) \oplus m\O(-1) \ar[d] \ar[r]^-{\f} & 3\O(-2) \oplus (m-4)\O(-1) \oplus 4\O \ar[d] \\
\O(-3) \oplus 3\O(-2) \ar[r]^-{\a} & 3\O(-2)
},
\]
\[
\a= \left[ \ba{cc} \eta & \f_{12} \ea \right].
\]
Thus $\F$ maps surjectively to $\Coker(\a)$. If $\rank(\f_{12})=0$,
then $\Coker(\a) \isom \Om^1$. 
If $\rank(\f_{12})=1$, then $\Coker(\a) \isom \I_x(-1)$,
where $\I_x \subset \O$ is the ideal sheaf of a point $x \in \P^2$.
These two cases are unfeasible because $\F$ has support of dimension $1$ so it cannot
map surjectively onto a sheaf supported on the entire plane.
If $\rank(\f_{12})=2$, then $\Coker(\a)$ would be isomorphic to $\O_{L}(-2)$ for a line $L \subset \P^2$,
so it would destabilise $\F$. We conclude that $\rank(\f_{12})=3$.
We may cancel $3\O(-2)$ to get the resolution
\[
0 \lra \O(-2) \stackrel{\psi}{\lra} \O(-3) \oplus m\O(-1) \stackrel{\f}{\lra}
(m-4) \O(-1) \oplus 4\O \lra \F \lra 0,
\]
\[
\psi = \left[
\begin{array}{c}
0 \\ \psi_{21}
\end{array}
\right], \qquad \qquad \f= \left[
\begin{array}{cc}
\f_{11} & 0 \\
\f_{21} & \f_{22}
\end{array}
\right].
\]
Note that $\F$ maps surjectively onto ${\mathcal Coker}(\f_{11})$, so the latter has
rank zero, forcing $m \le 5$. If $m=5$, then $\Coker(\f_{11})$ would be isomorphic to $\O_C(-1)$
for a conic curve $C \subset \P^2$, so it would destabilise $\F$.
We deduce that $m=4$ and we get the resolution
\[
0 \lra \O(-2) \stackrel{\psi}{\lra} \O(-3) \oplus 4\O(-1) \stackrel{\f}{\lra} 4\O
\lra \F \lra 0.
\]
Let $\bar{\psi} \colon V \to \C^4$ be the linear map induced by $\psi_{21}$.
Let $H$ be the image of $\bar{\psi}$ and let $K \subset \C^4$ be a linear subspace
such that $H \oplus K =\C^4$.
We have an exact sequence
\[
0 \lra \O(-2) \stackrel{\psi}{\lra} \O(-3) \oplus (K \tensor \O(-1)) \oplus (H \tensor \O(-1)) \stackrel{\f}{\lra} 4\O
\lra \F \lra 0,
\]
in which $\psi_{11}=0$, $\psi_{21}=0$. If $\dim(H)=1$, then $\psi_{31}$ is generically surjective.
As $\f$ vanishes on $\Im(\psi_{31})$, it must vanish on $H \tensor \O(-1)$,
hence $H \tensor \O(-1)$ is a subsheaf of $\O(-2)$.
This is absurd. If $\dim(H)=2$, then $\Coker(\psi_{31})$ is isomorphic to the ideal sheaf $\I_x$
of a point $x \in \P^2$.
We get a resolution
\[
0 \lra \O(-3) \oplus 2 \O(-1) \oplus \I_x \lra 4\O \lra \F \lra 0.
\]
The image of $\I_x$ is included into a factor $\O$ of $4\O$ because $\Hom(\I_x, \O) \isom \C$.
We obtain a commutative diagram
\[
\xymatrix
{
0 \ar[r] & \I_x \ar[r] \ar[d] & \O \ar[r] \ar[d] & \C_x \ar[r] \ar[d] & 0 \\
0 \ar[r] & \O(-3) \oplus 2 \O(-1) \oplus \I_x \ar[r] & 4\O \ar[r] & \F \ar[r] & 0
}
\]
in which the first two vertical maps are injective.
The induced map $\C_x \to \F$ is zero because $\F$ has no zero-dimensional torsion.
It follows that $\O$ is a subsheaf of $\O(-3) \oplus 2 \O(-1) \oplus \I_x$,
which is absurd. We deduce that $H$ has dimension $3$, so $\Coker(\psi_{31})\isom \Om^1(1)$
and we get the resolution
\[
0 \lra \O(-3) \oplus \O(-1) \oplus \Om^1(1) \stackrel{\f}{\lra} 4\O \lra \F \lra 0.
\]
Consider the canonical morphism $i \colon \Om^1(1) \to \Hom(\Om^1(1),\O)^*\tensor \O \isom 3\O$.
There is a morphism $\b \colon 3\O \to 4\O$ such that $\b \circ i = \f_{13}$.
If $\b$ were not injective, then $\f$ would be equivalent to a morphism represented by a matrix of the form
\[
\left[
\begin{array}{cc}
\g_{11} & 0 \\
\g_{21} & \g_{22}
\end{array}
\right],
\]
where $\g_{11} \in \Hom(\O(-3) \oplus \O(-1), 2\O)$.
But then $\Coker(\g_{11})$ would be a destabilising quotient sheaf of $\F$.
Thus $\b$ is injective, from which we deduce that $\Coker(\f_{13}) \isom \O \oplus \O(1)$.
We obtain the resolution
\[
0 \lra \O(-3) \oplus \O(-1) \stackrel{\f}{\lra} \O \oplus \O(1) \lra \F \lra 0.
\]
If $\f_{12}=0$, then $\F$ would have a destabilising subsheaf of the form $\O_C(1)$,
for a conic curve $C \subset \P^2$.
If $\f_{12}$ divided $\f_{22}$, then $\F$ would have a destabilising subsheaf of the form
$\O_L$ for a line $L \subset \P^2$.

Conversely, assume that $\F$ has a resolution as in the proposition.
Then $\F$ has no zero-dimensional
torsion because it has projective dimension $1$ at every point in its support.
Thus, it is enough to show that $\F$ cannot have a destabilising subsheaf. 
Let $\F' \subset \F$ be a non-zero subsheaf of multiplicity at most 4.
According to \ref{2.3.5}, $\F$ is isomorphic to $\J_Z(2)$, where $\J_Z \subset \O_C$
is the ideal sheaf of a zero-dimensional scheme $Z$ of length $2$ inside a quintic
curve $C$. According to \cite{maican}, lemma 6.7, there is a sheaf $\A \subset \O_C(2)$
containing $\F'$ such that $\A/\F'$ is supported on finitely many points and $\O_C(2)/\A \isom
\O_S(2)$ for a curve $S \subset C$ of degree $d \le 4$.
The slope of $\F'$ can be estimated as follows:
\begin{align*}
P_{\F'}(t) & = P_{\A}(t) - h^0(\A/\F') \\
& = P_{\O_C}(t+2) - P_{\O_S}(t+2) - h^0(\A/\F') \\
& = (5-d)t + \frac{(d-5)(d-2)}{2} - h^0(\A/\F'), \\
p(\F') & = \frac{2-d}{2} - \frac{h^0(\A/\F')}{5-d} \le \frac{1}{2} < \frac{3}{5} = p(\F).
\end{align*}
We conclude that $\F$ is semi-stable.
\end{proof}

\begin{prop}
\label{2.1.5}
Any sheaf $\G$ giving a point in $\M(5,2)$ satisfies the condition $h^0(\G(-1)) \le 1$.
\end{prop}

\begin{proof}
Let $\G$ be in $\M(5,2)$ and assume that $h^0(\G(-1)) >0$.
As in the proof of 2.1.3 \cite{drezet-maican}, there is an injective morphism $\O_C \to \G(-1)$
for a curve $C \subset \P^2$.
From the semi-stability of $\G(-1)$ we see that $C$ must be a quintic curve.
The quotient sheaf $\G(-1)/\O_C$ is a sheaf of dimension zero and length $2$;
it maps surjectively onto the structure sheaf $\C_x$ of a point $x$.
Let $\G'$ be the kernel of the composed morphism $\G \to \C_x$.
If $\G'$ is semi-stable, then, from \ref{3.1.5}, we have $h^0(\G'(-1)) \le 1$.
It follows that $h^0(\G(-1)) \le 1$ unless $h^0(\G'(-1)) = 1$ and the morphism $\G(-1) \to \C_x$
is surjective on global sections.
In this case we can apply the horseshoe lemma to the extension
\[
0 \lra \G'(-1) \lra \G(-1) \lra \C_x \lra 0,
\]
to the standard resolution of $\C_x$ and to resolution \ref{3.1.5} for $\G'$ tensored with $\O(-1)$,
which reads:
\[
0 \lra 2\O(-4) \lra \O(-3) \oplus \O \lra \G'(-1) \lra 0.
\]
We get a resolution of the form
\[
0 \lra 2\O(-4) \oplus \I_x \lra \O(-3) \oplus 2\O \lra \G(-1) \lra 0.
\]
We now arrive at a contradiction as in the proof of \ref{2.1.4}.
The image of $\I_x$ is included in a factor $\O$ of $2\O$.
As $\G(-1)$ has no zero-dimensional torsion, this factor $\O$
maps to zero in $\G(-1)$, which is absurd.

Assume now that $\G'$ is not semi-stable and let $\G'' \subset \G'$ be a destabilising subsheaf.
We may assume that $\G''$ itself is semi-stable, say it gives a point in $\M(r,\chi)$.
We have the inequalities
\[
\frac{1}{5} = p(\G') < \frac{\chi}{r} < p(\G) = \frac{2}{5}
\]
leaving only the possibilities $(r,\chi) = (4,1)$ or $(3,1)$.
Denote $\CC=\G/\G''$. If $\G''$ is in $\M(4,1)$, then $P_{\CC}(t)=t+1$.
Moreover, the zero-dimensional torsion of $\CC$ vanishes, otherwise its pull-back in $\G$
would be a destabilising subsheaf. We deduce that $\CC=\O_L$ for a line $L \subset \P^2$.
But $h^0(\O_L(-1))=0$ and, according to 2.1.3 \cite{drezet-maican}, also $h^0(\G''(-1))=0$.
We get $h^0(\G(-1))=0$, contradicting our hypothesis on $\G$.

The last case to examine is when $\G''$ is in $\M(3,1)$.
We have $P_{\CC}(t)=2t+1$. As before, $\CC$ has no zero-dimensional torsion.
Moreover, any quotient sheaf destabilising $\CC$ must also destabilise $\G$.
We conclude that $\CC$ is semi-stable, i.e. $\CC = \O_C$ for a conic curve $C \subset \P^2$.
But $h^0(\O_C(-1))=0$ and, according to 2.1.3 \cite{drezet-maican}, also $h^0(\G''(-1))=0$.
We conclude that $h^0(\G(-1))=0$, contrary to our hypothesis on $\G$.
\end{proof}

\begin{prop}
\label{2.1.6}
There are no sheaves $\G$ giving points in $\M(5,2)$ and satisfying the conditions
$h^0(\G(-1))=1$ and $h^1(\G)\ge 2$.
\end{prop}

\begin{proof}
Fix an integer $m \ge 0$ and let $X$ be the set of sheaves $\G$ in $\M(5,2)$
satisfying $h^0(\G(-1))=1$ and $h^0(\G \tensor \Om^1)=m$. Let $Y \subset X$ be the
subset of sheaves satisfying the additional condition $h^1(\G)=1$.
According to 2.1.3 \cite{drezet-maican}, for every sheaf in $X$ we have $\H^0(\G(-2))=0$.
The Beilinson free monad (2.2.1) \cite{drezet-maican} for $\G(-1)$ reads
\[
0 \lra 8\O(-2) \oplus m\O(-1) \lra (m+11)\O(-1) \oplus \O \lra 4\O \lra 0.
\]
Thus $X$ is parametrised by an open subset $M$ inside the space of monads of the form
\[
0 \lra 8\O(-1) \oplus m\O \stackrel{A}{\lra} (m+11)\O \oplus \O(1) \stackrel{B}{\lra} 4\O(1) \lra 0,
\]
where $A_{12}=0$, $B_{12}=0$. Let $\Gamma$ be the space of pairs $(A,B)$ of morphisms
\begin{align*}
A \in & \Hom(8\O(-1) \oplus m\O, (m+11)\O \oplus \O(1)), \\
B \in & \Hom((m+11)\O \oplus \O(1), 4\O(1)),
\end{align*}
such that $A$ is injective, $B$ is surjective, $A_{12}=0$, $B_{12}=0$.
Consider the algebraic map $\g \colon \Gamma \to \Hom(8\O(-1),4\O(1))$ given by
$\g(A,B)=B_{11} \circ A_{11}$.
Note that $M$ is an open subset inside $\g^{-1} (0)$.
We claim that $M$ is smooth. For this it is sufficient to show that $\g$ has surjective
differential at every point of $M$.
The tangent space of $\Gamma$ at an arbitrary point $(A,B)$ is the space of pairs $(\a, \b)$ of morphisms
\begin{align*}
\a \in & \Hom(8\O(-1) \oplus m\O, (m+11)\O \oplus \O(1)), \\
\b \in & \Hom((m+11)\O \oplus \O(1), 4\O(1)),
\end{align*}
such that $\a_{12}=0$, $\b_{12}=0$.
We have $\text{d}\g_{(A,B)}(\a,\b) = B_{11} \circ \a_{11} + \b_{11} \circ A_{11}$.
It is enough to prove that the map
$\a_{11} \to B_{11} \circ \a_{11}$ is surjective at a point $(A,B) \in M$.
For this we apply the long $\Ext(8\O(-1),\_\_)$-sequence to the exact sequence
\[
0 \lra \Ker(B_{11}) \lra (m+11)\O \stackrel{B_{11}\,\,}{\lra} 4\O(1) \lra 0
\]
and we use the vanishing of $\Ext^1(8\O(-1),\Ker(B_{11}))$. This vanishing follows
from the exact sequence
\[
0 \lra 8\O(-1) \oplus m\O \lra \Ker(B_{11}) \oplus \O(1) \lra \G \lra 0
\]
and the vanishing of $\H^1(\G(1))$, which is a consequence of 2.1.3 \cite{drezet-maican}.

Let $\upsilon \colon M \to X$ be the surjective morphism
which sends a monad to the isomorphism class of its cohomology.
The tangent space to $M$ at an arbitrary point $(A,B)$ is
\[
{\mathbb T}_{(A,B)}M= \{ (\a,\b) \mid \a_{12}=0, \b_{12}=0, \b \circ A + B \circ \a =0 \}.
\]
Consider the map $\Phi \colon M \to \Hom((m+11)\O,4\O(1))$, $\Phi(A,B)=B_{11}$.
It has surjective differential at every point. Indeed, $\text{d}\Phi_{(A,B)}(\a,\b) = \b_{11}$,
so we need to show that, given $\b_{11}$, there is $\a$ such that
$\b \circ A + B \circ \a =0$, that is $-\b_{11} \circ A_{11} = B_{11} \circ \a_{11}$.
This follows from the surjectivity of the map $\a_{11} \to B_{11} \circ \a_{11}$,
which we proved above.

We have $h^0(\G)=14 -\rank(\H^0(B_{11}))$.
The subset $N \subset M$ of monads with cohomology $\G$ satisfying $h^1(\G) \ge 2$
is the preimage under $\Phi$ of the set of morphisms of rank at most $10$.
Since any matrix of rank at most $10$ is the limit of a sequence of matrices of rank $11$,
and since the derivative of $\Phi$ is surjective at every point, we deduce that
$N$ is included in $\overline{\upsilon^{-1}(Y)} \setminus \upsilon^{-1}(Y)$.
But, according to \ref{2.1.4}, $Y$ is empty for $m \neq 0$. For $m=0$, we shall
prove at \ref{2.2.6} below that $Y$ is closed. We conclude that $N$ is empty.
\end{proof}


\subsection{Description of the strata as quotients}

In subsection 2.1 we found that the moduli space $\M(5,3)$ can be decomposed into
four strata:
\begin{enumerate}
\item[$-$] an open stratum $X_0$ given by the conditions \\
$h^0(\F(-1))=0$, $h^0(\F \tensor \Om^1(1))=1$;
\item[$-$] a locally closed stratum $X_1$ of codimension $2$ given by the conditions \\
$h^0(\F(-1))=0$, $h^0(\F \tensor \Om^1(1))=2$;
\item[$-$] a locally closed stratum $X_2$ of codimension $3$ given by the conditions \\
$h^0(\F(-1))=1$, $h^1(\F)=0$;
\item[$-$] the stratum $X_3$ of codimension $4$ given by the conditions \\
$h^0(\F(-1))=1$, $h^1(\F)=1$. We shall see below at \ref{2.2.6} that $X_3$ is closed.
\end{enumerate}
In the sequel $X_i$ will be equipped with the canonical induced reduced structure.
Let $W_0$, $W_1$, $W_2$, $W_3$ be the sets of morphisms $\f$
from \ref{2.1.2}(i), \ref{2.1.2}(ii), \ref{2.1.3}, respectively \ref{2.1.4}.
Each sheaf $\F$ giving a point in $X_i$ is the cokernel of a
morphism $\f \in W_i$. Let $\W_i = \Hom(\A_i,\B_i)$ denote the
ambient vector space containing $W_i$. Here $\A_i$, $\B_i$ are locally
free sheaves on $\P^2$, for instance $\A_0 = 2\O(-2) \oplus \O(-1)$, $\B_0 = 3\O$.
The natural group of automorphisms
$G_i = (\Aut(\A_i) \times \Aut(\B_i))/\C^*$ acts on $\W_i$ by conjugation, leaving
$W_i$ invariant (here $\C^*$ is embedded as the subgroup of homotheties).
In this subsection we shall prove that there exist geometric quotients $W_i/G_i$, which
are smooth quasiprojective varieties ($W_3/G_3$ is even projective), such that
$W_i/G_i \isom X_i$. Whenever possible, we shall give concrete descriptions of these
quotients.

\begin{prop}
\label{2.2.1}
There exists a geometric quotient $W_0/G_0$, which is a smooth quasiprojective
variety. $W_0/G_0$ is isomorphic to $X_0$.
\end{prop}

\begin{proof}
Let $\L=(\l_1,\l_2,\m_1)$ be a polarisation for the action of $G_0$ on $\W_0$
satisfying $1/6 < \l_1 < 1/3$ (see \cite{drezet-trautmann} for the notions of polarisation
and of semi-stable morphism). According to 4.3 \cite{maican}, $W_0$ is the open
invariant subset of injective morphisms inside the set
$\W_0^{ss}(\L)$ of semi-stable morphisms with respect to $\L$.
According to 6.4 \cite{drezet-2000}, if $\l_1<1/5$,
then there is a geometric quotient $\W_0^{ss}(\L)/G_0$,
which is a projective variety (see also 7.11 \cite{maican}).
We fix $\L$ satisfying $1/6 < \l_1 < 1/5$.
It is now clear that a geometric quotient $W_0/G_0$
exists and is an open subset of $\W_0^{ss}(\L)/G_0$.

The morphism $W_0 \to X_0$ sending $\f$ to the isomorphism class of $\Coker(\f)$
is surjective and its fibres are $G_0$-orbits, hence it factors through a bijective morphism
$W_0/G_0 \to X_0$. Since $X_0$ is smooth, Zariski's Main Theorem tells us that the
latter is an isomorphism.
\end{proof}

We remark that $W_0$ is a proper subset of $\W^{ss}_0(\L)$,  hence $W_0/G_0$
is a proper open subset of the projective variety $\W_0^{ss}(\L)/G_0$.
Indeed, the morphism $\f_0$ represented by the matrix
\[
\left[
\ba{ccc}
XY & \phantom{-}X^2 & 0 \\
XZ & \phantom{-}0 & X \\
0 & -XZ & Y
\ea
\right]
\]
is not injective but is semi-stable with respect to $\L$.
This follows from King's criterion of semi-stability \cite{king}, which, in our case, says
that a morphism is in $\W_0^{ss}(\L)$ if and only if it is not equivalent to a morphism
having one of the following forms:
\[
\left[
\ba{ccc}
\star & \star & 0 \\
\star & \star & 0 \\
\star & \star & \star
\ea
\right], \qquad \left[
\ba{ccc}
\star & 0 & 0 \\
\star & \star & \star \\
\star & \star & \star
\ea
\right], \qquad \left[
\ba{ccc}
0 & \star & \star \\
0 & \star & \star \\
0 & \star & \star
\ea
\right], \qquad \left[
\ba{ccc}
0 & 0 & \star \\
0 & 0 & \star \\
\star & \star & \star
\ea
\right].
\]
The first case is excluded by the fact that $\f_0$ has two linearly independent
entries on column 3, the second case is excluded by the fact that $\f_0$
has two linearly independent entries on row 1. To exclude the third case assume that
\[
\left[
\ba{ccc}
XY & \phantom{-}X^2 & 0 \\
XZ & \phantom{-}0 & X \\
0 & -XZ & Y
\ea
\right] \left[
\ba{c}
c_1 \\ c_2 \\ \ell
\ea
\right]= \left[
\ba{c}
0 \\ 0 \\ 0
\ea
\right],
\]
for $c_1, c_2 \in \C$ and $\ell \in V^*$.
Then the triple $(c_1 X, c_2 X, \ell)$ is a multiple of $(-X, Y, Z)$, which is absurd.
The last case can also be easily excluded.

We recall from 2.4 \cite{drezet-maican} the moduli spaces $\N(3,m,n)$
of semi-stable Kronecker modules $f \colon \C^m \tensor V \to \C^n$.

\begin{prop}
\label{2.2.2}
There exists a geometric quotient $W_1/G_1$ and it is a proper open subset
inside a fibre bundle over $\P^2 \times \N(3,2,3)$ with fibre $\P^{16}$.
\end{prop}

\begin{proof}
Let $W_1'$ be the locally closed subset of $\W_1$ given by the conditions that
$\f_{12}=0$, $\f_{11}$ have linearly independent entries and  $\f_{22}$ have linearly independent
maximal minors. The set of morphisms $\f_{11}$ form an open subset
$U_1 \subset \Hom(2\O(-2),\O(-1))$ and the set of morphisms $\f_{22}$ form an open subset
$U_2 \subset \Hom(2\O(-1),3\O)$. We denote $U = U_1 \times U_2$.
$W_1'$ is the trivial vector bundle over $U$ with fibre $\Hom(2\O(-2),3\O)$.
We represent the elements of $G_1$ by pairs of matrices
\[
(g,h) \in \Aut(2\O(-2) \oplus 2\O(-1)) \times \Aut(\O(-1) \oplus 3\O),
\]
\[
g = \left[
\ba{cc}
g_1 & 0 \\
u & g_2
\ea
\right], \quad \qquad h = \left[
\ba{cc}
h_1 & 0 \\
v & h_2
\ea
\right].
\]
Inside $G_1$ we distinguish three subgroups: a unitary subgroup $G_1'$ given by the conditions
that $g_1$, $g_2$, $h_1$, $h_2$ be the identity morphisms,
a reductive subgroup ${G_1}_{\text{red}}$ given by the conditions $u=0$, $v=0$ and
a subgroup $S$ of ${G_1}_{\text{red}}$ isomorphic to $\C^*$ given by the conditions
that $g_1$, $h_1$ be the morphisms of multiplication by a non-zero constant $a$
and that $g_2$, $h_2$ be the morphisms of multiplication by a non-zero constant $b$.
Note that $G_1 = G_1' {G_1}_{\text{red}}$.
Consider the $G_1$-invariant subset $\Sigma \subset W_1'$ given by the condition
\[
\f_{21}=\f_{22} u + v \f_{11}, \quad u \in \Hom(2\O(-2),2\O(-1)), \quad
v \in \Hom(\O(-1),3\O).
\]
Note that $W_1$ is the subset of injective morphisms inside $W_1' \setminus \Sigma$,
so it is open and $G_1$-invariant. Moreover, it is a proper subset as, for instance, the morphism
represented by the matrix
\[
\left[
\ba{cccc}
Y & X\phantom{^2} & \phantom{-} 0 & 0 \\
0 & Y^2 & \phantom{-} X & 0 \\
0 & Y Z & \phantom{-} 0 & X \\
0 & 0\phantom{^2} & -Z & Y
\ea
\right]
\]
is in $W_1' \setminus \Sigma$ but is not injective.
Our aim is to construct a geometric quotient of $W_1'\setminus \Sigma$ modulo $G_1$;
it will follow that $W_1/G_1$ exists and is a proper open subset of $(W_1'\setminus \Sigma)/G_1$.

Firstly, we construct the geometric quotient $W_1'/G_1'$. Because of the conditions on
$\f_{11}$ and $\f_{22}$ it is easy to check that $\Sigma$ is a subbundle of $W_1'$.
The quotient bundle, denoted $Q'$, has rank $17$.
The quotient map $W_1' \to Q'$ is a geometric quotient modulo $G_1'$.
Moreover, the canonical action of ${G_1}_{\text{red}}$ on $U$ is $Q'$-linearised
and the map $W_1' \to Q'$ is ${G_1}_{\text{red}}$-equivariant.
Let $\sigma$ be the zero-section of $Q'$.
The restricted map $W_1' \setminus \Sigma \to Q' \setminus \sigma$
is also a geometric quotient map modulo $G_1'$.

Let $x \in U$ be a point and let $\xi \in Q_x'$ be a non-zero vector lying over $x$.
The stabiliser of $x$ in ${G_1}_{\text{red}}$ is $S$ and $S \xi = \C^* \xi$.
Thus the canonical map $Q' \setminus \sigma \to \P(Q')$ is a geometric quotient modulo $S$.
It remains to construct a geometric quotient of $\P(Q')$ modulo the induced action of
${G_1}_{\text{red}}/S$.

The existence of a geometric quotient of $U$ modulo ${G_1}_{\text{red}}/S$ follows from the classical
geometric invariant theory. We notice that
\begin{multline*}
{G_1}_{\text{red}}/S \isom ((\Aut(2\O(-2)) \times \Aut(\O(-1)))/\C^*) \times \\
((\Aut(2\O(-1)) \times \Aut(3\O))/\C^*).
\end{multline*}
Using King's criterion of semi-stability \cite{king},
we can see that $U_1$ is the set of semi-stable points
for the canonical action by conjugation of
\[
(\Aut(2\O(-2)) \times \Aut(\O(-1)))/\C^* \quad \text{on} \quad \Hom(2\O(-2), \O(-1)).
\]
The resulting geometric quotient is $\N(3,2,1)$ and is clearly isomorphic to $\P^2$.
Analogously, $U_2$ is the set of semi-stable points for the action of
\[
(\Aut(2\O(-1)) \times \Aut(3\O))/\C^* \quad \text{on} \quad \Hom(2\O(-1), 3\O)
\]
and the resulting quotient is $\N(3,2,3)$. According to \cite{drezet-exceptionnels},
this is a smooth projective irreducible variety of dimension $6$.
We obtain:
\[
U/({G_1}_{\text{red}}/S) \isom \N(3,2,1) \times \N(3,2,3) \isom \P^2 \times \N(3,2,3).
\]
It remains to show that $\P(Q')$ descends to a fibre bundle over $U/({G_1}_{\text{red}}/S)$.
We consider the character $\chi$ of ${G_1}_{\text{red}}$ given by
$
\chi(g, h) = \det(g) \det(h)^{-1}.
$
Note that $\chi$ is well-defined because it is trivial on homotheties.
We multiply the action of ${G_1}_{\text{red}}$ on $Q'$ by $\chi$ and we denote the resulting
linearised bundle by $Q'_{\chi}$.
The action of $S$ on $Q'_{\chi}$ is trivial, hence $Q'_{\chi}$ is ${G_1}_{\text{red}}/S$-linearised.
The isotropy subgroup in ${G_1}_{\text{red}}/S$ for any point in $U$ is trivial,
so we can apply \cite{huybrechts}, lemma 4.2.15, to deduce that $Q'_{\chi}$
descends to a vector bundle $Q$ over $U/({G_1}_{\text{red}}/S)$.
The induced map $\P(Q') \to \P(Q)$ is a geometric quotient map modulo ${G_1}_{\text{red}}/S$.
We conclude that the composed map
\[
W_1' \setminus \Sigma \lra Q'\setminus \sigma \lra \P(Q') \lra \P(Q)
\]
is a geometric quotient map modulo $G_1$
and that a geometric quotient $W_1/G_1$ exists and is a proper open subset inside $\P(Q)$.
\end{proof}

\begin{prop}
\label{2.2.3}
The geometric quotient $W_1/G_1$ is isomorphic to $X_1$.
\end{prop}

\begin{proof}
As at \ref{2.2.1}, we have a canonical bijective morphism $W_1/G_1 \to X_1$.
To show that this is an isomorphism we shall use the method of 3.1.6 \cite{drezet-maican}.
Our aim is to construct resolution \ref{2.1.2}(ii) not merely for an individual sheaf
giving a point in $X_1$, but also for a flat family of sheaves giving points in $X_1$.
We achieve this for local flat families by obtaining resolution \ref{2.1.2}(ii)
in a natural manner from the relative Beilinson spectral sequence associated to the family.
Thus, for any sheaf $\F$ giving a point in $X_1$, we need to recover its resolution from 
its Beilinson spectral sequence. Diagram (2.2.3) \cite{drezet-maican} for $\F$ reads:
\[
\xymatrix
{
2\O(-2) \ar[r]^-{\f_1} & \O(-1) & 0 \\
0 & 2\O(-1) \ar[r]^-{\f_4} & 3\O
}.
\]
Since $\F$ is semi-stable and maps surjectively onto $\Coker(\f_1)$, we see that $\Coker(\f_1)$ is the
structure sheaf $\C_x$ of a point $x \in \P^2$ and that $\Ker(\f_1)$ is isomorphic to $\O(-3)$.
The exact sequence (2.2.5) \cite{drezet-maican}
\[
0 \lra \Ker(\f_1) \stackrel{\f_5}{\lra} \Coker(\f_4) \lra \F \lra \Coker(\f_1) \lra 0
\]
gives the extension
\[
0 \lra \Coker(\f_5) \lra \F \lra \Coker(\f_1) \lra 0.
\]
We apply the horseshoe lemma to the above extension and to the resolutions
\[
0 \lra \O(-3) \oplus 2\O(-1) \lra 3\O \lra \Coker(\f_5) \lra 0,
\]
\[
0 \lra \O(-3) \lra 2\O(-2) \lra \O(-1) \lra \Coker(\f_1) \lra 0.
\]
We arrive at the exact sequence
\[
0 \lra \O(-3) \lra \O(-3) \oplus 2\O(-2) \oplus 2\O(-1) \lra \O(-1) \oplus 3\O \lra \F \lra 0.
\]
Since $\H^1(\F)=0$, we see that $\O(-3)$ can be cancelled and we get resolution \ref{2.1.2}(ii), as desired.
\end{proof}

\begin{prop}
\label{2.2.4}
There exists a geometric quotient $W_2/G_2$, which is a proper open subset
inside a fibre bundle over $\N(3,3,2)$ with fibre $\P^{17}$.
Moreover, $W_2/G_2$ is isomophic to $X_2$.
\end{prop}

\begin{proof}
The existence of $W_2/G_2$ follows from the construction
of quotients given at 9.3 \cite{drezet-trautmann}.
Our situation is also analogous to 3.1.2 \cite{drezet-maican}.
We consider a polarisation $\L=(\l_1,\m_1,\m_2)$ as in \cite{drezet-trautmann}
for the action of $G_2$ on $\W_2$ satisfying the condition $0 < \m_2 < 1/3$.
According to op.cit., lemma 9.3.1,
the open subset $\W_2^{ss}(\L) \subset \W_2$ of semi-stable
morphisms with respect to $\L$ is the set of morphisms $\f$ for which of $\f_{11}$ is semi-stable
with respect to the action by conjugation
of $(\GL(3,\C) \times \GL(2,\C))/\C^*$ on $\Hom(3\O(-2),2\O(-1))$
and such that $\f$ is not equivalent to a morphism $\psi$ satisfying $\psi_{21}=0$.
According to King's criterion of semi-stability \cite{king}, the condition on $\f_{11}$
is the same as saying that $\f_{11}$ is not equivalent to a morphism represented
by a matrix having a zero-column or a zero-submatrix of size $1 \times 2$.
Furthermore, this is equivalent to the condition on $\f_{11}$ from \ref{2.1.3}.
We see now that $W_2$ is the open invariant subset of injective morphisms inside $\W_2^{ss}(\L)$.
It is a proper subset because it is easy to construct semi-stable morphisms that are not
injective, for example the morphism represented by the matrix
\[
\left[
\ba{ccc}
0\phantom{^3} & X\phantom{^2} & \phantom{-}Y \\
X\phantom{^3} & 0\phantom{^2} & -Z \\
Y^3 & ZY^2 & \phantom{-}0
\ea
\right].
\]
Adopting the notations of 3.1.2 \cite{drezet-maican}, let $\N(3,3,2)$ be the moduli space of semi-stable
Kronecker modules $f \colon 3\O(-2) \to 2\O(-1)$, let $\tau \colon E \tensor V \to F$
be the universal morphism on $\N(3,3,2)$, let $p_1$, $p_2$ be the projections of
$\N(3,3,2) \times \P^2$ onto its factors and let
\[
\theta \colon p_1^*(E) \tensor p_2^*(\O(-2)) \lra p_1^*(F) \tensor p_2^*(\O(-1))
\]
be the morphism induced by $\tau$. The sheaf $\U = {p_1}_*(\Coker(\theta^*) \tensor p_2^*\O(1))$
is locally free on $\N(3,3,2)$ of rank $18$. 
According to 9.3 \cite{drezet-trautmann}, $\P(\U)$ is a geometric
quotient of $\W_2^{ss}(\L)$ modulo $G_2$.
Thus $W_2/G_2$ exists and is a proper open subset of $\P(\U)$.

We shall now prove that the natural bijective morphism $W_2/G_2 \to X_2$ is an isomorphism.
Given $\F$ in $X_2$, we need to construct resolution \ref{2.1.3} starting from the Beilinson
spectral sequence of $\F$. It is easier to work, instead, with the dual sheaf $\G = \F^{\D}(1)$,
which gives a point in $\M(5,2)$. The Beilinson tableau (2.2.3) \cite{drezet-maican} for $\G$
takes the form
\[
\xymatrix
{
3\O(-2) \ar[r]^-{\f_1} & 3\O(-1) \ar[r]^-{\f_2} & \O \\
0 & 2\O(-1) \ar[r]^-{\f_4} & 3\O
}.
\]
According to 2.2 \cite{drezet-maican}, $\f_2$ is surjective while $\f_4$ is injective.
Thus $\Ker(\f_2) \isom \Om^1$. Consider the canonical morphism
\[
\rho \colon 3\O(-2) \isom \O(-2) \tensor \Hom(\O(-2), \Om^1) \lra \Om^1.
\]
There is a morphism $\a \colon 3\O(-2) \to 3\O(-2)$ such that $\rho \circ \a = \f_1$.
Since $\G$ maps surjectively onto $\Ker(\f_2)/\Im(\f_1)$, this sheaf has rank zero,
i.e. $\Im(\f_1)$ has rank $2$.
This excludes the possibility $\rank(\a)=1$, because in this case $\Im(\f_1)$
would be isomorphic to $\O(-2)$. If $\rank(\a)=2$, then $\Im(\f_1)$ would be isomorphic to $2\O(-2)$.
In this case $\Ker(\f_2)/\Im(\f_1)$ would have slope $-1$, hence it would destabilise $\G$.
We deduce that $\rank(\a)=3$, hence $\Im(\f_1)=\Ker(\f_2)$ and $\Ker(\f_1) \isom \O(-3)$.
The exact sequence (2.2.5) \cite{drezet-maican} takes the form
\[
0 \lra \O(-3) \stackrel{\f_5}{\lra} \Coker(\f_4) \lra \G \lra 0.
\]
This easily yields the dual to resolution \ref{2.1.3}.
\end{proof}

\begin{prop}
\label{2.2.5}
There exists a geometric quotient $W_3/G_3$ and it is a smooth projective
variety. $W_3/G_3$ is isomorphic to the Hilbert flag scheme
of quintic curves in $\P^2$ containing zero-dimensional subschemes of length $2$.
\end{prop}

\begin{proof}
Before constructing the quotient we notice that its existence already follows
from \cite{drezet-trautmann}.
Let $\L=(\l_1,\l_2,\m_1,\m_2)$ be a polarisation for the action of
$G_3$ on $\W_3$, as in loc.cit.
Using King's criterion of semi-stability \cite{king} we can verify
that for polarisations satisfying $\l_1 < \m_1$ and $\l_1 < \m_2$
the set of stable points $\W_3^{s}(\L)$ coincides with the set of semi-stable points
$\W_3^{ss}(\L)$ and is equal to $W_3$.
According to \cite{drezet-trautmann}, for polarisations satisfying
$\l_2 > 6 \l_1$ and $\m_1 > 3 \m_2$ there is a good and projective
quotient $\W_3^{ss}(\L)/\!/G_3$ containing the smooth geometric quotient $\W_3^s(\L)/G_3$
as an open subset.
We now choose a polarisation satisfying all the above conditions, i.e. satisfying
$0 < \l_1 < 1/7$ and $\l_1 < \m_2 < 1/4$.
We conclude that there is a smooth geometric quotient $W_3/G_3$, which is a projective variety.

Next we give two constructions of $W_3/G_3$, firstly as a bundle and secondly
as a Hilbert flag scheme. The first construction uses the method of \ref{2.2.2},
which consisted of finding successively quotients modulo subgroups. 
Let $W_3'$ be the open subset of $\W_3$ given by the conditions
that $\f_{12} \neq 0$ and that $\f_{22}$ be non-divisible by $\f_{12}$.
The pairs of morphisms $(\f_{12},\f_{22})$ form an open subset
$U \subset \Hom(\O(-1),\O \oplus \O(1))$ and $W_3'$ is the trivial vector bundle over $U$
with fibre $\Hom(\O(-3),\O \oplus \O(1))$.
We represent the elements of $G_3$ by pairs of matrices
\[
(g,h) \in \Aut(\O(-3) \oplus \O(-1)) \times \Aut(\O \oplus \O(1)),
\]
\[
g = \left[
\ba{cc}
g_1 & 0 \\
u & g_2
\ea
\right], \quad \quad h = \left[
\ba{cc}
h_1 & 0 \\
v & h_2
\ea
\right].
\]
Inside $G_3$ we distinguish two subgroups: a unitary subgroup $G_3'$ given by the conditions
that $h$ be the identity morphism, $g_1=1$, $g_2=1$ and a subgroup $G_3''$ given by the
condition that $g$ be the identity morphism.
Consider the $G_3$-invariant subset $\Sigma \subset W_3'$ given by the conditions
\[
\f_{11}= \f_{12} u, \quad \f_{21}= \f_{22} u, \quad u \in \Hom(\O(-3),\O(-1)).
\]
Note that $W_3=W_3' \setminus \Sigma$.
Clearly $\Sigma$ is a subbundle of $W_3'$.
The quotient bundle $E'$ has rank $19$.
The quotient map $W_3' \to E'$ is a geometric quotient modulo $G_3'$.
Moreover, the canonical action of $G_3''$ on $U$ is $E'$-linearised
and the map $W_3' \to E'$ is $G_3''$-equivariant.
Let $\s'$ be the zero-section of $E'$. The restricted map $W_3 \to E' \setminus \s'$
is also a geometric quotient modulo $G_3'$.

We now construct a geometric quotient of $E'$ modulo $G_3''$.
The quotient for the base $U$ can be described explicitly as follows.
On $\P(V^*)$ we consider the trivial vector bundle with fibre $S^2V^*$ and the subbundle
with fibre $vV^*$ at any point $\langle v \rangle \in \P(V^*)$. Let $Q$ be the quotient bundle.
Clearly $U/G_3''$ is isomorphic to $\P(Q)$. Moreover, $U$ is a principal $G_3''$-bundle over $\P(Q)$.
According to 4.2.14 \cite{huybrechts}, $E'$ descends to a vector bundle $E$ on $\P(Q)$.
$E$ is the geometric quotient $E'/G_3''$. Let $\s$ be the zero-section of $E$.
The composed map $W_3 \to E' \setminus \s' \to E \setminus \s$ is a geometric
quotient modulo $G_3'G_3''$. It is now clear that the fibre bundle $\P(E)$
is the geometric quotient $W_3/G_3$.
Thus $W_3/G_3$ is a fibre bundle with fibre $\P^{18}$ and base
a fibre bundle $\P(Q)$ with base $\P^2$ and fibre $\P^2$.

It is clear that $\P(Q)$ is isomorphic to the Hilbert scheme of zero-dimensional subschemes of $\P^2$
of length $2$. Let $F$ be the Hilbert flag scheme from the proposition viewed as a subscheme
of $\P(Q) \times \P(S^5V^*)$.
Consider the map $W_3 \to F$ defined by
\[
\f \lra (\langle \f_{12} \rangle, \langle \f_{22} \mod \f_{12} \rangle, \langle \det(\f) \rangle).
\]
The fibres of this map are obviously $G_3$-orbits. To show that
this map is a geometric quotient we shall construct local sections.
We choose a point $x=(\langle f \rangle, \langle g \mod f \rangle, \langle h \rangle)$ in $F$.
To fix notations we write $f=X$ and
we may assume that $g$ is a quadratic form in $Y$ and $Z$.
There are unique forms $h_1(Y,Z)$ and $h_2(X,Y,Z)$ such that
$h=h_1 +Xh_2$. By hypothesis $h_1$ is divisible by $g$. We put
\[
\s (x)= \left[
\begin{array}{lc}
\phantom{-} h_1/g & f \\
-h_2 & g
\end{array}
\right].
\]
Note that $\s$ extends to a local section in a neighbourhood of $x$ because
$h_2$ and $h_1$, hence also $h_1/g$, depend algebraically on $x$.
\end{proof}

\begin{prop}
\label{2.2.6}
The geometric quotient $W_3/G_3$ is isomorphic to $X_3$.
In particular, $X_3$ is a smooth closed subvariety of $\M(5,3)$.
\end{prop}

\begin{proof}
As above, in order to show that the bijective morphism
$W_3/G_3 \to X_3$ is an isomorphism,
we need to construct resolution \ref{2.1.4} starting from the Beilinson
tableau (2.2.3) \cite{drezet-maican} for $\F$, which takes the form:
\[
\xymatrix
{
3\O(-2) \ar[r]^-{\f_1} & 3\O(-1) \ar[r]^-{\f_2} & \O \\
\O(-2) \ar[r]^-{\f_3} & 4\O(-1) \ar[r]^-{\f_4} & 4\O
}.
\]
As at \ref{2.2.4}, $\Ker(\f_2)$ is equal to $\Im(\f_1)$ and $\Ker(\f_1)$ is isomorphic to $\O(-3)$.
The exact sequence (2.2.5) \cite{drezet-maican} gives the resolution
\[
0 \lra \O(-3) \stackrel{\f_5}{\lra} \Coker(\f_4) \lra \F \lra 0.
\]
We combine this sequence with the exact sequence (2.2.4) \cite{drezet-maican} that reads as follows:
\[
0 \lra \O(-2) \stackrel{\f_3}{\lra} 4\O(-1) \stackrel{\f_4}{\lra} 4\O \lra \Coker(\f_4) \lra 0.
\]
Indeed, $\f_5$ lifts to a map $\O(-3) \to 4\O$ because $\Ext^1(\O(-3),\Coker(\f_3))=0$.
We arrive at the resolution
\[
0 \lra \O(-2) \lra \O(-3) \oplus 4\O(-1) \lra 4\O \lra \F \lra 0.
\]
We have already seen at \ref{2.1.4} how to derive the desired resolution of $\F$
from the above exact sequence.
\end{proof}


\subsection{Geometric description of the strata}

We recall that the stratum $X_0$ of $\M(5,3)$ consists of isomorphism classes
of cokernels of morphisms $\f=(\f_{11},\f_{12})$ as at \ref{2.1.2}(i).
We distinguish a subset $X_{01} \subset X_0$ given by the condition
$\Coker(\f_{12}) \isom \I_x(1) \oplus \O$,
where $\I_x \subset \O$ is the ideal sheaf of a point $x \in \P^2$.
Clearly $X_{01}$ is closed in $X_0$ and has codimension $1$.

\begin{prop}
\label{2.3.1}
The sheaves $\F$ giving points in $X_0 \setminus X_{01}$
are precisely the sheaves admitting a resolution of the form
\[
0 \lra 2\O(-2) \lra \Om^1(2) \lra \F \lra 0.
\]
\end{prop}

\begin{proof}
Assume that $\F$ gives a point in $X_0 \setminus X_{01}$.
From \ref{2.1.2}(i) we have the exact sequence
\[
0 \lra 2\O(-2) \lra \Coker(\f_{12}) \lra \F \lra 0.
\]
By hypothesis $\Coker(\f_{12})$ is isomorphic to $\Om^1(2)$.

Assume now that $\F$ has a resolution as in the proposition.
Combining with the Euler sequence we find an injective morphism
$\f \colon 2\O(-2) \oplus \O(-1) \to 3\O$ such that $\F \isom \Coker(\f)$.
The fact that $\f_{12}$ has linearly independent entries ensures that $\f$
satisfies the conditions from \ref{2.1.2}(i).
\end{proof}

\begin{prop}
\label{2.3.2}
The generic sheaves $\F$ from $X_{01}$ are precisely
the non-split extension sheaves of the form
\[
0 \lra \J_x(1) \lra \F \lra \O_Z \lra 0,
\]
such that there is a global section
of $\F$ taking the value 1 at every point of $Z$. Here $\J_x \subset \O_C$ is the ideal sheaf
of a point $x$ on a quintic curve $C \subset \P^2$ and $Z \subset C$ is a union of four distinct
points, also distinct from $x$, no three of which are colinear.

There is an open subset inside $X_{01}$ consisting of the isomorphism classes
of all sheaves of the form $\O_C(1)(P_1+P_2+P_3+P_4-P_5)$, where $C \subset \P^2$ is a smooth
quintic curve, $P_1, \ldots, P_5$ are distinct points on $C$ and $P_1, P_2, P_3, P_4$ are in
general linear position.
\end{prop}

\begin{proof}
We begin by noting that the sheaves
giving points in $X_{01}$ are precisely the sheaves $\F$ admitting a resolution
\[
0 \lra 2\O(-2) \oplus \O(-1) \stackrel{\f}{\lra} 3\O \lra \F \lra 0,
\]
\[
\f = \left[
\ba{ccc}
q_1 & q_2 & 0 \\
\star & \star & \ell_1 \\
\star & \star & \ell_2
\ea
\right],
\]
where $q_1, q_2$ are linearly independent two-forms
and $\ell_1,\ell_2$ are linearly independent one-forms.
For generic $\F$, $q_1$ and $q_2$ have no common linear factor
and the conic curves they define intersect in the union $Z$ of four distinct points,
no three of which are colinear and also distinct from the common zero of $\ell_1$ and $\ell_2$.
We apply the snake lemma to the exact diagram:
\[
\xymatrix
{
& & 0 \ar[d] & 0 \ar[d] \\
& 0 \ar[r] & \O(-1) \ar[r]^-{\scriptsize \left[ \begin{array}{c} \ell_1 \\ \ell_2 \end{array} \right]} \ar[d]
& 2\O \ar[r] \ar[d] & \I_x(1) \ar[r] & 0 \\
& 0 \ar[r] &  2\O(-2) \oplus \O(-1) \ar[r]^-{\f} \ar[d]
& \O \oplus 2\O \ar[r] \ar[d] & \F \ar[r] & 0 \\
0 \ar[r] & \O(-4) \ar[r]^-{\scriptsize \left[ \! \begin{array}{c} -q_2 \\ \phantom{-} q_1 \end{array} \! \right]}
& 2\O(-2) \ar[d] \ar[r]^-{\scriptsize \left[ \begin{array}{cc} q_1 &  q_2 \end{array} \right]}
& \O \ar[r] \ar[d] & \O_Z \ar[r] & 0 \\
& & 0 & 0
}
\]
The vertical maps are injections into the second factors, respectively projections onto the first factors.
We get the exact sequence
\[
0 \lra \O(-4) \lra \I_x(1) \lra \F \lra \O_Z \lra 0,
\]
from which the conclusion follows.
For the converse we apply the horseshoe lemma to the diagram below.
\begin{table}[ht]
\[
\xymatrix
{
& & & 0 \ar[d] \\
& 0 \ar[d] & & \O(-4) \ar[dll]_-{\d} \ar[d] \\
& \O(-4) \ar[d] & & 2\O(-2) \ar[dll]_-{\g} \ar[ddll]_-{\b} \ar[d] \\
& \I_x(1) \ar[d]_-{\nu} & & \O \ar[dl]_-{\a} \ar[d]^-{\pi} \\
0 \ar[r] & \J_x(1) \ar[r]^-{\xi} \ar[d] & \F \ar[r]^-{\z} & \O_Z \ar[r] \ar[d] & 0 \\
& 0 & & 0
}
\]
\end{table}
By hypothesis, the morphism $\pi \colon \O \to \O_Z$ lifts to a morphism $\a \colon \O \to \F$.
Then $\b$, $\g$, $\d$ are defined in the usual way and we claim that $\d \neq 0$.
If $\d$ were zero, then $\g$ would factor through a morphism $\Ker(\pi) \to \I_x(1)$.
Since $\Ext^1(\O_Z, \I_x(1))=0$, this morphism would lift to a map $\eta \colon \O \to \I_x(1)$.
The composite map
$
2\O(-2) \to \O \stackrel{\a}{\to} \F
$
would then coincide with the composition
\[
2\O(-2) \lra \O \stackrel{-\eta}{\lra} \I_x(1) \stackrel{\nu}{\lra} \J_x(1) \stackrel{\xi}{\lra} \F,
\]
hence $\a + \xi \circ \nu \circ \eta$ would factor through a morphism $\s \colon \O_Z \to \F$.
We would have
$
\pi = \z \circ \a = \z \circ \a + \z \circ \xi \circ \nu \circ \eta = \z \circ \s \circ \pi,
$
hence $\z \circ \s$ would be the identity morphism.
The extension would split, contradicting our hypothesis on $\F$.
Combining the resolutions for $\J_x(1)$ and $\O_Z$ and cancelling $\O(-4)$ we obtain the resolution
\[
0 \lra 2\O(-2) \lra \O \oplus \I_x(1) \lra \F \lra 0.
\]
From this we easily get a resolution for $\F$ as at the beginning of this proof.

Assume now that $C$ is smooth and write $Z=\{ P_1, P_2, P_3, P_4 \}$, $x=P_5$.
Clearly, the only non-trivial extension sheaf of $\O_Z$ by $\J_x(1)$ is isomorphic to 
$\F=\O_C(1)(P_1+P_2+P_3+P_4-P_5)$.
To finish the proof of the proposition we must show that $\F$ has a global section
that does not vanish at any point of $Z$. For $1 \le i \le 4$, let $\e_i \colon \H^0(\O_Z) \to \C$
be the linear form of evaluation at $P_i$.
Let $\d \colon \H^0(\O_Z) \to \H^1(\J_x(1))$ be the connecting homomorphism
in the long exact cohomology sequence associated to the short exact sequence
\[
0 \lra \O_C(1)(-x)  \lra \F \lra \O_Z \lra 0.
\]
We must show that each $\e_i$ is not orthogonal to $\operatorname{Ker}(\d)$.
This is equivalent to saying that $\e_i$ is not in the image of the dual map $\d^*$.
By Serre duality $\d^*$ is the restriction morphism
\[
\xymatrix
{
\H^0(\O_C(-1)(x) \tensor \omega_C) \egal[d] \ar[r]
& \H^0((\O_C(-1)(x) \tensor \omega_C)|_Z) \egal[d] \\
\H^0(\O_C(1)(x)) \egal[d] & \H^0(\O_C(1)(x)|_Z) \egal[d] \\
\H^0(\O_C(1)) & \H^0(\O_C(1)|_Z)
}.
\]
The identity $\H^0(\O_C(1)(x)) \isom \H^0(\O_C(1)) \isom V^*$ follows from the fact that the
connecting homomorphism in the long exact cohomology sequence associated to the short
exact sequence
\[
0 \lra \O_C(1) \lra \O_C(1)(x) \lra \C_x \lra 0
\]
is non-zero. Indeed, its dual is the restriction map $\H^0(\O_C(1)) \to \H^0(\O_C(1)|_x)$.
This map is clearly non-zero.
Now $\d^*(u)$ is a multiple of $\e_i$ if and only if the linear form $u$ vanishes at $P_j$ for all $j \neq i$.
By hypothesis the points $P_j$, $j\neq i$, are non-colinear, so there is no such form $u$ and we conclude that
$\e_i$ is not in the image of $\d^*$.
\end{proof}

\begin{prop}
\label{2.3.3}
The sheaves $\F$ in $X_1$ are precisely the non-split extension sheaves of the form
\[
0 \lra \E^{\D}(1) \lra \F \lra \C_x \lra 0,
\]
where $\C_x$ is the structure sheaf of a point $x \in \P^2$ and $\E$ is in $X_2$.
Here $\E^{\D}={\mathcal Ext}^1(\E,\omega_{\P^2})$
signifies the dual sheaf of $\E$. Taking into account the duality isomorphism 
\cite{maican-duality}, the sheaves $\E^{\D}(1)$ are precisely the sheaves
$\G$ in the dual stratum $X_2^{\D} \subset \M(5,2)$ defined by the relations
\[
h^0(\G(-1))=0, \qquad h^1(\G)=1, \qquad h^1(\G \tensor \Om^1(1))=3.
\]
The generic sheaves in $X_1$ are of the form
$\O_C(2)(-P_1-P_2-P_3+P_4)$, where $C \subset \P^2$ is a smooth quintic curve,
$P_i$ are four distinct points on $C$ and $P_1, P_2, P_3$ are non-colinear.
In particular, $X_1$ lies in the closure of $X_{01}$.
\end{prop}

\begin{proof}
Let $\F$ be in $X_1$. As in the proof of \ref{2.3.2}, the snake lemma gives
an exact sequence
\[
0 \lra \Ker(\f_{11}) \stackrel{\a}{\lra} \Coker(\f_{22}) \lra \F \lra \Coker(\f_{11}) \lra 0.
\]
Because of the form of $\f_{11}$ given at \ref{2.1.2}(ii),
we have the isomorphisms $\Ker(\f_{11})\isom \O(-3)$ and
$\Coker(\f_{11})\isom \C_x$ for a point $x \in \P^2$.
Denoting $\G= \Coker(\a)$, we have an extension
\[
0 \lra \G \lra \F \lra \C_x \lra 0.
\]
Again from \ref{2.1.2}(ii), we know that $\f_{22}$ is injective, hence $\G$ has a resolution of the form
\[
0 \lra \O(-3) \oplus 2\O(-1) \stackrel{\psi}{\lra} 3\O \lra \G \lra 0,
\]
with $\psi_{12}=\f_{22}$.
According to the proof of \ref{2.1.3},  $\G$ is in the dual stratum $X_2^{\D}$.

Conversely, assume that $\F$ is an extension as in the proposition.
Using the horseshoe lemma we combine the resolutions
\[
0 \lra \O(-3) \oplus 2\O(-1) \stackrel{\psi}{\lra} 3\O \lra \G \lra 0
\]
and
\[
0 \lra \O(-3) \lra 2\O(-2) \lra \O(-1) \lra \C_x \lra 0
\]
to obtain a resolution
\[
0 \lra \O(-3) \lra \O(-3) \oplus 2\O(-2) \oplus 2\O(-1) \lra \O(-1) \oplus 3\O \lra \F \lra 0.
\]
Note that $\Ext^1(\C_x, 3\O)=0$, so we can use the arguments at \ref{2.3.2}
to show that the extension would split if the morphism $\O(-3) \to \O(-3)$
in the above complex were zero.
We deduce that this morphism is non-zero,
so we may cancel $\O(-3)$ to get a resolution as in \ref{2.1.2}(ii).

The part of the proposition concerning generic sheaves follows from the corresponding
part of proposition \ref{2.3.4} below.

To see that $X_1$ is included in $\overline{X}_{01}$ we choose a point in $X_1$ represented by
$\O_C(2)(-P_1-P_2-P_3+P_4)$. We may assume that the line through $P_1$ and $P_2$
intersects $C$ at five distinct points $P_1, P_2, Q_1, Q_2, Q_3$, which are also distinct
from $P_3$ and $P_4$. Then
\[
\O_C(2)(-P_1-P_2-P_3+P_4) \isom \O_C(1)(Q_1+Q_2+Q_3-P_3+P_4).
\]
Clearly, we can find points $R_1, R_2, R_3$ on $C$,
converging to $Q_1, Q_2, Q_3$ respectively, which are distinct from $P_3$ and such that
$R_1, R_2, R_3, P_4$ are in general linear position. Then $\O_C(1)(R_1+R_2+R_3+P_4-P_3)$
represents a point in $X_{01}$ converging to the chosen point in $X_1$.
\end{proof}

We recall from the proof of \ref{2.1.3} that the sheaves $\G$ giving points in the dual stratum
$X_2^\D \subset \M(5,2)$ are precisely the sheaves that admit a resolution of the form
\[
0 \lra \O(-3) \oplus 2\O(-1) \stackrel{\psi}{\lra} 3\O \lra \G \lra 0,
\]
where $\psi_{12}$ has linearly independent maximal minors.
We consider the open subset $X_{20}^{\D}$ of $X_2^{\D}$ given by the condition that the
maximal minors of $\psi_{12}$ have no common linear factor and we denote
$X_{21}^{\D}= X_2^{\D} \setminus X_{20}^{\D}$.

\begin{prop}
\label{2.3.4}
\textup{(i)} The sheaves $\G$ from $X_{20}^{\D}$
are precisely the twisted ideal sheaves $\J_Z(2)$,
where $Z \subset \P^2$ is a zero-dimensional scheme of length $3$ not contained in a line,
contained in a quintic curve $C \subset \P^2$, and $\J_Z \subset \O_C$ is its ideal sheaf.

The generic sheaves in $X_2^{\D}$
are of the form $\O_C(2)(-P_1-P_2-P_3)$,
where $C$ is a smooth quintic curve and $P_1, P_2, P_3$ are non-colinear points on $C$.

By duality, the generic sheaves in $X_2$ are of the form $\O_C(1)(P_1+P_2+P_3)$.
In particular, $X_2$ lies in the closure of $X_1$.

\medskip

\noindent
\textup{(ii)} The sheaves $\G$ from $X_{21}^{\D}$ are precisely the extension sheaves of the form
\[
0 \lra \O_L(-1) \lra \G \lra \O_C(1) \lra 0
\]
where $L \subset \P^2$ is a line, $C \subset \P^2$ is a quartic curve and such that the image
of $\G$ under the canonical map
\[
\Ext^1(\O_C(1),\O_L(-1)) \lra \Ext^1(\O(1),\O_L(-1))
\]
is non-zero.
\end{prop}

\begin{proof}
(i) According to 4.5 and 4.6 \cite{modules-alternatives}, $\Coker(\psi_{12}) \isom \I_Z(2)$,
where $Z \subset \P^2$ is a zero-dimensional scheme of length $3$ not contained in a line
and $\I_Z \subset \O$ is its ideal sheaf. Conversely, every $\I_Z(2)$ is the cokernel of some
morphism $\psi_{12} \colon 2\O(-1) \to 3\O$ whose maximal minors are linearly independent and
have no common linear factor. Thus, the sheaves $\G \in X_{20}^{\D}$ are precisely the
cokernels of injective morphisms $\O(-3) \to \I_Z(2)$. If $C$ is the quintic curve defined by the
inclusion $\O(-3) \subset \I_Z(2) \subset \O(2)$, then it is easy to see that $\G \isom \J_Z(2)$.

To see that $X_2$ is included in $\overline{X}_1$ we choose a generic sheaf in $X_2$ of the form
$\O_C(1)(P_1+P_2+P_3)$. We may assume that the line through $P_1$ and $P_2$
intersects $C$ at five distinct points $P_1, P_2, Q_1, Q_2, Q_3$.
For non-colinear points $R_1, R_2, R_3$ on $C$, converging to $Q_1, Q_2, Q_3$ respectively,
the sheaf
\[
\O_C(2)(-R_1-R_2-R_3+P_3) \isom \O_C(1)(P_1+P_2+P_3+Q_1+Q_2+Q_3-R_1-R_2-R_3)
\]
represents a point in $X_1$ converging to the point given by $\O_C(1)(P_1+P_2+P_3)$.

\medskip

\noindent
(ii) Let $\ell$ be a common linear factor of the maximal minors of $\psi_{12}$.
Consider the line $L$ with equation $\ell=0$.
According to 3.3.3 \cite{drezet-maican}, $\Coker(\psi_{12}) \isom \E_L$, where $\E_L$
is the unique non-split extension
\[
0 \lra \O_L(-1) \lra \E_L \lra \O(1) \lra 0.
\]
Conversely, every $\E_L$ is the cokernel of some morphism $\psi_{12} \colon 2\O(-1) \to 3\O$
with linearly independent maximal minors which have a common linear factor.
Thus, the sheaves $\G$ giving points in $X_{21}^{\D}$ are precisely the cokernels of the injective morphisms
$\O(-3) \to \E_L$. Let $C \subset \P^2$ be the quartic curve defined by the composition
$\O(-3) \to \E_L \to \O(1)$. We apply the snake lemma to the diagram with exact rows
\[
\xymatrix
{
0 \ar[r] & \O(-3) \ar[r] \ar@{=}[d] & \E_L \ar[d]^\a \ar[r] & \G \ar[r] & 0 \\
0 \ar[r] & \O(-3) \ar[r] &\O(1) \ar[r] & \O_C(1) \ar[r] & 0
}.
\]
As $\Ker(\a) \isom \O_L(-1)$, we obtain an extension
\[
0 \lra \O_L(-1) \lra \G \lra \O_C(1) \lra 0
\]
which maps to the class of $\E_L$ in $\P(\Ext^1(\O(1),\O_L(-1)))$.
The converse is clear, in view of the fact that $\Ext^1(\O(1),\O_L(-1))\isom \C$.
\end{proof}

\begin{prop}
\label{2.3.5}
The sheaves $\F$ giving points in $X_3$ are precisely the twisted ideal sheaves
$\J_Z(2)$, where $Z \subset \P^2$ is a zero-dimensional scheme of length $2$ contained
in a quintic curve $C$ and $\J_Z \subset \O_C$ is its ideal sheaf.

The generic sheaves in $X_3$ are of the form $\O_C(1)(P_1+P_2+P_3)$, where $C \subset \P^2$
is a smooth quintic curve and $P_1, P_2, P_3$ are distinct colinear points on $C$.
In particular, $X_3$ lies in the closure of $X_2$.
\end{prop}

\begin{proof}
Adopting the notations of \ref{2.1.4}, we notice that the restriction of $\f$ to $\O(-1)$
has cokernel $\I_Z(2)$, where $Z$ is the
intersection of the line with equation $\f_{12}=0$ and the conic with equation $\f_{22}=0$.
Thus the sheaves $\F$ in $X_3$ are precisely the cokernels of injective morphisms
$\O(-3) \to \I_Z(2)$.
Let $C$ be the quintic curve defined by the inclusion $\O(-3) \subset \I_Z(2)
\subset \O(2)$. Clearly $\F \isom \J_Z(2)$.

To see that $X_3 \subset \overline{X}_2$ choose a generic sheaf $\O_C(1)(P_1+P_2+P_3)$
in $X_3$. Clearly, we can find non-colinear points $Q_1, Q_2, Q_3$ on $C$ converging to $P_1, P_2, P_3$
respectively. Then $\O_C(1)(Q_1+Q_2+Q_3)$ represents a point in $X_2$ converging to the chosen
point in $X_3$.
\end{proof}

\noindent
From what was said above we can summarise:

\begin{prop}
\label{2.3.6}
$\{ X_0 \setminus X_{01}, X_{01}, X_1, X_2, X_3 \}$ represents a stratification
of $\M(5,3)$ by locally closed irreducible subvarieties of codimensions $0, 1, 2, 3, 4$.
\end{prop}



\section{Euler characteristic one or four}


\subsection{Locally free resolutions for semi-stable sheaves}

\begin{prop}
\label{3.1.1}
Every sheaf  $\F$ giving a point in $\M(5,1)$ and satisfying the condition
$h^1(\F)=0$ also satisfies the condition $h^0(\F(-1))=0$.
These sheaves are precisely the sheaves with resolution
\[
0 \lra 4\O(-2) \stackrel{\f}{\lra} 3\O(-1) \oplus \O \lra \F \lra 0,
\]
where $\f_{11}$ is not equivalent to a morphism represented by a matrix of the form
\[
\left[
\begin{array}{cc}
\psi & 0 \\
\star & \star
\end{array}
\right], \quad \text{with} \quad \psi \colon m\O(-2) \lra m\O(-1), \quad m=1,2,3.
\]
\end{prop}

\begin{proof}
According to 4.2 \cite{maican}, every sheaf $\G$ giving a point in $\M(5,4)$
and satisfying the condition $h^0(\G(-1))=0$ also satisfies the condition $h^1(\G)=0$ and has a resolution
\[
0 \lra \O(-2) \oplus 3\O(-1) \stackrel{\f}{\lra} 4\O \lra \G \lra 0,
\]
where $\f_{12}$ is not equivalent to a morphism represented by a matrix of the form
\[
\left[
\begin{array}{cc}
\star & \psi \\
\star & 0
\end{array}
\right], \quad \text{with} \quad \psi \colon m\O(-1) \lra m\O, \quad m=1,2,3.
\]
The result follows by duality.
\end{proof}

\begin{prop}
\label{3.1.2}
Let $\F$ be a sheaf giving a point in $\M(5,1)$ satisfying the conditions $h^1(\F)=1$
and $h^0(\F(-1))=0$.
Then $h^0(\F \tensor \Om^1(1)) = 0$ or $1$. The sheaves in the first case are precisely the
sheaves that have a resolution of the form
\[
\tag{i}
0 \lra \O(-3) \oplus \O(-2) \stackrel{\f}{\lra} 2\O \lra \F \lra 0,
\]
where $\f_{12}$ and $\f_{22}$ are linearly independent two-forms.
The sheaves from the second case are precisely the sheaves with resolution
\[
\tag{ii}
0 \lra \O(-3) \oplus \O(-2) \oplus \O(-1) \stackrel{\f}{\lra} \O(-1) \oplus 2\O
\lra \F \lra 0,
\]
\[
\f= \left[
\begin{array}{ccc}
q & \ell & 0 \\
\f_{21} & \f_{22} & \ell_1 \\
\f_{31} & \f_{32} & \ell_2
\end{array}
\right],
\]
where $\ell$ is non-zero, $q$ is non-divisible by $\ell$ and $\ell_1, \ell_2$ are linearly
independent one-forms.
\end{prop}

\begin{proof}
Let $\F$ give a point in $\M(5,1)$ and satisfy the conditions $h^1(\F)=1$ and $h^0(\F(-1))=0$.
Put $m=h^0(\F \tensor \Om^1(1))$. The Beilinson free monad (2.2.1) \cite{drezet-maican} for $\F$ reads
\[
0 \lra 4\O(-2) \oplus m\O(-1) \lra (m+3)\O(-1) \oplus 2\O \lra \O \lra 0
\]
and gives the resolution
\[
0 \lra 4\O(-2) \oplus m\O(-1) \lra \Om^1 \oplus m\O(-1) \oplus 2\O \lra \F \lra 0.
\]
Combining this with the standard resolution for $\Om^1$ we obtain the following exact
sequence:
\[
0 \lra \O(-3) \oplus 4\O(-2) \oplus m\O(-1) \stackrel{\f}{\lra} 3\O(-2) \oplus m\O(-1)
\oplus 2\O \lra \F \lra 0,
\]
with $\f_{13}=0$, $\f_{23}=0$.
As in the proof of \ref{2.1.4}, we have $\rank(\f_{12})=3$.
Canceling $3\O(-2)$ we get the resolution
\[
0 \lra \O(-3) \oplus \O(-2) \oplus m\O(-1) \stackrel{\f}{\lra} m\O(-1) \oplus 2\O \lra \F \lra 0,
\]
with $\f_{13}=0$. From the injectivity of $\f$ we must have $m \le 2$.
If $m=2$, then $\Coker(\f_{23})$ is a destabilising subsheaf of $\F$.
We conclude that $m=0$ or 1.

Assume that $h^0(\F \tensor \Om^1(1))=0$. We arrive at resolution (i).
If $\f_{12}$ and $\f_{22}$ were linearly dependent, then $\F$
would have a destabilising subsheaf of the form $\O_C$, for a conic curve $C \subset \P^2$.
Conversely, we assume that $\F$ has resolution (i)
and we must show that $\F$ cannot have a destabilising subsheaf $\E$.
We may restrict our attention to semi-stable sheaves $\E$.
As $\F$ is generated by global sections, we must have $h^0(\E) < h^0(\F)=2$.
Thus $\E$ is in $\M(r,1)$ for some $1 \le r \le 4$ and we have $h^1(\E)=0$.
Moreover, $\H^0(\E \tensor \Om^1(1))$ vanishes because the
corresponding cohomology group for $\F$ vanishes.
This excludes the possibility $r=1$. In the case $r=2$, $\E$ is the structure
sheaf of a conic curve, but this, by virtue of our hypothesis on $\f_{12}$ and
$\f_{22}$, is not allowed.
If $\E$ is in $\M(3,1)$, then, according to \cite{lepotier}, $\E$ has resolution
\[
0 \lra 2\O(-2) \lra \O(-1) \oplus \O \lra \E \lra 0.
\]
If $\E$ is in $\M(4,1)$, then, from the description of this moduli space found in
\cite{drezet-maican}, we see that $\E$ has resolution
\[
0 \lra 3\O(-2) \lra 2\O(-1) \oplus \O \lra \E \lra 0.
\]
It is easy to see that the first exact sequence must fit into a commutative diagram
\[
\xymatrix
{
0 \ar[r] & 2\O(-2) \ar[r]^-{\psi} \ar[d]^-{\b} & \O(-1) \oplus \O \ar[r] \ar[d]^-{\a}
& \E \ar[r] \ar[d] & 0 \\
0 \ar[r] & \O(-3) \oplus \O(-2) \ar[r]^-{\f} & 2\O \ar[r] & \F \ar[r] & 0
}.
\]
From the fact that $\a$ and $\a(1)$ are injective on global
sections we see that $\Coker(\a)$ is supported on a line.
This is impossible because $\O(-3)$ maps injectively to $\Coker(\b)$
which maps injectively to $\Coker(\a)$. The same argument applies to the
second exact sequence as well,
except that $\Coker(\a)$ this time would be supported on a point.

Assume now that $h^0(\F \tensor \Om^1(1))=1$. We arrive at resolution (ii).
If $\ell_1, \ell_2$ were linearly dependent, then
$\F$ would have a destabilising subsheaf of the form $\O_L$, for a line $L \subset \P^2$.
If $\ell=0$, then $\F$ would have a destabilising quotient sheaf of the form $\O_C(-1)$,
for a conic curve $C \subset \P^2$.
If $\ell$ divided $q$, then $\F$ would have a destabilising quotient sheaf of the form $\O_L(-1)$.
Conversely, we assume that $\F$ has resolution (ii) and we must
show that there is no destabilising subsheaf.
Let $x$ be the point with equations $\ell_1=0$, $\ell_2=0$ and let $Z \subset \P^2$
be the zero-dimensional subscheme of length $2$ given by the equations $\ell=0$, $q=0$.
We apply the snake lemma to the exact diagram:
\[
\xymatrix
{
& & 0 \ar[d] & 0 \ar[d] \\
& 0 \ar[r] & \O(-1) \ar[r]^-{\f_{23}} \ar[d] & 2\O \ar[r] \ar[d] & \I_x(1) \ar[r] & 0 \\
& 0 \ar[r] & \O(-3) \oplus \O(-2) \oplus \O(-1) \ar[r]^-{\f} \ar[d]
& \O(-1) \oplus 2\O \ar[r] \ar[d]  & \F \ar[r] & 0\\
0 \ar[r] & \O(-4) \ar[r]^-{\scriptsize \left[ \! \ba{r} -\ell \\ q \ea \! \right]} & \O(-3) \oplus \O(-2) \ar[d]
\ar[r]^-{\scriptsize \left[ \begin{array}{cc} q & \ell \end{array} \right]}
& \O(-1) \ar[r] \ar[d] & \O_Z \ar[r] & 0\\
& & 0 & 0
}
\]
We get the exact sequence
\[
0 \lra \O(-4) \lra \I_x(1) \lra \F \lra \O_Z \lra 0.
\]
Let $C$ be the quintic curve defined by the inclusion
$\O(-4) \subset \I_x(1) \subset \O(1)$. We obtain an exact sequence:
\[
0 \lra \J_x(1) \lra \F \lra \O_Z \lra 0,
\]
where $\J_x \subset \O_C$ is the ideal sheaf of $x$ on $C$.
Let $\F' \subset \F$ be a non-zero subsheaf of multiplicity at most 4.
Denote by $\CC'$ its image in $\O_Z$ and put $\K = \F' \cap \J_x(1)$.
By \cite{maican}, lemma 6.7, there is a sheaf
$\A \subset \O_C(1)$ containing $\K$ such that $\A/\K$ is supported on finitely many points
and $\O_C(1)/\A \isom \O_S(1)$ for a curve $S \subset \P^2$ of degree $d \le 4$.
The slope of $\F'$ can be estimated as follows:
\begin{align*}
P_{\F'}(t) & = P_{\K}(t) + h^0(\CC') \\
& = P_{\A}(t) - h^0(\A/\K) + h^0(\CC') \\
& = P_{\O_C}(t+1) - P_{\O_S}(t+1) - h^0(\A/\K) + h^0(\CC') \\
& = (5-d)t + \frac{d^2-5d}{2} - h^0(\A/\K) + h^0(\CC'), \\
p(\F') & = -\frac{d}{2} + \frac{h^0(\CC') - h^0(\A/\K)}{5-d} \le -\frac{d}{2} + \frac{2}{5-d} < \frac{1}{5} = p(\F).
\end{align*}
We conclude that $\F$ is semi-stable.
\end{proof}

\begin{prop}
\label{3.1.3}
There are no sheaves $\F$ giving points in $\M(5,1)$
and satisfying the conditions $h^0(\F(-1))=0$ and $h^1(\F)=2$.
\end{prop}

\begin{proof}
By duality, we need to show that there are no sheaves $\G$ in $\M(5,4)$
satisfying the conditions $h^0(\G(-1))=2$ and $h^1(\G)=0$. Assume that there is such a sheaf $\G$.
Write $m=h^1(\G \tensor \Om^1(1))$. The Beilinson monad gives a resolution
\[
0 \lra 2\O(-2) \stackrel{\eta}{\lra} 3\O(-2) \oplus (m+3)\O(-1)
\stackrel{\f}{\lra} m\O(-1) \oplus 4\O \lra \G \lra 0,
\]
\[
\eta = \left[
\ba{c}
0 \\ \psi
\ea
\right].
\]
Here $\f_{12}=0$. As $\G$ maps surjectively onto $\Coker(\f_{11})$,
the latter has rank zero, forcing $m \le 3$.
In the case $m=3$, $\Coker(\f_{11})$ has Hilbert polynomial $P(t)=3t$, so the semi-stability of $\G$
gets contradicted. Thus $m \le 2$.

We claim that any matrix representing a morphism equivalent to $\psi$
has three linearly independent entries on each column.
The argument uses the fact that $\G$ has no zero-dimensional torsion and is analogous
to the proof that the vector space $H$ from \ref{2.1.4} has dimension $3$.
Thus we may assume that one of the columns of $\psi$ is
\[
\left[ \begin{array}{cccccc}
0 & \cdots & 0 & X & Y & Z
\end{array} \right]^\T.
\]
Let $\f_0$ be the matrix made of the last three columns of $\f_{22}$.
The rows of $\f_0$ are linear combinations of the rows of the matrix
\[
\left[
\begin{array}{ccc}
-Y & \phantom{-}X & 0 \\
-Z & \phantom{-}0 & X \\
\phantom{-}0 & -Z & Y
\end{array}
\right].
\]
It is easy to see that the elements on any row of $\f_0$ are linearly dependent.
The rows of $\f_0$ cannot span a vector space of dimension $1$,
otherwise $\f_{22}$ would be equivalent to a morphism
represented by a matrix having a zero-column, hence
$\O(-1) \subset \Ker(\f)$, which is absurd.
$\Ker(\f_0)$ is isomorphic to $\O(-2)$ because $\f_0$ has at least two linearly independent rows.
This excludes the case $m=0$ because in that case $\f_0= \f_{22}$
and $\Ker(\f_{22}) \isom 2\O(-2)$.
In the remaining two cases we shall prove that the rows of $\f_0$ cannot span a vector space of
dimension $2$. We argue by contradiction. Assume that $m=2$ and that $\f_0$ is equivalent
to a matrix of the form
\[
\left[
\ba{c}
0 \\
\xi
\ea
\right],
\]
where $\xi$ is a $2 \times 3$-matrix with linearly independent rows.
Then $\Ker(\xi) \isom \O(-2)$ and $\Coker(\xi) \isom \O_L(1)$ for a line $L \subset \P^2$.
The first isomorphism is obvious and tells us that the maximal minors of $\xi$ are linearly
independent and have a common linear factor, say $\ell$.
Let $L \subset \P^2$ be the line with equation $\ell=0$.
$\Coker(\xi)$ is supported on $L$ and has Hilbert polynomial $P(t)=t+2$.
Moreover, it is easy to see that $\xi$ has rank $1$ at every point of $L$,
hence $\Coker(\xi)$ has no zero-dimensional torsion. This proves the second isomorphism.
We now use the argument from the proof of \ref{2.1.4}. There is a commutative diagram
\[
\xymatrix
{
3\O(-1) \ar[d] \ar[r]^-{\xi} & 2\O \ar[r] \ar[d] & \O_L(1) \ar[r] \ar[d] & 0 \\
3\O(-2) \oplus 5\O(-1) \ar[r]^-{\f} & 2\O(-1) \oplus 4\O \ar[r] & \G \ar[r] & 0
}
\]
in which the first two vertical maps are injective.
The induced morphism $\O_L(1) \to \G$ is zero because both sheaves are stable and
$p(\O_L(1)) > p(\G)$.
Thus the map $4\O \to \G$ is not injective on global sections. On the other hand,
$\H^0(\Coker(\eta))$ vanishes, hence the map $4\O \to \G$ is injective
on global sections. We have arrived at a contradiction. We conclude that the rows of $\f_0$
span a vector space of dimension $3$.

Modulo elementary operations on rows and columns, $\psi$ is equivalent to a morphism represented
by a matrix having one of the following forms:
\[
\left[
\ba{cc}
0 & 0 \\
0 & 0 \\
X & R \\
Y & S \\
Z & T
\ea
\right] \qquad \text{or} \qquad \left[
\ba{cc}
0 & 0 \\
X & 0 \\
Y & R \\
Z & S \\
0 & T
\ea
\right] \qquad \text{or} \qquad \left[
\ba{cc}
X & 0 \\
Y & 0 \\
Z & R \\
0 & S \\
0 & T
\ea
\right].
\]
Here $R, S, T$ form a basis of $V^*$. In the first case the triple $(R,S,T)$ is a multiple of
$(X,Y,Z)$, because, as we saw above, $\Ker(\f_0) \isom \O(-2)$. Thus $\psi$ is represented
by a matrix with a zero-column. This is absurd.
In the second case we can perform elementary row operations on the matrix
\[
\left[
\begin{array}{cc}
\phantom{-}X & 0 \\
\phantom{-}0 & X \\
-Z & Y
\end{array}
\right] \quad \text{to get the matrix} \quad \left[
\begin{array}{cc}
-S & \phantom{-}R \\
-T & \phantom{-}0 \\
\phantom{-}0 & -T \\
\end{array}
\right].
\]
It follows that
\[
\spann\{ X \} = \spann\{ X,Z \} \cap \spann\{ X,Y \}= \spann\{ S,T \} \cap \spann\{R,T\}
=\spann\{T\}
\]
and $(-S,R)=a(-Z,Y)+(bX,cX)$ for some $a, b , c \in \C$.
Thus $\psi$ is equivalent to the morphism represented by the matrix
\[
\left[
\begin{array}{cccccc}
0 & X & Y & Z & 0 \\
0 & 0 & 0 & 0 & X
\end{array}
\right]^{\T}.
\]
This, as we saw above, is not possible. In the third case we can perform elementary row
operations on the matrix
\[
\left[
\begin{array}{c}
0 \\ X \\ Y
\end{array}
\right] \quad \text{to get the matrix} \quad \left[
\begin{array}{c}
S \\ T \\ 0
\end{array}
\right].
\]
Thus, we may assume that
$S=X$, $T=Y$, $R=Z$. Performing elementary row and column operations on
$\psi$ we can get a matrix with three zeros on a column.
This, as we saw above, is not possible.
Thus far we have eliminated the case when $m=2$.
The case when $m=1$ can be eliminated in an analogous fashion.
We conclude that there are no sheaves $\G$ as above.
\end{proof}

\begin{prop}
\label{3.1.4}
There are no sheaves $\F$ giving points in $\M(5,1)$
and satisfying the conditions $h^0(\F(-1))=0$ and $h^1(\F) \ge 2$.
\end{prop}

\begin{proof}
The argument is the same as at proposition \ref{2.1.6} or at 3.2.3 \cite{drezet-maican}.
Using the Beilinson monad for $\F(-1)$ we see that the open subset of $\M(5,1)$ given by the condition
$h^0(\F(-1))=0$ is parametrised by an open subset $M$ inside the space of monads of the form
\[
0 \lra 9\O(-1) \stackrel{A}{\lra} 13 \O \stackrel{B}{\lra} 4\O(1) \lra 0.
\]
The map $\Phi \colon M \to \Hom(13\O, 4\O(1))$ is defined by
$\Phi(A,B)=B$. Using the vanishing of $\H^1(\F(1))$ for an arbitrary sheaf in $\M(5,1)$,
we prove that $\Phi$ has surjective differential at every point of $M$.
This further leads to the conclusion that the set of monads in $M$ whose cohomology
sheaf $\F$ satisfies $h^1(\F) \ge 2$ is included in the closure
of the set of monads for which $h^1(\F)=2$.
According to \ref{3.1.3}, the latter set is empty, hence the former set is empty, too.
\end{proof}

\begin{prop}
\label{3.1.5}
The sheaves $\F$ giving points in $\M(5,1)$
and satisfying the condition $h^0(\F(-1))>0$
are precisely the sheaves with resolution of the form
\[
0 \lra 2\O(-3) \stackrel{\f}{\lra} \O(-2) \oplus \O(1) \lra \F \lra 0,
\]
\[
\f= \left[
\ba{cc}
\ell_1 & \ell_2 \\
f_1 & f_2
\ea
\right],
\]
where $\ell_1, \ell_2$ are linearly independent one-forms.
For these sheaves we have $h^0(\F(-1))=1$ and $h^1(\F)=2$.
These sheaves are precisely the non-split extension sheaves of the form
\[
0 \lra \O_C(1) \lra \F \lra \C_x \lra 0,
\]
where $C \subset \P^2$ is a quintic curve and $\C_x$ is the structure sheaf of a point.
\end{prop}

\begin{proof}
Assume that $\F$ gives a point in $\M(5,1)$ and satisfies $h^0(\F(-1))>0$.
As in the proof of 2.1.3 \cite{drezet-maican}, there is an injective morphism
$\O_C \to \F(-1)$ for some quintic curve $C \subset \P^2$.
We obtain a non-split extension
\[
0 \lra \O_C(1) \lra \F \lra \C_x \lra 0.
\]
Conversely, using the fact that $\O_C$ is stable, it is easy to see that any non-split
extension sheaf as above gives a point in $\M(5,1)$.

Assume now that $\F$ has a resolution as in the proposition.
Let $x$ be the point given by the equations $\ell_1 = 0$, $\ell_2 = 0$ and let $\I_x \subset \O$ be its ideal sheaf.
Let $f= \ell_1 f_2 - \ell_2 f_1$ and let $C$ be the quintic curve with equation $f=0$.
We apply the snake lemma to the commutative diagram with exact rows
\[
\xymatrix
{
0 \ar[r] & \O(-4) \ar[d]^-{f} \ar[r]^-{\scriptsize \left[ \! \ba{r} -\ell_2 \\ \ell_1 \ea \! \right]} & 2\O(-3) \ar[r] \ar[d]^-{\f}
& \I_x(-2) \ar[r] \ar[d] & 0 \\
0 \ar[r] & \O(1) \ar[r]^-{i} & \O(-2) \oplus \O(1) \ar[r]^-{p} & \O(-2) \ar[r] & 0
}.
\]
Here $i$ is the inclusion into the second factor and $p$ is the projection onto the first factor.
We deduce that $\F$ is an extension of $\C_x$ by $\O_C(1)$.
As $h^0(\F)=3$, the extension does not split.

Conversely, assume that $\F$ is a non-split extension of $\C_x$ by $\O_C(1)$.
We construct a resolution of $\F$ from the standard resolution of $\O_C(1)$ and from the resolution
\[
0 \lra \O(-4) \lra 2\O(-3) \lra \O(-2) \lra \C_x \lra 0,
\]
using the horseshoe lemma. We obtain a resolution of the form
\[
0 \lra \O(-4) \lra \O(-4) \oplus 2\O(-3) \stackrel{\f}{\lra} \O(-2) \oplus \O(1) \lra \F \lra 0.
\]
If the map $\O(-4) \to \O(-4)$ in the above resolution were zero,
then, as in the proof of \ref{2.3.2}, the extension would split.
This would be contrary to our hypothesis.
We conclude that $\O(-4)$ can be cancelled in the above exact sequence
and we arrive at the resolution from the proposition.
\end{proof}


\subsection{Description of the strata as quotients}

In subsection 3.1 we found that the moduli space $\M(5,1)$ can be decomposed into
four strata:
\begin{enumerate}
\item[$-$] an open stratum $X_0$ given by the condition $h^1(\F)=0$;
\item[$-$] a locally closed stratum $X_1$ of codimension $2$ given by the conditions \\
$h^0(\F(-1))=0$, $h^1(\F)=1$, $h^0(\F \tensor \Om^1(1))=0$;
\item[$-$] a locally closed stratum $X_2$ of codimension $3$ given by the conditions \\
$h^0(\F(-1))=0$, $h^1(\F)=1$, $h^0(\F \tensor \Om^1(1))=1$;
\item[$-$] the stratum $X_3$ of codimension $5$ given by the conditions \\
$h^0(\F(-1))=1$, $h^1(\F)=2$. We shall see below at \ref{3.2.5} that $X_3$ is closed.
\end{enumerate}
In the sequel $X_i$ will be equipped with the canonical reduced structure induced from
$\M(5,1)$. Let $W_0$, $W_1$, $W_2$, $W_3$ be the sets of morphisms $\f$ from \ref{3.1.1}, \ref{3.1.2}(i),
\ref{3.1.2}(ii), respectively \ref{3.1.5}. Each sheaf $\F$ giving a point in $X_i$ is the cokernel of a morphism $\f \in W_i$.
Let $\W_i$ be the ambient vector spaces of morphisms of sheaves containing $W_i$, e.g.
$
\W_0 = \Hom(4\O(-2),3\O(-1) \oplus \O).
$
Let $G_i$ be the natural groups of automorphisms acting by conjugation on $\W_i$.
In this subsection we shall prove that there exist geometric quotients $W_i/G_i$,
which are smooth quasiprojective varieties, such that $W_i/G_i \isom X_i$.
We shall also give concrete descriptions of these quotients.

\begin{prop}
\label{3.2.1}
There exists a geometric quotient $W_0/G_0$, which is a proper open subset
inside a fibre bundle over $\N(3,4,3)$ with fibre $\P^{14}$.
Moreover, $W_0/G_0$ is isomorphic to $X_0$.
\end{prop}

\begin{proof}
The situation is analogous to \ref{2.2.4}.
Let $\L=(\l_1,\m_1,\m_2)$ be a polarisation for the action of $G_0$ on $\W_0$ satisfying
$0 < \m_2 < 1/4$. $W_0$ is the proper open invariant subset of injective morphisms
inside $\W_0^{ss}(\L)$. Let $\N(3,4,3)$ be the moduli space of semi-stable Kronecker
modules $f \colon 4\O(-2) \to 3\O(-1)$ and let
\[
\theta \colon p_1^*(E) \tensor p_2^*(\O(-2)) \lra p_1^*(F) \tensor p_2^*(\O(-1))
\]
be the morphism of sheaves on $\N(3,4,3) \times \P^2$
induced from the universal morphism $\tau$.
Then $\U={p_1}_{*}(\Coker(\theta^*))$
is a vector bundle of rank $15$ on $\N(3,4,3)$ and $\P(\U)$ is the
geometric quotient $\W_0^{ss}(\L)/G_0$. Thus $W_0/G_0$ exists and is a proper open
subset of $\P(\U)$.

The canonical morphism $W_0/G_0 \to X_0$ is bijective and,
since $X_0$ is smooth, it is an isomorphism.
\end{proof}

\begin{prop}
\label{3.2.2}
There exists a geometric quotient $W_1/G_1$ and it is a proper open subset
inside a fibre bundle with fibre $\P^{16}$ and base the
Grassmann variety $\Grass(2,S^2V^*)$. Moreover, $W_1/G_1$ is isomorphic to $X_1$.
\end{prop}

\begin{proof}
The existence of $W_1/G_1$ follows from 9.3 \cite{drezet-trautmann}.
Let $\L=(\l_1,\l_2,\m_1)$ be a polarisation for the action of $G_1$ on $\W_1$ satisfying
$0 < \l_1 < 1/2$. Then $\W_1^{ss}(\L)$
is given by the conditions that $\f_{12}$, $\f_{22}$ be linearly independent 
two-forms and that the first column of $\f$ be not a multiple of the second column.
$W_1$ is the proper open invariant subset of injective morphisms inside $\W_1^{ss}(\L)$.
The semi-stable morphisms that are not injective are represented by matrices of the form
\[
\left[
\begin{array}{ll}
q\ell_1 & \ell \ell_1 \\
q \ell_2 & \ell \ell_2
\end{array}
\right]
\]
with $\ell \in V^*$ non-zero, $q \in S^2 V^*$ non-divisible by $\ell$ and $\ell_1, \ell_2 \in V^*$ linearly
independent.
The moduli space $\N(6,1,2)$ of semi-stable Kronecker modules $f \colon \O(-2) \to 2\O$
is isomorphic to $\Grass(2,S^2V^*)$. Let
\[
\theta \colon p_1^*(E) \tensor p_2^*(\O(-2)) \lra p_1^*(F)
\]
be the morphism of sheaves on $\N(6,1,2) \times \P^2$
induced from the universal morphism $\tau$.
Then $\U={p_1}_{*}(\Coker(\theta) \tensor p_2^*(\O(3)))$
is a vector bundle of rank $17$ over $\N(6,1,2)$
and $\P(\U)$ is the geometric quotient $\W_1^{ss}(\L)/G_1$.
Thus $W_1/G_1$ exists and is a proper open subset of the projective variety $\P(\U)$.

To show that the canonical bijective morphism $W_1/G_1 \to X_1$ is an isomorphism
we shall construct resolution \ref{3.1.2}(i) for a sheaf $\F$ giving a point in $X_1$ in a natural manner 
from the Beilinson diagram (2.2.3) \cite{drezet-maican} for $\F$, which has the form
\[
\xymatrix
{
4\O(-2) \ar[r]^-{\f_1} & 3\O(-1) \ar[r]^-{\f_2} & \O \\
0 & 0 & 2\O
}.
\]
According to 2.2 \cite{drezet-maican}, $\f_2$ is surjective, so $\Ker(\f_2) \isom \Om^1$.
Recall the morphism $\rho$ introduced in the proof of \ref{2.2.4}.
There is a morphism $\a \colon 4\O(-2) \to 3\O(-2)$ such that $\rho \circ \a = \f_1$.
As at \ref{2.2.4}, we have $\rank(\a)=3$, forcing
\[
\Ker(\f_2)= \Im(\f_1) \quad \text{and} \quad \Ker(\f_1) \isom \O(-3) \oplus \O(-2).
\]
The exact sequence (2.2.5) \cite{drezet-maican} takes the form
\[
0 \lra \Ker(\f_1) \stackrel{\f_5}{\lra} 2\O \lra \F \lra 0
\]
and gives us resolution \ref{3.1.2}(i). In this fashion we construct a local inverse to the
morphism $W_1/G_1 \to X_1$. We conclude that this is an isomorphism.
\end{proof}

\begin{prop}
\label{3.2.3}
There exists a geometric quotient $W_2/G_2$ and it is a proper open subset
inside a fibre bundle with fibre $\P^{17}$ and base $Y \times \P^2$,
where $Y$ is the Hilbert scheme of zero-dimensional subschemes of $\P^2$ of length $2$.
\end{prop}

\begin{proof}
To obtain $W_2/G_2$ we shall construct successively quotients modulo subgroups
of $G_2$, as at \ref{2.2.2} and \ref{2.2.5}.
Let $W_2' \subset \W_2$ be the locally closed subset of morphisms $\f$ satisfying
the conditions from proposition \ref{3.1.2}(ii), except injectivity.
The pairs of morphisms $(\f_{11},\f_{12})$ form an open subset
$U_1 \subset \Hom(\O(-3) \oplus \O(-2),\O(-1))$
and the morphisms $\f_{23}$ form an open subset $U_2$ inside $\Hom(\O(-1),2\O)$.
We denote $U = U_1 \times U_2$. $W_2'$ is the trivial vector bundle on $U$
with fibre $\Hom(\O(-3) \oplus \O(-2), 2\O)$. 
We represent the elements of $G_2$ by pairs of matrices
\[
(g,h) \in \Aut(\O(-3) \oplus \O(-2) \oplus \O(-1)) \times \Aut(\O(-1) \oplus 2\O),
\]
\[
g = \left[
\ba{ccc}
g_{11} & 0 & 0 \\
u_{21} & g_{22} & 0 \\
u_{31} & u_{32} & g_{33}
\ea
\right], \qquad \quad h = \left[
\ba{ccc}
h_{11} & 0 & 0 \\
v_{21} & h_{22} & h_{23} \\
v_{31} & h_{32} & h_{33}
\ea
\right].
\]
Inside $G_2$ we distinguish four subgroups: a reductive subgroup ${G_2}_{\text{red}}$
given by the conditions $u_{ij}=0$, $v_{ij}=0$, the subgroup $S$ of pairs $(g,h)$ of the form
\[
g = \left[
\ba{ccc}
a & 0 & 0 \\
0 & a & 0 \\
0 & 0 & b
\ea
\right], \qquad \quad h= \left[
\ba{ccc}
a & 0 & 0 \\
0 & b & 0 \\
0 & 0 & b
\ea
\right],
\]
with $a, b \in \C^*$, and two unitary subgroups $G_2'$ and $G_2''$.
Here $G_2'$ consists of pairs $(g,h)$ of morphisms of the form
\[
g= \left[
\ba{ccc}
1 & 0 & 0 \\
0 & 1 & 0 \\
u_{31} & u_{32} & 1
\ea
\right], \qquad \quad h = \left[
\ba{ccc}
1 & 0 & 0 \\
v_{21} & 1 & 0 \\
v_{31} & 0 & 1
\ea
\right],
\]
while $G_2''$ is given by pairs $(g,h)$, where
\[
g= \left[
\ba{ccc}
1 & 0 & 0 \\
u_{21} & 1 & 0 \\
0 & 0 & 1
\ea
\right], \qquad \quad h = \left[
\ba{ccc}
1 & 0 & 0 \\
0 & 1 & 0 \\
0 & 0 & 1
\ea
\right].
\]
Note that $G_2=G_2' G_2'' {G_2}_{\text{red}}$.
Consider the $G_2$-invariant subset $\Sigma$ of $W_2'$ of morphisms of the form
\[
\left[
\ba{ccc}
q & \ell & 0 \\
v_{21}q+\ell_1u_{31} & v_{21}\ell + \ell_1u_{32} & \ell_1 \\
v_{31}q+\ell_2u_{31} & v_{31}\ell + \ell_2u_{32} & \ell_2
\ea
\right].
\]
Note that $W_2$ is the subset of injective morphisms inside $W_2' \setminus \Sigma$,
so it is open and $G_2$-invariant. Moreover, it is a proper subset as, for instance,
the morphism represented by the matrix
\[
\left[
\ba{ccc}
X^2 -Y^2 & X\phantom{^2} & 0 \\
X Z^2 & Z^2 & Y \\
Y Z^2 & 0\phantom{^2} & X
\ea
\right]
\]
is in $W_2' \setminus \Sigma$ but is not injective.
Our aim is to construct a geometric quotient of $W_2' \setminus \Sigma$
modulo $G_2$; it will follow that $W_2/G_2$ exists and is a proper open subset of
$(W_2' \setminus \Sigma)/G_2$.

Firstly, we construct the geometric quotient $W_2'/G_2'$.
Because of the conditions on $q, \ell, \ell_1, \ell_2$, it is easy to see that
$\Sigma$ is a subbundle of $W_2'$ of rank $14$.
The quotient bundle, denoted $E'$, has rank $18$.
The quotient map $W_2' \to E'$ is a geometric quotient modulo $G_2'$.
Moreover, the canonical action of $G_2'' {G_2}_{\text{red}}$ on $U$ is $E'$-linearised
and the map $W_2' \to E'$ is $G_2'' {G_2}_{\text{red}}$-equivariant.
Let $\s'$ be the zero-section of $E'$.
The restricted map $W_2' \setminus \Sigma \to E' \setminus \s'$
is also a geometric quotient modulo $G_2'$.

Secondly, we construct a geometric quotient of $E'$ modulo $G_2''$.
The quotient for the base $U$ can be described explicitly as follows.
On $V^*$ we consider the trivial bundle with fibre $S^2 V^*$ and the subbundle
with fibre $vV^*$ at any point $v \in V^*$. The quotient bundle $Q'$ is the geometric
quotient $U_1/G_2''$ and $U/G_2'' \isom (U_1/G_2'') \times U_2$.
Clearly $U$ is a principal $G_2''$-bundle over $U/G_2''$.
According to 4.2.14 \cite{huybrechts}, $E'$ descends to a vector bundle $E$
over $U/G_2''$. The canonical map $E' \to E$ is a geometric quotient modulo $G_2''$.
The composed map $W_2' \to E' \to E$ is a geometric quotient modulo $G_2' G_2''$.
Moreover, the canonical action of ${G_2}_{\text{red}}$ on $U/G_2''$
is linearised with respect to $E$ and the map $W_2' \to E$ is ${G_2}_{\text{red}}$-equivariant.
Let $\s$ be the zero-section of $E$. The restricted map
$W_2' \setminus \Sigma \to E' \setminus \s' \to E \setminus \s$
is also a geometric quotient modulo $G_2' G_2''$.

Let $x \in U/G_2''$ be a point and let $\xi \in E_x$ be a non-zero vector lying over $x$.
The stabiliser of $x$ in ${G_2}_{\text{red}}$ is $S$ and $S \xi = \C^* \xi$.
Thus the canonical map $E \setminus \sigma \to \P(E)$ is a geometric quotient modulo $S$.
It remains to construct a geometric quotient of $\P(E)$ modulo the induced action of
${G_2}_{\text{red}}/S$.
Clearly, $(U/G_2'')/({G_2}_{\text{red}}/S)$ exists and is isomorphic to $\P(Q) \times \P^2$,
where $Q$ is the bundle on $\P(V^*)$ to which $Q'$ descends.
As noted in the proof of \ref{2.2.5}, $\P(Q)$ is the Hilbert scheme of zero-dimensional
subschemes of $\P^2$ of length $2$.
It remains to show that $\P(E)$ descends to a fibre bundle on $\P(Q) \times \P^2$.
We consider the character $\chi$ of ${G_2}_{\text{red}}$ given by
$
\chi(g,h) = \det(g) \det(h)^{-1}.
$
Note that $\chi$ is well-defined because it is trivial on homotheties.
We multiply the action of ${G_2}_{\text{red}}$ on $E$ by $\chi$ and we denote the resulting
linearised bundle by $E_{\chi}$.
The action of $S$ on $E_{\chi}$ is trivial, hence $E_{\chi}$ is ${G_2}_{\text{red}}/S$-linearised.
The isotropy subgroup in ${G_2}_{\text{red}}/S$ for any point in $U/G_2''$ is trivial,
so we can apply \cite{huybrechts}, lemma 4.2.15, to deduce that $E_{\chi}$
descends to a vector bundle $F$ over $\P(Q) \times \P^2$.
The induced map $\P(E) \to \P(F)$ is a geometric quotient map modulo ${G_2}_{\text{red}}/S$.
We conclude that the composed map
\[
W_2' \setminus \Sigma \lra E' \setminus \sigma' \lra E \setminus \sigma \lra \P(E) \lra \P(F)
\]
is a geometric quotient map modulo $G_2$
and that a geometric quotient $W_2/G_2$ exists and is a proper open subset inside $\P(F)$.
\end{proof}

\begin{prop}
\label{3.2.4}
The geometric quotient $W_2/G_2$ is isomorphic to $X_2$.
\end{prop}

\begin{proof}
We must construct resolution \ref{3.1.2}(ii) starting from the
Beilinson spectral sequence for $\F$.
We prefer to work, instead, with the sheaf $\G=\F^{\D}(1)$, which gives a point in
$\M(5,4)$. Diagram (2.2.3) \cite{drezet-maican} for $\G$ takes the form
\[
\xymatrix
{
2\O(-2) \ar[r]^-{\f_1} & \O(-1) & 0 \\
\O(-2) \ar[r]^-{\f_3} & 4\O(-1) \ar[r]^-{\f_4} & 4\O
}.
\]
Since $\G$ maps surjectively onto $\Coker(\f_1)$ and is semi-stable,
$\f_1$ cannot be zero and $\Coker(\f_1)$ cannot be isomorphic to $\O_L(-1)$ for a line $L \subset \P^2$.
Thus $\Coker(\f_1)$ is the structure sheaf of a point $x \in \P^2$ and $\Ker(\f_1) \isom \O(-3)$.
The exact sequence (2.2.5) \cite{drezet-maican} reads:
\[
0 \lra \O(-3) \stackrel{\f_5}{\lra} \Coker(\f_4) \lra \G \lra \C_x \lra 0.
\]
We see from this that $\Coker(\f_4)$ has no zero-dimensional torsion.
The exact sequence (2.2.4) \cite{drezet-maican} takes the form
\[
0 \lra \O(-2) \stackrel{\f_3}{\lra} 4\O(-1) \stackrel{\f_4}{\lra} 4\O \lra \Coker(\f_4) \lra 0.
\]
We claim that $\f_3$ is equivalent to the morphism represented by the matrix
\[
\left[
\ba{cccc}
X & Y & Z & 0
\ea
\right]^{\T}.
\]
The argument uses the fact that $\Coker(\f_4)$ has no zero-dimensional torsion
and is analogous to the proof that the vector space $H$ from \ref{2.1.4} has dimension $3$.
Now we can describe $\f_4$. We claim that $\f_4$ is equivalent to a morphism represented
by a matrix of the form
\[
\left[
\begin{array}{cccc}
-Y & \phantom{-}X & 0 & \star \\
-Z & \phantom{-}0 & X & \star \\
\phantom{-}0 & -Z & Y & \star \\
\phantom{-}0 & \phantom{-}0 & 0 & \ell
\end{array}
\right]
\]
with $\ell \in V^*$. The argument, we recall from the proof of \ref{3.1.3}, uses the fact that the map
$4\O \to \Coker(\f_4)$ is injective on global sections
and the fact that the only morphism $\O_L(1) \to \Coker(\f_4)$
for any line $L \subset \P^2$ is the zero-morphism. Indeed, such a morphism must factor
through $\f_5$ because the composed map $\O_L(1) \to \Coker(\f_4) \to \G$ is zero.
This follows from the fact that both $\O_L(1)$ and $\G$ are semi-stable and $p(\O_L(1)) > p(\G)$.

If $\ell =0$, then $\Coker(\f_4)$ would have a direct summand with Hilbert polynomial $P(t)=2t+3$.
Such a sheaf must map injectively to $\G$, because its intersection with $\O(-3)$
could only be the zero-sheaf. This contradicts the semi-stability of $\G$. Thus $\ell \neq 0$.
Let $L$ be the line with equation $\ell = 0$.
We obtain the extension
\[
0 \lra \O(1) \lra \Coker(\f_4) \lra \O_L \lra 0,
\]
which yields the resolution
\[
0 \lra \O(-1) \lra \O \oplus \O(1) \lra \Coker(\f_4) \lra 0.
\]
Write $\CC= \Coker(\f_5)$. Since $\Ext^1(\O(-3),\O(-1))=0$,
the morphism $\f_5$ lifts to a morphism $\O(-3) \to \O \oplus \O(1)$.
We obtain the resolution
\[
0 \lra \O(-3) \oplus \O(-1) \lra \O \oplus \O(1) \lra \CC \lra 0.
\]
We now apply the horseshoe lemma to the extension $\G$ of $\C_x$ by $\CC$,
to the above resolution of $\CC$ and to the standard resolution of $\C_x$ tensored with $\O(-1)$.
We obtain the resolution
\[
0 \lra \O(-3) \lra \O(-3) \oplus 2\O(-2) \oplus \O(-1) \lra \O(-1) \oplus \O \oplus \O(1) \lra \G \lra 0.
\]
The morphism $\O(-3) \to \O(-3)$ is non-zero because $\H^1(\G)$ vanishes.
We may cancel $\O(-3)$ we get the dual of resolution \ref{3.1.2}(ii).
\end{proof}

\begin{prop}
\label{3.2.5}
There exists a geometric quotient $W_3/G_3$, which is isomorphic
to the universal quintic inside $\P^2 \times \P(S^5V^*)$. Moreover, $W_3/G_3$ is
isomorphic to $X_3$, so this is a smooth closed subvariety of $\M(5,1)$.
\end{prop}

\begin{proof}
For the first part of the proposition we refer to 3.2 \cite{drezet-maican}.
Succinctly, the map of $W_3$ to the universal quintic given by
\[
\left[
\ba{cc}
\ell_1 & \ell_2 \\
f_1 & f_2
\ea
\right] \lra (x, \langle \ell_1 f_2 - \ell_2 f_1 \rangle), \quad
\text{where $x$ is the zero-set of $\ell_1$ and $\ell_2$},
\]
is a geometric quotient map. Clearly, the natural morphism $W_3/G_3 \to X_3$ is bijective.
In order to show that it is an isomorphism we need to derive resolution \ref{3.1.5}
starting from the Beilinson spectral sequence of $\F$ and performing algebraic operations
(compare 3.1.6 \cite{drezet-maican}).
By duality, we may also start with the Beilinson spectral sequence for the sheaf
$\G = \F^{\D}(1)$. Table (2.2.3) \cite{drezet-maican} for $\operatorname{E}^1(\G)$ takes the form
\[
\xymatrix
{
3\O(-2) \ar[r]^-{\f_1} & 3\O(-1) \ar[r]^-{\f_2} & \O \\
2\O(-2) \ar[r]^-{\f_3} & 6\O(-1) \ar[r]^-{\f_4} & 5\O
}.
\]
As in the proof of \ref{2.2.4}, we have $\Ker(\f_2)=\Im(\f_1)$ and
$\Ker(\f_1) \isom \O(-3)$. The exact sequence (2.2.5) \cite{drezet-maican}
\[
0 \lra \O(-3) \stackrel{\f_5}{\lra} \Coker(\f_4) \lra \G \lra 0
\]
yields the resolution
\[
0 \lra 2\O(-2) \stackrel{\eta}{\lra} \O(-3) \oplus 6\O(-1) \lra 5\O \lra \G \lra 0,
\]
\[
\eta = \left[
\ba{c}
0 \\ \f_3
\ea
\right].
\]
As in the proof of \ref{3.1.3}, we can show that any matrix equivalent to the matrix representing $\f_3$
has three linearly independent entries on each column.
It follows that, modulo elementary operations on rows and columns, $\f_3$ is represented by a matrix
having one of the following forms:
\[
\left[
\ba{cc}
0 & 0 \\
0 & 0 \\
0 & 0 \\
X & R \\
Y & S \\
Z & T
\ea
\right] \quad \text{or} \quad \left[
\ba{cc}
0 & 0 \\
0 & 0 \\
X & 0 \\
Y & R \\
Z & S \\
0 & T
\ea
\right] \quad \text{or} \quad \left[
\ba{cc}
0 & 0 \\
X & 0 \\
Y & 0 \\
Z & R \\
0 & S \\
0 & T
\ea
\right] \quad \text{or} \quad \left[
\ba{cc}
X & 0 \\
Y & 0 \\
Z & 0 \\
0 & X \\
0 & Y \\
0 & Z
\ea
\right].
\]
Here $R, S, T$ form a basis of $V^*$. As in the proof of \ref{3.1.3}, it can be shown
that the first three matrices are unfeasible. We are left with the last possibility.

By virtue of \cite{maican-duality}, lemma 3, taking duals of the locally free sheaves occuring
in the above resolution of $\G$ yields a monad with middle cohomology $\F$ of the form
\[
0 \lra 5\O(-2) \lra 6\O(-1) \oplus \O(1) \stackrel{\eta^{\T}}{\lra} 2\O \lra 0.
\]
From this we get the resolution
\[
0 \lra 5\O(-2) \lra 2\Omega^1 \oplus \O(1) \lra \F \lra 0.
\]
Combining with the standard resolution of $\Omega^1$ yields the exact sequence
\[
0 \lra 2\O(-3) \oplus 5\O(-2) \stackrel{\f}{\lra} 6\O(-2) \oplus \O(1) \lra \F \lra 0.
\]
From the semi-stability of $\F$ we see that $\rank(\f_{12})=5$, so we may cancel
$5\O(-2)$ to get the desired resolution for $\F$.
\end{proof}



\subsection{Geometric description of the strata}

Let $\F= \Coker(\f)$ be a sheaf in $X_0$ with $\f$ as in \ref{3.1.1}.
We recall that $\f_{11}$ is semi-stable as a Kronecker $V$-module.
We shall decompose $X_0$ into locally closed subsets according to the kernel of $\f_{11}$.
We have an exact sequence
\[
0 \lra \O(-d) \stackrel{\eta}{\lra} 4\O(-2) \stackrel{\f_{11} \,}{\lra} 3\O(-1) \lra \Coker(\f_{11}) \lra 0,
\]
\[
\eta= \left[ \ba{cccc} \eta_1 & \eta_2 & \eta_3 & \eta_4 \ea \right]^{\T}, \qquad \eta_i = (-1)^i \f_i/g,
\]
where $\f_i$ is the maximal minor of $\f_{11}$ obtained by deleting the $i$-th
column and $g = \gcd(\f_1,\f_2,\f_3,\f_4)$.
The maximal minors of a generic morphism $\f_{11}$ have no common factor,
i.e. $\Ker(\f_{11}) \isom \O(-5)$.
We denote by $X_{01}$ and $X_{02}$ the subsets of $X_0$ for which $\Ker(\f_{11})$
is isomorphic to $\O(-4)$, respectively to $\O(-3)$.
The case $\deg(g)=3$ is not feasible, because in this case
$\f_{11}$ is equivalent to a morphism represented by a matrix with a zero-column,
contrary to semi-stability.
As before, the superscript ${}^\D$ applied to a subset of $\M(5,1)$ will signify the corresponding
subset of $\M(5,4)$ obtained by duality.

\begin{prop}
\label{3.3.1}
The sheaves $\G$ from
$(X_0 \setminus (X_{01} \cup X_{02}))^{\D} \subset \M(5,4)$
have the form $\J_Z(3)$, where $Z \subset \P^2$ is a zero-dimensional scheme of length $6$
not contained in a conic curve, contained in a quintic curve $C$, and $\J_Z \subset \O_C$ is its
ideal sheaf.

The generic sheaves $\G$ in $X_0^{\D}$
have the form $\O_C(3)(-P_1-\ldots - P_6)$, where $C \subset \P^2$ is a smooth quintic
curve and $P_i$, $1 \le i \le 6$, are distinct points on $C$ not contained in a conic curve.
By duality, the generic sheaves $\F$ in $X_0$ have the form $\O_C(P_1+ \ldots + P_6)$.
\end{prop}

\begin{proof}
The sheaves $\G$ from $(X_0 \setminus (X_{01} \cup X_{02}))^{\D}$
are precisely the sheaves with resolution
\[
0 \lra \O(-2) \oplus 3\O(-1) \stackrel{\psi}{\lra} 4\O \lra \G \lra 0,
\]
where $\psi_{12}$ is semi-stable as a Kronecker $V$-module and its maximal minors
have no common factor. According to 4.5 and 4.6 \cite{modules-alternatives},
$\Coker(\psi_{12}) \isom \I_Z(3)$, where $Z \subset \P^2$ is a zero-dimensional scheme
of length $6$ not contained in a conic curve. Conversely, any $\I_Z(3)$ is the cokernel of some
$\psi_{12}$ with the above properties. The conclusion now follows as at \ref{2.3.4}(i).
\end{proof}

\begin{prop}
\label{3.3.2}
The sheaves $\F$ giving points in $X_{02}$ are precisely the extension sheaves
\[
0 \lra \O_{C'} \lra \F \lra \O_C \lra 0,
\]
satisfying $\H^1(\F)= 0$. Here $C'$ and $C$ are arbitrary cubic, respectively conic curves in $\P^2$.
\end{prop}

\begin{proof}
Assume that $\F$ is in $X_{02}$, i.e. $\Ker(\f_{11}) \isom \O(-3)$.
The entries of $\eta$ span $V^*$, otherwise the semi-stability of $\f_{11}$, as a Kronecker
$V$-module, would get contradicted. For instance, if
\[
\eta \sim \left[
\ba{c}
X \\ Y \\ 0 \\ 0
\ea
\right], \qquad \text{then} \qquad \f_{11} \sim \left[
\ba{cccc}
-Y & X & \star & \star \\
\phantom{-}0 & 0 & \star & \star \\
\phantom{-}0 & 0 & \star & \star
\ea
\right].
\]
\[
\text{Thus} \ \ \eta \sim \left[
\ba{c}
X \\ Y \\ Z \\ 0
\ea
\right], \quad \text{forcing} \quad \f_{11} \sim \left[
\ba{cccc}
-Y & \phantom{-}X & 0 & \star \\
-Z & \phantom{-}0 & X & \star \\
\phantom{-}0 & -Z & Y & \star
\ea
\right] = \left[
\ba{cc}
\rho & \ba{c} \star \\ \star \\ \star \ea
\ea \right].
\]
We have an exact sequence
\[
0 \lra \O(-2) \lra \Coker(\rho) \lra \Coker(\f_{11}) \lra 0
\]
hence, since $\Coker(\rho) \isom \O$, we have an isomorphism $\Coker(\f_{11}) \isom \O_C$
for a conic curve $C \subset \P^2$.
Applying the snake lemma to the exact diagram
\[
\xymatrix
{
& & & 0 \ar[d] \\
& & & \O \ar[d] \\
& 0 \ar[r] & 4\O(-2) \ar[r]^-{\f} \ar@{=}[d] & 3\O(-1) \oplus \O \ar[r] \ar[d] & \F \ar[r] & 0 \\
0 \ar[r] & \O(-3) \ar[r] & 4\O(-2) \ar[r]^-{\f_{11}} & 3\O(-1) \ar[r] \ar[d] & \O_C \ar[r] & 0 \\
& & & 0
}
\]
we get an extension as in the proposition. Conversely, assume we are given an extension
\[
0 \lra \O_{C'} \lra \F \lra \O_C \lra 0
\]
satisfying $\H^1(\F)=0$.
We shall first show that there is a resolution for $\O_C$ as in the diagram above.
Combining the exact sequences
\[
0 \lra \O(-3) \lra 3\O(-2) \stackrel{\rho}{\lra} 3\O(-1) \lra \O \lra 0
\]
and
\[
0 \lra \O(-2) \lra \O \lra \O_C \lra 0
\]
we obtain the resolution
\[
0 \lra \O(-3) \stackrel{\eta}{\lra} 4\O(-2) \stackrel{\psi}{\lra} 3\O(-1) \lra \O_C \lra 0.
\]
We need to prove that $\psi$ is semi-stable as a Kronecker $V$-module.
Since $\eta$ has three linearly independent entries, $\psi$ must have three linearly independent
maximal minors, and this rules out the cases when $\psi$ could be equivalent to a matrix having
a zero-column or a zero-submatrix of size $2 \times 2$. It remains to rule out the case
\[
\psi = \left[
\ba{cccc}
-Y & X & 0 & R \\
-Z & 0 & X & S \\
\phantom{-}0 & 0 & 0 & T
\ea
\right]. \qquad \text{Denote} \quad \xi = \left[
\ba{ccc}
-Y & X & 0 \\
-Z & 0 & X
\ea
\right]
\]
and let $L_1$, $L_2$ be the lines with equations $X=0$, respectively $T=0$.
The snake lemma applied to the exact diagram 
\[
\xymatrix
{
& & 0 \ar[d] & 0 \ar[d] \\
0 \ar[r] & \O(-3) \ar[r] \ar@{=}[d] & 3\O(-2) \ar[r]^-{\xi} \ar[d] & 2\O(-1) \ar[r] \ar[d] & \O_{L_1} \ar[r]
& 0 \\
0 \ar[r] & \O(-3) \ar[r] & 4\O(-2) \ar[r]^-{\psi} \ar[d] & 3\O(-1) \ar[r] \ar[d] & \O_C \ar[r] & 0 \\
& 0 \ar[r] & \O(-2) \ar[d] \ar[r]^-{\scriptsize T} & \O(-1) \ar[r] \ar[d] & \O_{L_2}(-1) \ar[r] & 0 \\
& & 0 & 0
}
\]
yields an extension
\[
0 \lra \O_{L_1} \lra \O_C \lra \O_{L_2}(-1) \lra 0.
\]
This gives $h^0(\O_C \tensor \Omega^1(1))=1$, which is absurd, namely
$\H^0(\O_C \tensor \Omega^1(1))$ vanishes. Thus $\psi$ is semi-stable.
We now apply the horseshoe lemma to the extension
\[
0 \lra \O_{C'} \lra \F \lra \O_C \lra 0,
\]
to the standard resolution of $\O_{C'}$ and to the resolution of $\O_C$ from above.
We obtain the exact sequence
\[
0 \lra \O(-3) \lra \O(-3) \oplus 4\O(-2) \lra 3\O(-1) \oplus \O \lra \F \lra 0.
\]
By hypothesis $\H^1(\F)$ vanishes, hence the map $\O(-3) \to \O(-3)$ is non-zero.
Cancelling $\O(-3)$ we obtain a resolution as in \ref{3.1.1}
in which $\f_{11}=\psi$ is a semi-stable Kronecker $V$-module.
We conclude that $\F$ gives a point in $X_{02}$.
\end{proof}

Let $X_{10} \subset X_1$ be the open subset given by the condition that $\f_{12}$ and $\f_{22}$
have no common linear term. We denote by $X_{11}=X_1 \setminus X_{10}$ the complement.

\begin{prop}
\label{3.3.3}
\textup{(i)} The sheaves $\F$ giving points in $X_{10}$ are precisely the sheaves $\J_Z(2)$,
where $Z \subset \P^2$ is the intersection of two conic curves without common component,
$Z$ is contained in a quintic curve $C \subset \P^2$ and
$\J_Z \subset \O_C$ is its ideal sheaf.

The generic sheaves in $X_1$ are of the form $\O_C(2)(-P_1-P_2-P_3-P_4)$, where $C\subset \P^2$
is a smooth quintic curve and $P_i$, $1 \le i \le 4$, are distinct points on $C$ in general linear
position.

\medskip

\noindent
\textup{(ii)} The sheaves $\F$ giving points in $X_{11}$ are precisely the extension sheaves
\[
0 \lra \O_L(-1) \lra \F \lra \J_x(1) \lra 0
\]
satisfying $\H^0(\F \tensor \Om^1(1))=0$. Here $L \subset \P^2$ is a line and
$\J_x \subset \O_C$ is the ideal sheaf of a point $x$ on a quartic curve $C \subset \P^2$.
\end{prop}

\begin{proof}
(i) Adopting the notations of \ref{3.1.2}(i), we notice that the restriction of $\f$ to $\O(-2)$
has cokernel $\I_Z(2)$, where $Z$ is the subscheme of length $4$ in $\P^2$
given by the equations $\f_{12}=0$, $\f_{22}=0$.
The sheaves in $X_{10}$ are precisely the cokernels of injective morphisms
$\O(-3) \to \I_Z(2)$. Let $C$ be the quintic curve defined by the inclusion
$\O(-3) \subset \I_Z(2) \subset \O(2)$. We have $\F \isom \J_Z(2)$.

\medskip

\noindent
(ii) Let us write $\f_{12} = \ell \psi_{12}$, $\f_{22} = \ell \psi_{22}$, with $\ell, \psi_{12}, \psi_{22}$
non-zero one-forms, $\psi_{12}$ and $\psi_{22}$ linearly independent.
Consider the morphism
\[
\psi \colon \O(-3) \oplus \O(-1) \lra 2\O, \qquad \psi = \left[
\ba{cc}
\f_{11} & \psi_{12} \\
\f_{21} & \psi_{22}
\ea
\right].
\]
$\Coker(\psi)$ is isomorphic to a sheaf of the form $\J_x(1)$ as in the proposition.
Conversely, any sheaf $\J_x(1)$ is the cokernel of some injective morphism $\psi$
with linearly independent entries $\psi_{12}$ and $\psi_{22}$.
Let $L$ be the line with equation $\ell = 0$. We apply the snake lemma to the diagram
with exact rows
\[
\xymatrix
{
0 \ar[r] & \O(-3) \oplus \O(-2) \ar[r]^-{\f} \ar[d]^\a & 2\O \ar[r] \ar@{=}[d] & \F \ar[r] & 0 \\
0 \ar[r] & \O(-3) \oplus \O(-1) \ar[r]^-{\psi} & 2\O \ar[r] & \J_x(1) \ar[r] & 0
},
\]
\[
\a= \left[
\ba{cc}
1 & 0 \\
0 & \ell
\ea
\right].
\]
As $\Coker(\a) \isom \O_L(-1)$, we get the extension
\[
0 \lra \O_L(-1) \lra \F \lra \J_x(1) \lra 0.
\]
Conversely, assume that $\F$ is an extension of $\J_x(1)$ by $\O_L(-1)$
satisfying the condition $\H^0(\F \tensor \Om^1(1))=0$.
Combining the resolutions for these two sheaves we get the exact sequence
\[
0 \lra \O(-3) \oplus \O(-2) \oplus \O(-1) \lra \O(-1) \oplus 2\O \lra \F \lra 0.
\]
Our cohomological condition in the hypothesis ensures that $\O(-1)$ may be cancelled,
hence we obtain a resolution as in \ref{3.1.2}(i) with $\f_{12} = \ell \psi_{12}$ and $\f_{22}= \ell \psi_{22}$.
Thus $\F$ gives a point in $X_{11}$.
\end{proof}

\begin{prop}
\label{3.3.4}
The generic sheaves from $X_2$ are precisely the non-split extension sheaves
\[
0 \lra \J_x(1) \lra \F \lra \O_Z \lra 0
\]
for which there is a global section of $\F(1)$ taking the value 1 at every point of $Z$.
Here $\J_x \subset \O_C$ is the ideal sheaf of a point $x$ on a quintic curve $C \subset \P^2$
and $Z \subset C$ is the union of two distinct points, also distinct from $x$.

There is an open subset of $X_2$ consisting of the isomorphism classes
of all sheaves of the form $\O_C(1)(-P_1+P_2+P_3)$,
where $C \subset \P^2$ is a smooth quintic curve and $P_1, P_2, P_3$ are distinct points on $C$.
In particular, $X_2$ lies in the closure of $X_1$ and $X_3$ lies in the closure of $X_2$.
\end{prop}

\begin{proof}
One direction was proven at \ref{3.1.2}(ii). Given $\F$ in $X_2$, there is an extension
as in the proposition with $x$ given by the equations $\ell_1=0$, $\ell_2=0$,
$Z$ given by the equations $q=0$, $\ell=0$ and $C$ given by the equation $\det(\f)=0$.

For the converse we apply the horseshoe lemma to the resolutions
\[
0 \lra \O(-4) \lra \I_x(1) \lra \J_x(1) \lra 0
\]
and
\[
0 \lra \O(-4) \stackrel{\zeta}{\lra} \O(-3) \oplus \O(-2) \stackrel{\xi}{\lra}
\O(-1) \stackrel{\pi}{\lra} \O_Z \lra 0,
\]
\[
\zeta = \left[
\ba{c}
-\ell \\
\phantom{-}q
\ea
\right], \qquad \quad \xi = \left[
\ba{cc}
q & \ell
\ea
\right].
\]
By hypothesis, $\pi$ lifts to a morphism $\a \colon \O(-1) \to \F$.
We define morphisms $\b, \g, \d$ as at \ref{2.3.2}.
By the reason given there, $\d$ is non-zero, namely, if $\d$ were zero, then the extension
for $\F$ would split. We arrive at the resolution
\[
0 \lra \O(-3) \oplus \O(-2) \lra \O(-1) \oplus \I_x(1) \lra \F \lra 0,
\]
which, further, yields resolution \ref{3.1.2}(ii).

Assume now that $C$ is smooth and write $x=P_1$, $Z= \{ P_2, P_3 \}$.
The only non-trivial extension sheaf of $\O_Z$ by $\J_x(1)$
is isomorphic to the sheaf $\F= \O_C(1)(-P_1+P_2+P_3)$.
We must show that $\F(1)$ has a global section
that does not vanish at $P_2$ and $P_3$.
We argue as at \ref{2.3.2}. Let $\e_2, \e_3 \colon \H^0(\O_Z) \to \C$
be the be the linear forms of evaluation at $P_2, P_3$.
Let $\d \colon \H^0(\O_Z) \to \H^1(\J_x(2))$ be the connecting homomorphism
in the long exact cohomology sequence associated to the short exact sequence
\[
0 \lra \O_C(2)(-x)  \lra \F(1) \lra \O_Z \lra 0.
\]
We must show that neither $\e_2$ nor $\e_3$ is orthogonal to $\operatorname{Ker}(\d)$.
This is equivalent to saying that neither $\e_2$ nor $\e_3$ are in the image of the dual map $\d^*$.
By Serre duality $\d^*$ is the restriction morphism
\[
\xymatrix
{
\H^0(\O_C(-2)(x) \tensor \omega_C) \egal[d] \ar[r]
& \H^0((\O_C(-2)(x) \tensor \omega_C)|_Z) \egal[d] \\
\H^0(\O_C(x)) \egal[d] & \H^0(\O_C(x)|_Z) \egal[d] \\
\H^0(\O_C) \isom \C \ar[r]^-{\scriptsize \left[ \ba{c} 1 \\ 1 \ea \right]} & \C^2 \isom \H^0(\O_C|_Z)
}.
\]
The linear forms $\e_2$ and $\e_3$ correspond to the vectors $(1,0)$ and $(0,1)$ in $\C^2$,
so they are clearly not in the image of $\d^*$.
The identity $\H^0(\O_C(x)) = \H^0(\O_C)$ follows from the fact that there is no rational
function on $C$ that has exactly one pole of multiplicity 1.
If this were the case, $C$ would have genus 0.

To see that $X_2 \subset \overline{X}_1$ choose a point in $X_2$ represented by the sheaf
$\O_C(1)(-P_1+P_2+P_3)$. We may assume that $P_1, P_2, P_3$ are non-colinear and that
the line through $P_2$ and $P_3$ intersects $C$ at five distinct points denoted $P_2, P_3, Q_1, Q_2, Q_3$.
Then $\O_C(1)(-P_1+P_2+P_3)$ is isomorphic to $\O_C(2)(-P_1-Q_1-Q_2-Q_3)$.
Clearly, we can find points $R_1, R_2, R_3$ on $C$, converging to $Q_1, Q_2, Q_3$ respectively,
such that $P_1, R_1, R_2, R_3$ are in general linear position. Thus $\O_C(2)(-P_1-R_1-R_2-R_3)$
gives a point in $X_1$ converging to the chosen point in $X_2$.

According to \ref{3.1.5}, the generic sheaves in $X_3$ have the form $\O_C(1)(P)$,
where $C \subset \P^2$ is a smooth quintic curve and $P$ is a point on $C$.
Choose distinct points $P_1, P_2$ on $C$, which are
also distinct from $P$, such that $P_2$ converges to $P_1$.
The stable-equivalence class of $\O_C(1)(-P_1+P_2+P)$
is in $X_2$ and converges to the stable-equivalence class of $\O_C(1)(P)$.
We conclude that $X_3 \subset \overline{X}_2$.
\end{proof}

The following result will be helpful in the discussion about sheaves from $X_{01}$,
which we have left for the end.

\begin{prop}
\label{3.3.5}
Let $\psi \colon 4\O(-2) \to 3\O(-1)$ be a Kronecker $V$-module.
Let $\psi_i$, $1 \le i \le 4$, denote the maximal minor of $\psi$
obtained by deleting the $i$-th column.
Assume that the minors $\psi_i$ have a common linear factor.
Then $\Ker(\psi) \isom \O(-4)$ and $\psi$ is semi-stable if and only if
$\psi_i$, $1 \le i \le 4$, are linearly independent three-forms.
\end{prop}

\begin{proof}
Assume that $\Ker(\psi) \isom \O(-4)$ and that $\psi$ is semi-stable.
We argue by contradiction. If the maximal minors of $\psi$ were linearly dependent, then,
performing possibly column operations on $\psi$,
we could assume that one of them is zero, say $\psi_4=0$.
Let $\psi'$ be the matrix obtained from $\psi$ by deleting the fourth column.
It is easy to see that $\psi'$ is semi-stable as a Kronecker $V$-module.
It follows that $\psi'$ is equivalent to the morphism represented by the matrix
\[
\left[
\ba{ccc}
-Y & \phantom{-}X & 0 \\
-Z & \phantom{-}0 & X \\
\phantom{-}0 & -Z & Y
\ea
\right]. \qquad \text{Thus the vector} \qquad \left[
\ba{c}
X \\ Y \\ Z \\ 0
\ea
\right]
\]
is in the kernel of $\psi$. This contradicts our hypothesis that $\Ker(\psi)$ be isomorphic
to $\O(-4)$.

Conversely, assume that $\psi_i$, $1 \le i \le 4$, are linearly independent.
Then they cannot have a common factor of degree 2, that is,
in view of the comments at the beginning of this subsection, we have $\Ker(\psi) \isom \O(-4)$.
The semi-stability of $\psi$ is also clear: if $\psi$ were equivalent to a matrix having a zero-column,
then the $\psi_i$ would span a vector space of dimension at most $1$. If $\psi$ were equivalent to a matrix
having a zero-submatrix of size $2 \times 2$, then the $\psi_i$ would span a vector space of dimension
at most two. If $\psi$ were equivalent to a matrix having a zero-submatrix of size $1 \times 3$, then
the $\psi_i$ would span a vector space of dimension at most $3$.
\end{proof}

\begin{prop}
\label{3.3.6}
The sheaves $\F$ giving points in $X_{01}$ occur as non-split extension sheaves
of one of the following three kinds:
\begin{align*}
\tag{i}
0 \lra \G \lra \F \lra \O_L \lra 0,
\end{align*}
where $\H^1(\F) = 0$.
Here $L \subset \P^2$ is a line and $\G$ is in the exceptional divisor of $\M(4,0)$.
For fixed $L$ and $\G$ the feasible extension sheaves form
a locally closed subset of $\P(\Ext^1(\O_L,\G))$.
\begin{align*}
\tag{ii}
0 \lra \E \lra \F \lra \O_Z \lra 0.
\end{align*}
Here $Z \subset \P^2$ is a zero-dimensional scheme of length $3$ not contained in a line
and $\E$ is a sheaf in $\M(5,-2)$ such that $\E(1)$ belongs to the stratum $X_3$ of $\M(5,3)$.
\begin{align*}
\tag{iii}
0 \lra \O_L(-1) \lra \F \lra \J_Z(1)^{\D} \lra 0.
\end{align*}
Here $L \subset \P^2$ is a line and $Z \subset \P^2$ is a zero-dimensional scheme
of length $3$ not contained in a line, contained in a quartic curve $C \subset \P^2$,
and $\J_Z \subset \O_C$ is its ideal sheaf.
For fixed $\J_Z$ and $L$ the feasible extension sheaves form a locally
closed subset of $\P(\Ext^1(\J_Z(1)^{\D},\O_L(1)))$.
\end{prop}

\begin{proof}
Let $\F$ give a point in $X_{01}$. Recall resolution \ref{3.1.1}.
We have the isomorphism $\Ker(\f_{11}) \isom \O(-4)$ and we denote $\CC = \Coker(\f_{11})$.
We have $P_{\CC}(t)=t+3$, so this sheaf is the direct sum
of a zero-dimensional sheaf and $\O_L(d)$ for a line $L \subset \P^2$ and an integer $d$.
It is thus clear that $\CC$ has a subsheaf $\CC'$ with Hilbert polynomial $P_{\CC'}(t)=t+2$.

Applying the snake lemma to a diagram similar to the first diagram in the proof of \ref{3.3.2}
we obtain an extension
\[
0 \lra \O_C \lra \F \lra \CC \lra 0,
\]
where $C \subset \P^2$ is a quartic curve. Let $\F' \subset \F$ be the preimage of $\CC'$.
We have $P_{\F'}(t)=5t$ and it is easy to see that $\F'$ is semi-stable.
We now use the possible resolutions for sheaves in $\M(5,0)$ found in section 4, which we
obtain independently of any result in this subsection.
Taking into account that $\H^0(\F' \tensor \Om^1(1)) =0$ leaves only two possible resolutions,
the ones at \ref{4.1.2} and \ref{4.1.3}. The first resolution must fit into a commutative diagram
\[
\xymatrix
{
0 \ar[r] & 5\O(-2) \ar[r]^-{\psi} \ar[d]^-{\b} & 5\O(-1) \ar[r] \ar[d]^-{\a} & \F' \ar[r] \ar[d] & 0 \\
0 \ar[r] & 4\O(-2) \ar[r]^-{\f} & 3\O(-1) \oplus \O \ar[r] & \F \ar[r] & 0
}.
\]
Since $\a(1)$ is injective on global sections, we have one of the following two possibilities:
\[
\a \sim \left[
\ba{ccccc}
0 & 0 & 0 & 0 & 0 \\
1& 0 & 0 & 0 & 0 \\
0 & 1 & 0 & 0 & 0 \\
0 & 0 & X & Y & Z
\ea
\right] \qquad \text{or} \qquad \a \sim \left[
\ba{ccccc}
1& 0 & 0 & 0 & 0 \\
0 & 1 & 0 & 0 & 0 \\
0 & 0 & 1 & 0 & 0 \\
0 & 0 & 0 & X & Y
\ea
\right].
\]
In the first case $\Ker(\a)$ is isomorphic to $\Om^1$, so the latter is isomorphic
to a direct sum of copies of $\O(-2)$. This is absurd. 
In the second case we have $\Ker(\b) \isom \O(-2)$,
hence, without loss of generality, we may assume that $\b$ is the projection onto the first
four terms. From the commutativity of the diagram we get
\[
\psi = \left[
\ba{cc}
\f_{11} & \ba{c} \phantom{-}0 \\ \phantom{-}0 \\ \phantom{-}0 \ea \\
\ba{cccc}
\star & \star & \star & \star \\
\star & \star & \star & \star
\ea
& \ba{c} -Y \\ \phantom{-}X \ea
\ea
\right].
\]
This shows that $\F'$ maps surjectively onto the cokernel of $\f_{11}$. But this is impossible because,
by construction, the image of $\F'$ in $\CC$ is the proper subsheaf $\CC'$.
Thus far we have shown that resolution \ref{4.1.2} for $\F'$ is unfeasible. It remains to examine
resolution \ref{4.1.3}. This fits into a commutative diagram
\[
\xymatrix
{
0 \ar[r] & \O(-3) \oplus 2\O(-2) \ar[r]^-{\psi} \ar[d]^-{\b} & 2\O(-1) \oplus \O \ar[r] \ar[d]^-{\a}
& \F' \ar[r] \ar[d] & 0 \\
0 \ar[r] & 4\O(-2) \ar[r]^-{\f} & 3\O(-1) \oplus \O \ar[r] & \F \ar[r] & 0
}.
\]
Since $\a$ and $\a(1)$ are injective on global sections, we see that $\a$ and $\b$ are injective
and we may write
\[
\b = \left[
\ba{ccc}
-\ell_2 & 0 & 0 \\
\phantom{-}\ell_1 & 0 & 0 \\
\phantom{-}0 & 1 & 0 \\
\phantom{-}0 & 0 & 1
\ea
\right], \qquad \qquad \a = \left[
\ba{ccc}
0 & 0 & 0 \\
1 & 0 & 0 \\
0 & 1 & 0 \\
0 & 0 & 1
\ea
\right].
\]
From the commutativity of the diagram and the semi-stability of $\f_{11}$, we see that $\ell_1$ and $\ell_2$
are linearly independent one-forms and
\[
\f_{11}= \left[
\ba{cc}
\ba{cc} \ell_1 & \ell_2 \ea & \ba{cc} 0 & 0 \ea \\
\ba{cc} \star & \star \\ \star & \star \ea & \xi
\ea
\right].
\]
We recall that the greatest common divisor of the maximal minors of $\f_{11}$ is a linear form $g$.
Since $g$ divides both $\ell_1 \det(\xi)$ and $\ell_2 \det(\xi)$, we see that $g$ divides
$\det(\xi)$, hence $\xi$ is equivalent to a matrix having a zero-entry. Thus we may write
\[
\f_{11}= \left[
\ba{cccc}
\ell_1 & \ell_2 & 0 & 0 \\
\star & \star & \xi_3 & 0 \\
\star & \star & \star & \xi_4
\ea
\right] = \left[
\ba{cc}
\z & \ba{c} 0 \\ 0 \ea \\
\ba{ccc} \star & \star & \star \ea & \xi_4
\ea
\right].
\]
It is clear that $\z$ is semi-stable as a Kronecker $V$-module.
Assume that the maximal minors of $\z$ have a common linear factor, say $Z$.
We may then write
\[
\f = \left[
\ba{cccc}
X & Z & 0 & 0 \\
Y & 0 & Z & 0 \\
\star & \star & \star & S \\
\star & \star & \star & T
\ea
\right] = \left[
\ba{cc}
\f' & \ba{c} 0 \\ 0 \\ S \ea \\
\ba{ccc} \star & \star & \star \ea & T
\ea
\right].
\]
Notice that $g$ is a multiple of $Z$, $S$ is non-zero and does not divide $\det(\f')/Z$.
We have $\Coker(\z) \isom \O_L$, where $L \subset \P^2$ is the line with equation $Z = 0$.
We apply the snake lemma to the exact diagram from below
\begin{table}[ht]{Diagram for the snake lemma.}
\[
\xymatrix
{
& & 0 \ar[d] & 0 \ar[d] \\
& 0 \ar[r] & \O(-2) \ar[d] \ar[r]^-{\scriptsize \left[ \ba{c} S \\ T \ea \right]} & \O(-1) \oplus \O \ar[d] \\
& 0 \ar[r] & 4\O(-2) \ar[r]^-{\f} \ar[d] & 3\O(-1) \oplus \O \ar[r] \ar[d] & \F \ar[r] & 0 \\
0 \ar[r] & \O(-3) \ar[r] & 3\O(-2) \ar[r]^-{\z} \ar[d] & 2\O(-1) \ar[r] \ar[d] & \O_L \ar[r] & 0 \\
& & 0 & 0
}
\]
\end{table}
in order to obtain a non-split extension of the form
\[
0 \lra \G \lra \F \lra \O_L \lra 0,
\]
where $\G$ has resolution
\[
0 \lra \O(-3) \oplus \O(-2) \stackrel{\psi}{\lra} \O(-1) \oplus \O \lra \G \lra 0,
\]
with $\psi_{12}=S$ different from zero.
From 5.2.1 \cite{drezet-maican} we see that $\G$ is in the exceptional divisor of $\M(4,0)$.
Conversely, any $\G$ of $\M(4,0)$, which is in the exceptional divisor, i.e. satisfying the condition $h^0(\G)=1$,
occurs as the cokernel of a morphism $\psi$ as above with $\psi_{12} \neq 0$.
Assume now that $\F$ is an extension of $\O_L$ with a sheaf $\G$ as above, satisfying $\H^1(\F)=0$.
Choose an equation $Z = 0$ for $L$. We combine the resolution of $\G$ with the resolution
\[
0 \lra \O(-3) \lra 3\O(-2) \stackrel{\z}{\lra} 2\O(-1) \lra \O_L \lra 0
\]
and we obtain a resolution
\[
0 \lra \O(-3) \lra \O(-3) \oplus 4\O(-2) \lra 3\O(-1) \oplus \O \lra \F \lra 0.
\]
The morphism $\O(-3) \to \O(-3)$ in the above complex is non-zero
because, by hypothesis, $\H^1(\F)$ vanishes.
Thus we may cancel $\O(-3)$ to get resolution \ref{3.1.1} with
\[
\f = \left[
\ba{cc}
\z & \ba{c} 0 \\ 0 \ea \\
\ba{ccc} \star & \star & \star \\
\star & \star & \star \ea & \ba{c} \psi_{12} \\ \psi_{22} \ea
\ea
\right] = \left[
\ba{cc}
\f' & \ba{c} 0 \\ 0 \\ \psi_{12} \ea \\
\ba{ccc} \star & \star & \star \ea & \psi_{22}
\ea
\right].
\]
In view of \ref{3.3.5}, the condition that $\F$ be in $X_{01}$ is equivalent to saying that
$\det(\f')/Z$, $\psi_{12}X, \psi_{12}Y, \psi_{12}Z$ are linearly independent two-forms.
This defines an open subset inside the closed set of extension sheaves of $\O_L$ by $\G$
with vanishing first cohomology.

It remains to examine the case when the maximal minors of $\z$ have no common factor.
Then $g$ is a multiple of $\xi_4$. We have $\Ker(\z) \isom \O(-4)$.
According to 4.5 and 4.6 \cite{modules-alternatives},
the cokernel of $\z$ is isomorphic to the structure sheaf of a zero-dimensional
scheme $Z$ of length $3$ not contained in a line. Write as above
\[
\f = \left[
\ba{cc}
\z & \ba{c} 0 \\ 0 \ea \\
\ba{ccc} \star & \star & \star \\
\star & \star & \star \ea & \ba{c} S \\ T \ea
\ea
\right] 
\]
and note that the snake lemma gives an extension
\[
0 \lra \E \lra \F \lra \O_Z \lra 0,
\]
where $\E$ has a resolution
\[
0 \lra \O(-4) \oplus \O(-2) \stackrel{\psi}{\lra} \O(-1) \oplus \O \lra \E \lra 0
\]
in which $\psi_{12} =S$ and $\psi_{22}=T$.
We have $P_{\E}(t)=5t-2$. According to \ref{2.1.4}, $\E$ is in $\M(5,-2)$ precisely if $S$ does not
divide $T$. In that case $\E(1)$ gives a point in the stratum $X_3$ of $\M(5,3)$.
Finally, assume that $S$ divides $T$. We have a non-split extension of sheaves
\[
0 \lra \O_L(-1) \lra \F \lra \S \lra 0,
\]
where $L \subset \P^2$ is given by the equation $S=0$ and $\S$ has a resolution of the form
\[
0 \lra 3\O(-2) \stackrel{\psi}{\lra} 2\O(-1) \oplus \O \lra \S \lra 0,
\]
where $\psi_{11}=\z$.
According to 3.3.2 \cite{drezet-maican}, 
the subset of $\M(4,3)$ of sheaves of the form $\S^{\D}(1)$
is an open subset consisting of all sheaves
of the form $\J_Z(2)$, where $Z \subset \P^2$ is a zero-dimensional scheme of length $3$
not contained in a line, contained in a quartic curve $C \subset \P^2$ and $\J_Z \subset \O_C$
is its ideal sheaf. 
Assume we are given $\S$ as above, $L \subset \P^2$ a line with equation $S=0$
and $\F$ a non-split extension of $\S$ by $\O_L(-1)$.
We combine the above resolution for $\S$ with the standard resolution of $\O_L(-1)$
to get a resolution for $\F$ as in \ref{3.1.1}.
By \ref{3.3.5}, the condition that $\F$ be in $X_{01}$ is equivalent to
saying that $S$ divides $\det(\f')$ and $\det(\f')/S$ together with the
maximal minors of $\z$ form a linearly independent set in $S^2 V^*$.
These  conditions define a locally closed subset of $\P(\Ext^1(\S,\O_L(-1)))$.
\end{proof}

\noindent
From what was said above we can summarise:

\begin{prop}
\label{3.3.7}
$\{ X_0, X_1, X_2, X_3 \}$ represents a stratification
of $\M(5,1)$ by locally closed subvarieties of codimensions $0, 2, 3, 5$.
\end{prop}



\section{Euler characteristic zero}


\subsection{Locally free resolutions for semi-stable sheaves}

\begin{prop}
\label{4.1.1}
Every sheaf $\F$ giving a point in $\M(5,0)$
and satisfying the condition $h^0(\F(-1)) > 0$ is of the
form $\O_C(1)$ for a quintic curve $C \subset \P^2$.
\end{prop}

\begin{proof}
Consider a non-zero morphism $\O \to \F(-1)$.
As in the proof of 2.1.3 \cite{drezet-maican}, it must factor through an injective
morphism $\O_C \to \F(-1)$. Here $C \subset \P^2$ is a curve;
its degree must be 5, otherwise $\O_C$ would destabilise $\F(-1)$.
As both $\O_C$ and $\F(-1)$ have the same Hilbert polynomial,
the injective morphism from above must be an isomorphism.

The converse follows from the general fact that the structure sheaf of a curve in $\P^2$ is stable.
\end{proof}

\begin{prop}
\label{4.1.2}
The sheaves $\F$ giving points in $\M(5,0)$ and satisfying the condition $h^1(\F)=0$
are precisely the sheaves with resolution
\[
0 \lra 5\O(-2) \stackrel{\f}{\lra} 5\O(-1) \lra \F \lra 0.
\]
Moreover, such a sheaf $\F$ is properly semi-stable if and only if $\f$ is equivalent to a
morphism of the form
\[
\left[
\begin{array}{cc}
\star & \psi \\
\star & 0
\end{array}
\right] \quad \text{for some} \quad \psi \colon m\O(-2) \lra m\O(-1), \quad 1 \le m \le 4.
\]
\end{prop}

\begin{proof}
Assume that $\F$ gives a point in $\M(5,0)$ and its first cohomology vanishes.
For a suitable line $L \subset \P^2$ we have an exact sequence
\[
0 \lra \F \lra \F(1) \lra \F(1)|_{L} \lra 0.
\]
The associated long cohomology sequence shows that $\H^1(\F(1))$ vanishes, too.
The same argument applied to the exact sequence
\[
0 \lra \F(-1) \lra \F \tensor V \lra \F \tensor \Om^1(2) \lra 0
\]
shows that $\H^1(\F \tensor \Om^1(2))=0$.
The Beilinson free monad (2.2.1) \cite{drezet-maican} for $\F(1)$ gives the resolution
\[
0 \lra 5\O(-1) \lra  5\O \lra \F(1) \lra 0.
\]
Conversely, assume that $\F$ is the cokernel of a morphism $\f$ as in the proposition.
Trivially, $\F$ has no zero-dimensional torsion, because it has a locally free resolution
of length $1$. For any subsheaf $\F' \subset \F$ we have $\H^0(\F')=0$ because the
corresponding cohomology group for $\F$ vanishes. We get $\chi(\F') \le 0$,
hence $p(\F') \le 0 = p(\F)$ and we conclude that $\F$ is semi-stable.

To finish the proof we must show that for properly semi-stable sheaves $\F$ the morphism
$\f$ has the special form given in the proposition. Consider a proper subsheaf $\F' \subset \F$
which gives a point in $\M(m,0)$, $1 \le m \le 4$.
$\H^0(\F')$ vanishes, hence also $\H^1(\F')$ vanishes and, repeating the above steps with
$\F'$ instead of $\F$, we arrive at a resolution
\[
0 \lra m\O(-2) \stackrel{\psi}{\lra} m\O(-1) \lra \F' \lra 0.
\]
This fits into a commutative diagram
\[
\xymatrix
{
0 \ar[r] & m\O(-2) \ar[r]^-{\psi} \ar[d]^-{\b} & m\O(-1) \ar[r] \ar[d]^-{\a} & \F' \ar[r] \ar[d] & 0 \\
0 \ar[r] & 5\O(-2) \ar[r]^-{\f} & 5\O(-1) \ar[r] & \F \ar[r] & 0
}.
\]
Since $\a(1)$ is injective on global sections we see that $\a$, hence also $\b$, are injective.
Thus $\f$ has the required special form.
\end{proof}

\begin{prop}
\label{4.1.3}
The sheaves $\F$ giving points in $\M(5,0)$
and satisfying the cohomological conditions $h^0(\F(-1))=0$ and $h^1(\F)=1$
are precisely the sheaves with resolution
\[
0 \lra \O(-3) \oplus 2\O(-2) \stackrel{\f}{\lra} 2\O(-1) \oplus \O \lra \F \lra 0,
\]
where $\f_{12} \colon 2\O(-2) \to 2\O(-1)$ is an injective morphism.
\end{prop}

\begin{proof}
The Beilinson free monad (2.2.1) \cite{drezet-maican} for $\F$ reads as follows:
\[
0 \lra 5\O(-2) \oplus m\O(-1) \lra (m+5)\O(-1) \oplus \O \lra \O \lra 0.
\]
From this we obtain the exact sequences
\[
0 \lra 5\O(-2) \oplus m\O(-1) \lra \Om^1 \oplus (m+2)\O(-1) \oplus \O \lra \F \lra 0
\]
and
\[
0 \! \lra \O(-3) \oplus 5\O(-2) \oplus m\O(-1) \stackrel{\f}{\lra} 3\O(-2) \oplus (m+2)\O(-1) \oplus \O
\lra \F \lra \! 0,
\]
with $\f_{13}=0$, $\f_{23}=0$. As in the proof of \ref{2.1.4}, we see that $\rank(\f_{12})=3$, so we
may cancel $3\O(-2)$ to get the resolution
\[
0 \lra \O(-3) \oplus 2\O(-2) \oplus m\O(-1) \stackrel{\f}{\lra} (m+2)\O(-1) \oplus \O \lra \F
\lra 0,
\]
with $\f_{13}=0$.
By the injectivity of $\f$ we must have $m \le 1$. If $m=1$, then $\F$ has a subsheaf
of the form $\O_L$, for a line $L \subset \P^2$, contrary to semi-stability.
We conclude that $m=0$ and we obtain a resolution as in the proposition.
If $\f_{12}$ were not injective, then $\f_{12}$ would be equivalent to a morphism represented
by  a matrix with a zero-row or a zero-column.
Thus $\F$ would have a destabilising subsheaf of the form $\O_C$
or a destabilising quotient sheaf of the form $\O_C(-1)$ for a conic curve $C \subset \P^2$.

Conversely, we assume that $\F$ has a resolution as in the proposition and we need to show that there are
no destabilising subsheaves $\E$. Such a subsheaf must satisfy
$h^0(\E)=1$, $h^1(\E)=0$, $P_{\E}(t)=m t+1$, $1 \le t \le 4$.
Moreover, $\H^0(\E(-1))$ and $\H^0(\E \tensor \Om^1(1))$ vanish because the corresponding
cohomology groups for $\F$ vanish.
We can now write the Beilinson free monad for $\E$.
We get a resolution that fits into a commutative diagram
\[
\xymatrix
{
0 \ar[r] & (m-1)\O(-2) \ar[r]^-{\psi} \ar[d]^-{\b} & (m-2)\O(-1) \oplus \O \ar[r] \ar[d]^-{\a} & \E \ar[r]
\ar[d] & 0 \\
0 \ar[r] & \O(-3) \oplus 2\O(-2) \ar[r]^-{\f} & 2\O(-1) \oplus \O \ar[r] & \F \ar[r] & 0
}.
\]
Since $\a$ and $\a(1)$ are injective on global sections, we see that $\a$ is injective,
forcing $\b$ to be injective, too.
Thus $m=2$ or $m=3$. In both cases $\f_{12}$ fails to be injective,
contradicting our hypothesis.
\end{proof}

\begin{prop}
\label{4.1.4}
The sheaves $\F$ giving points in $\M(5,0)$ and
satisfying the cohomological conditions $h^0(\F(-1))=0$ and $h^1(\F)=2$
are precisely the sheaves with resolution
\[
0 \lra 2\O(-3) \oplus \O(-1) \stackrel{\f}{\lra} \O(-2) \oplus 2\O \lra \F \lra 0
\]
such that $\f_{11}$ has linearly independent entries and,
likewise, $\f_{22}$ has linearly independent entries.
\end{prop}

\begin{proof}
Let $\F$ give a point in $\M(5,0)$ and satisfy the conditions from the proposition.
The Beilinson free monad for $\F$ reads
\[
0 \lra 5\O(-2) \oplus m\O(-1) \lra (m+5)\O(-1) \oplus 2\O \lra 2\O \lra 0.
\]
Dualising and tensoring with $\O(1)$ we get the following resolution for the sheaf
$\G = \F^{\D}(1)$, which gives a point in $\M(5,5)$:
\[
0 \lra 2\O(-2) \stackrel{\eta}{\lra} 2\O(-2) \oplus (m+5)\O(-1) \stackrel{\f}{\lra}
m\O(-1) \oplus 5\O \lra \G \lra 0,
\]
\[
\eta = \left[
\ba{c}
0 \\ \psi
\ea
\right].
\]
Here $\f_{12}=0$. As $\G$ has rank zero and maps surjectively onto $\CC=\Coker(\f_{11})$,
we see that $m \le 2$. If $m=2$, then $\f_{11}$ must be injective, otherwise
$\CC$ will have positive rank. We get $P_{\CC}(t)=2t$, hence $\CC$ destabilises $\G$.
The case $m=0$ can be eliminated as in the proof of \ref{3.1.3}.
Thus $m=1$. As in the proof of \ref{3.2.5}, we may assume that $\psi$ is represented by the matrix
\[
\left[
\ba{cccccc}
X & Y & Z & 0 & 0 & 0 \\
0 & 0 & 0 & X & Y & Z
\ea
\right]^{\T}.
\]
From the Beilinson monad for $\F$ we obtain the resolution
\[
0 \lra 5\O(-2) \oplus \O(-1) \lra 2\Om^1 \oplus 2\O \lra \F \lra 0,
\]
which, combined with the standard resolution for $\Om^1$, yields the exact sequence
\[
0 \lra 2\O(-3) \oplus 5\O(-2) \oplus \O(-1) \stackrel{\f}{\lra} 6\O(-2) \oplus 2\O \lra \F \lra 0.
\]
Note that $\F$ maps surjectively onto $\Coker(\f_{11},\f_{12})$,
so this sheaf is supported on a curve, forcing 
$\rank(\f_{12}) \ge 4$. If $\rank(\f_{12}) = 4$, then $\Coker(\f_{11},\f_{12})$
would have Hilbert polynomial $P(t)=2t-2$, so it would destabilise $\F$.
We deduce that $\rank(\f_{12})=5$,
so we may cancel $5\O(-2)$ to get a resolution as in the proposition.
If the entries of $\f_{11}$ were linearly dependent, then $\F$ would have a destabilising quotient sheaf
of the form $\O_L(-2)$ for a line $L \subset \P^2$.
If the entries of $\f_{22}$ were linearly dependent, then $\F$ would have a destabilising subsheaf of the
form $\O_L$.

Conversely, we assume that $\F$ has a resolution as in the proposition and we need to show that
there is no destabilising subsheaf. Let $\F' \subset \F$ be a non-zero subsheaf of multiplicity
at most 4. We shall use the extension
\[
0 \lra \J_x(1) \lra \F \lra \C_z \lra 0
\]
from \ref{4.3.2}. Denote by $\CC'$ the image of $\F'$ in $\C_z$ and put $\K= \F' \cap \J_x(1)$.
Let $\A$ and $\O_S$ be as in the proof of \ref{3.1.2}. Recall that $S$ is a curve of degree $d \le 4$.
We can estimate the slope of $\F'$ as in the proof of loc.cit. and we get
\[
p(\F') = -\frac{d}{2} + \frac{h^0(\CC') - h^0(\A/\K)}{5-d} \le -\frac{d}{2} + \frac{1}{5-d} < 0 = p(\F).
\]
We conclude that $\F$ is semi-stable.
\end{proof}

Let $X_i$, $i=0, 1, 2, 3$, be the subset of $\M(5,0)$ of stable-equivalence classes
of sheaves $\F$ as in \ref{4.1.2}, \ref{4.1.3}, \ref{4.1.4}, respectively \ref{4.1.1}.

\begin{prop}
\label{4.1.5}
The subsets $X_0, X_1, X_2, X_3$ are disjoint.
The subset of $\M(5,0)$ of stable-equivalence classes of properly semi-stable sheaves
is included in $X_0 \cup X_1$.
\end{prop}

\begin{proof}
Let $\F$ be a properly semi-stable sheaf in $\M(5,0)$. We have an exact sequence
\[
0 \lra \F' \lra \F \lra \F'' \lra 0,
\]
with $\F'$ giving a point in $\M(r,0)$, $\F''$ giving a point in $\M(s,0)$, $r+s=5$.
From the description of $\M(r,0)$, $1 \le r \le 4$, found in \cite{drezet-maican},
we have the relations
\[
h^0(\F') = 0 \quad \text{if $r=1, 2$}, \qquad h^0(\F') \le 1 \quad \text{if $r= 3, 4$}.
\]
In all possible situations we get $h^0(\F) \le 1$, hence the stable-equivalence class of $\F$
is in $X_0 \cup X_1$.
Thus all sheaves in $X_2$ and $X_3$ are stable, so $X_2$ is disjoint from the other $X_i$
and the same is true for $X_3$.
It remains to show that $X_0$ and $X_1$ are disjoint.
Let $\F$ be a properly semi-stable sheaf as in \ref{4.1.2} and let $\G$ be a sheaf in the same
class of stable-equivalence as $\F$.
Let $\F'$ be one of the terms of a Jordan-H\"older filtration of $\F$.
From the proof of \ref{4.1.2} it transpires that $\F'$ has resolution
\[
0 \lra m\O(-2) \lra m\O(-1) \lra \F \hspace{0.12pt} ' \lra 0,
\]
for some integer $1 \le m \le 4$.
Thus $\H^0(\F')=0$.
Any term of a Jordan-H\"older filtration of $\G$ is also a term of a Jordan-H\"older filtration
of $\F$, hence its group of global sections vanishes.
We deduce that $\H^0(\G)=0$. Thus $\F$ cannot give a point in $X_1$.
\end{proof}

\begin{prop}
\label{4.1.6}
There are no sheaves $\F$ giving points in $\M(5,0)$
and satisfying the cohomological conditions $h^0(\F(-1))=0$ and $h^1(\F) \ge 3$.
\end{prop}

\begin{proof}
In view of \ref{4.1.5}, we may restrict our attention to stable sheaves $\F$ in $\M(5,0)$.
Suppose that $\F$ satisfies $h^0(\F(-1))=0$ and $h^1(\F) \neq  0$.
Consider a non-zero morphism $\O \to \F$.
As in the proof of 2.1.3 \cite{drezet-maican}, this must factor through an injective morphism
$\O_C \to \F$, where $C \subset \P^2$ is a curve.
From the stability of $\F$ we see that $C$ can only have degree 4 or 5.

Assume that $C$ has degree 5. The quotient sheaf $\CC=\F/\O_C$ is supported on finitely many points and has
length $5$. Take a subsheaf $\CC' \subset \CC$ of length $4$, and let $\F'$ be its preimage in $\F$.
We get an exact sequence
\[
0 \lra \F' \lra \F \lra \C_x \lra 0,
\]
where $\C_x$ is the structure sheaf of a point.
Any destabilising subsheaf of $\F'$ would ruin the stability of $\F$, hence $\F'$ is in $\M(5,-1)$.
From subsection 3.1 we know that $h^0(\F') \le 2$, hence $h^0(\F) \le 2$ unless $h^0(\F')= 2$ and
the morphism $\F \to \C_x$ is surjective on global sections.
In this case we can apply the horseshoe lemma to the above extension, to the standard resolution
of $\C_x$ and to the resolution
\[
0 \lra \O(-4) \oplus \O(-1) \lra 2\O \lra \F' \lra 0.
\]
We obtain a resolution
\[
0 \lra \O(-2) \lra \O(-4) \oplus 3\O(-1) \stackrel{\f}{\lra} 3\O \lra \F \lra 0,
\]
which yields an exact sequence
\[
0 \lra \O(-4) \lra \Coker(\f_{12}) \lra \F \lra 0.
\]
We claim that the morphism $\O(-2) \to 3\O(-1)$ in the above resolution
is equivalent to the morphism represented by the matrix
\[
\left[
\ba{c}
X \\ Y \\ Z
\ea
\right].
\]
The argument uses the fact that $\F$ has no zero-dimensional torsion
and is analogous to the proof that the vector space $H$ at \ref{2.1.4} has dimension $3$.
We can now describe $\f_{12}$. We claim that $\f_{12}$ is equivalent
to the morphism represented by the matrix
\[
\left[
\ba{ccc}
-Y & \phantom{-}X & 0  \\
-Z & \phantom{-}0 & X \\
\phantom{-}0 & -Y & Z
\ea
\right].
\]
The argument, we recall from the proof of \ref{3.1.3}, uses the fact that the map
$3\O \to \F$ is injective on global sections and the fact that the only morphism
$\O_L(1) \to \F$ for any line $L \subset \P^2$ is the zero-morphism.
We deduce that $\Coker(\f_{12})$ is isomorphic to $\O(1)$.
We obtain $h^0(\F(-1)) =1$, contradicting our hypothesis.

Assume now that $C$ has degree 4.
The zero-dimensional torsion $\CC'$ of the quotient sheaf $\CC=\F/\O_C$
has length at most $1$, otherwise its preimage in $\F$ would violate stability.
Assume that $\CC'$ has length $1$. Let $\F'$ be its preimage in $\F$. We have an extension
\[
0 \lra \F' \lra \F \lra \O_L \lra 0.
\]
Here $L \subset \P^2$ is a line and it is easy to see that $\F'$ gives a point in $\M(4,-1)$.
From the description of $\M(4,1)$ found in \cite{drezet-maican} we know that $h^0(\F') \le 1$,
hence $h^0(\F) \le 2$.

Assume, finally, that $\CC$ has no zero-dimensional torsion.
Then $\CC \isom \O_L(1)$ for a line $L \subset \P^2$.
We have $h^0(\F) \le 2$ unless the morphism $\F \to \O_L(1)$ is surjective on global sections.
In that case we can apply the horseshoe lemma to the extension
\[
0 \lra \O_C \lra \F \lra \O_L(1) \lra 0,
\]
to the standard resolution of $\O_C$ and, fixing an equation for $L$, say $X=0$, to the resolution
\[
0 \lra \O(-2) \stackrel{\eta}{\lra} 3\O(-1) \stackrel{\xi}{\lra} 2\O \lra \O_L(1) \lra 0.
\]
\[
\eta= \left[
\begin{array}{c}
X \\ Y \\ Z
\end{array}
\right], \quad  \quad \xi= \left[
\begin{array}{rcc}
-Y & X & 0 \\
-Z & 0 & X
\end{array}
\right].
\]
We obtain the exact sequence
\[
0 \lra \O(-2) \lra \O(-4) \oplus 3\O(-1) \lra 3\O \lra \F \lra 0.
\]
We saw above that this leads to the relation $h^0(\F(-1))=1$, which is contrary to our
hypothesis.
\end{proof}


\subsection{Description of the strata as quotients}

In subsection 4.1 we found that the moduli space $\M(5,0)$ can be decomposed into four
strata:
\begin{enumerate}
\item[$-$] an open stratum $X_0$ given by the condition $h^1(\F)=0$;
\item[$-$] a locally closed stratum $X_1$ of codimension $1$ given by the conditions \\
$h^0(\F(-1))=0$, $h^1(\F)=1$;
\item[$-$] a locally closed stratum $X_2$ of codimension $4$ given by the conditions \\
$h^0(\F(-1))=0$, $h^1(\F)=2$;
\item[$-$] the closed stratum $X_3$ given by the condition $h^0(\F(-1))>0$, consisting of sheaves
of the form $\O_C(1)$, where $C \subset \P^2$ is a quintic curve. $X_3$ is isomorphic to
$\P(S^5V^*)$.
\end{enumerate}
In the sequel $X_i$ will be equipped with the canonical reduced structure induced from $\M(5,0)$.
Let $W_0$, $W_1$, $W_2$ be the sets of morphisms $\f$ from \ref{4.1.2}, \ref{4.1.3}, respectively \ref{4.1.4}.
Each sheaf $\F$ giving a point in $X_i$, $i=0, 1, 2$, is the cokernel of a morphism $\f \in W_i$.
Let $\W_i$ be the ambient vector spaces of homomorphisms of sheaves containing $W_i$,
e.g. $\W_0 = \Hom(5\O(-2),5\O(-1))$.
Let $G_i$ be the natural groups of automorphisms acting by conjugation on $\W_i$.
In this subsection we shall prove that there exist a good quotient $W_0/\!/G_0$,
a categorical quotient of $W_1$ by $G_1$ and a geometric quotient $W_2/G_2$.
We shall prove that each quotient is isomorphic to the corresponding subvariety $X_i$.
We shall give concrete descriptions of $W_0/\!/G_0$ and $W_2/G_2$.

\begin{prop}
\label{4.2.1}
There exists a good quotient $W_0/\!/G_0$ and it is a proper open
subset inside $\N(3,5,5)$. Moreover, $W_0/\!/G_0$ is isomorphic to $X_0$.
In particular, $\M(5,0)$ and $\N(3,5,5)$ are birational.
\end{prop}

\begin{proof}
Let $\W_0^{ss} \subset \W_0$ denote the subset of morphisms that are semi-stable for the
action of $G_0$. This group is reductive, so by the classical geometric invariant theory
there is a good quotient $\W_0^{ss}/\!/G_0$, which is nothing but the Kronecker
moduli space $\N(3,5,5)$. 
According to King's criterion of semi-stability \cite{king}, a morphism $\f \in \W_0$
is semi-stable if and only if it is not in the $G_0$-orbit of a morphism of the form
\[
\left[
\begin{array}{cc}
\star & \psi \\
\star & 0
\end{array}
\right] \quad \text{for some} \quad \psi \colon (m+1)\O(-2) \lra m\O(-1), \quad 0 \le m \le 4.
\]
It is now clear that $W_0$ is the subset of injective morphisms inside $\W_0^{ss}$,
so it is open and $G_0$-invariant. In point of fact, it is easy to check that $W_0$ is the preimage
in $\W_0^{ss}$ of a proper open subset inside $\W_0^{ss}/\!/G_0$.
This subset is the good quotient of $W_0$ by $G_0$.

We shall now prove the injectivity of the canonical map $W_0/\!/G_0 \to X_0$.
Consider the map $\upsilon \colon W_0 \to X_0$ sending $\f$ to the stable-equivalence class of its
cokernel. Consider a properly semi-stable sheaf $\F = \Coker(\f)$, $\f \in W_0$, giving
a point $[\F]$ in $X_0$. For simplicity of notations we assume that $\F$ has a Jordan-H\"older
filtration of length $2$, i.e. there is an extension
\[
0 \lra \F' \lra \F \lra \F'' \lra 0
\]
of stable sheaves $\F' \in \M(r,0)$ and $\F'' \in \M(s,0)$. From the proof of \ref{4.1.2} we see that
there are resolutions
\[
0 \lra r\O(-2) \stackrel{\f'}{\lra} r\O(-1) \lra \F' \lra 0, 
\]
\[
0 \lra s\O(-2) \stackrel{\f''}{\lra} s\O(-1) \lra \F'' \lra 0.
\]
Using the horseshoe lemma we see that $\f$ is in the orbit of a morphism represented by a matrix
of the form
\[
\left[
\ba{cc}
\f'' & 0 \\
\star & \f'
\ea
\right].
\]
It is clear that $\f'' \oplus \f'$ is in the closure of the orbit of $\f$.
Thus $\upsilon^{-1}([\F])$ is a union of orbits, each containing $\f'' \oplus \f'$ in its closure.
It follows that the preimage of $[\F]$ in $W_0/\!/G_0$ is a point.
Thus far we have proved that the canonical map $W_0/\!/G_0 \to X_0$ is bijective.
To show that it is an isomophism we use the method of 3.1.6 \cite{drezet-maican}.
We must produce resolution \ref{4.1.2} starting from the Beilinson spectral sequence for $\F$.
Diagram (2.2.3) \cite{drezet-maican} for $\F$ reads
\[
\xymatrix
{
5\O(-2) \ar[r]^-{\f_1} & 5\O(-1) & 0 \\
0 & 0 & 0
}.
\]
From the exact sequence (2.2.5) \cite{drezet-maican} we deduce that $\f_1$ is injective
and its cokernel is isomorphic to $\F$.
\end{proof}

\begin{prop}
\label{4.2.2}
There exists a categorical quotient of $W_1$ modulo $G_1$, which is isomorphic to $X_1$.
\end{prop}

\begin{proof}
Let $\upsilon \colon W_1 \to X_1$ be the canonical map sending a morphism $\f$ 
to the stable-equivalence class of its cokernel.
As in the proof of \ref{4.2.1}, one can check that the preimage of an arbitrary point in $X_1$
under $\upsilon$ is a union of $G_1$-orbits whose closures have non-empty intersection.
This shows that $\upsilon$ is bijective.
To show that $\upsilon$ is a categorical quotient map we proceed as at 3.1.6 \cite{drezet-maican}.
Given $\F$ in $X_1$, we need to produce resolution \ref{4.1.3} starting from the Beilinson spectral
sequence.
We shall work, instead, with the dual sheaf $\G = \F^{\D}(1)$, which gives a point in $\M(5,5)$.
Diagram (2.2.3) \cite{drezet-maican} for $\G$ takes the form
\[
\xymatrix
{
\O(-2) & 0 & 0 \\
\O(-2) \ar[r]^-{\f_3} & 5\O(-1) \ar[r]^-{\f_4} & 5\O
}.
\]
The exact sequence (2.2.5) \cite{drezet-maican} reads
\[
0 \lra \O(-2) \lra \Coker(\f_4) \lra \G \lra 0.
\]
Repeating the arguments from the proof of \ref{3.2.4} it is easy to see that we may write
\[
\f_3 = \left[
\begin{array}{c}
X \\ Y \\ Z \\ 0 \\ 0
\end{array}
\right] \qquad \text{and} \qquad \f_4 = \left[
\begin{array}{ccccc}
-Y & \phantom{-}X & 0 & \star & \star \\
-Z & \phantom{-}0 & X & \star & \star \\
\phantom{-}0 & -Z & Y & \star & \star \\
\phantom{-}0 & \phantom{-}0 & 0 & \psi_{11} & \psi_{12} \\
\phantom{-}0 & \phantom{-}0 & 0 & \psi_{21} & \psi_{22}
\end{array}
\right].
\]
If the morphism $\psi \colon 2\O(-1) \to 2\O$ represented by the matrix $(\psi_{ij})_{1 \le i,j \le 2}$ were not injective,
then $\psi$ would be equivalent to a morphism represented by a matrix with a zero-row or a zero-column.
From the snake lemma it would follow that $\Coker(\f_4)$ has a subsheaf $\S$
with Hilbert polynomial $P(t)=3t+4$ or $2t+3$.
This sheaf would map injectively to $\G$
because $\S \cap \O(-2)= \{ 0 \}$. The semi-stability of $\G$ would be violated.
We deduce that $\psi$ is injective and we obtain the extension
\[
0 \lra \O(1) \lra \Coker(\f_4) \lra \Coker(\psi) \lra 0,
\]
which yields the resolution
\[
0 \lra 2\O(-1) \lra 2\O \oplus \O(1) \lra \Coker(\f_4) \lra 0.
\]
Combining with the resolution of $\G$ from above we obtain the exact sequence
\[
0 \lra \O(-2) \oplus 2\O(-1) \lra 2\O \oplus \O(1) \lra \G \lra 0.
\]
By duality, this corresponds to resolution \ref{4.1.3} for $\F$.
\end{proof}

\begin{prop}
\label{4.2.3}
There exists a geometric quotient $W_2/G_2$ and it is a proper open subset
inside a fibre bundle over $\P^2 \times \P^2$ with fibre $\P^{18}$.
\end{prop}

\begin{proof}
The construction of $W_2/G_2$ is analogous to the construction of the geometric quotient
$W_1/G_1$ from \ref{2.2.2}.
Let $W_2' \subset \W_2$ be the locally closed subset given by the conditions
$\f_{12}=0$, $\f_{11}$ has linearly independent entries, $\f_{22}$ has linearly independent
entries. The pairs of morphisms $(\f_{11}, \f_{22})$ form an open subset
\[
U \subset \Hom(2\O(-3),\O(-2)) \times \Hom(\O(-1),2\O).
\]
The reductive subgroup ${G_2}_{\text{red}}$ of $G_2$ acts on $U$ with kernel $S$
and $U/({G_2}_{\text{red}}/S)$ is isomorphic to $\P^2 \times \P^2$.
Note that $W_2'$ is the trivial bundle over $U$ with fibre $\Hom(2\O(-3), 2\O)$.
The subset $\Sigma \subset W_2'$ given by the condition
\[
\f_{21} = \f_{22} u + v \f_{11}, \quad u \in \Hom(2\O(-3), \O(-1)), \quad v \in \Hom(\O(-2),2\O),
\]
is a subbundle. The quotient bundle $Q'$ has rank $19$ and descends to a vector bundle
$Q$ on $U/({G_2}_{\text{red}}/S)$ as at \ref{2.2.2}. Then $\P(Q)$ is the geometric quotient
$(W_2' \setminus \Sigma)/G_2$.

$W_2$ is the open invariant subset of injective morphism inside $W_2' \setminus \Sigma$.
It is a proper subset as, for instance, the morphism represented by the matrix
\[
\left[
\ba{ccc}
X\phantom{^3} & Y & \phantom{-}0 \\
Z^3 & 0 & \phantom{-}Y \\
0\phantom{^3} & Z^3 & -X
\ea
\right]
\]
is in $W_2 \setminus \Sigma$ but is not injective. We conclude that $W_2/G_2$
exists and is a proper open subset inside $\P(Q)$.
\end{proof}

\begin{prop}
\label{4.2.4}
The geometric quotient $W_2/G_2$ is isomorphic to $X_2$.
\end{prop}

\begin{proof}
The canonical morphism $W_2/G_2 \to X_2$ is easily seen to be injective,
there being no properly semi-stable sheaves in $X_2$, cf. \ref{4.1.5}.
To show that it is an isomorphism we must construct resolution \ref{4.1.4} starting from the
Beilinson spectral sequence of a sheaf $\F$ in $X_2$.
We prefer to work, instead, with the dual sheaf $\G = \F^{\D}(1)$, which gives a point in $\M(5,5)$.
Diagram (2.2.3) \cite{drezet-maican} for $\G$ takes the form
\[
\xymatrix
{
2\O(-2) \ar[r]^-{\f_1} & \O(-1) & 0 \\
2\O(-2) \ar[r]^-{\f_3} & 6\O(-1) \ar[r]^-{\f_4} & 5\O
}.
\]
As in the proof of \ref{3.2.4}, we see that $\Coker(\f_1)$ is the structure sheaf of a point $x \in \P^2$
and $\Ker(\f_1) \isom \O(-3)$. The exact sequence (2.2.5) \cite{drezet-maican} reads
\[
0 \lra \O(-3) \stackrel{\f_5}{\lra} \Coker(\f_4) \lra \G \lra \C_x \lra 0.
\]
We see from this that $\Coker(\f_4)$ has no zero-dimensional torsion.
The exact sequence (2.2.4) \cite{drezet-maican} reads
\[
0 \lra 2\O(-2) \stackrel{\f_3}{\lra} 6\O(-1) \stackrel{\f_4}{\lra} 5\O \lra \Coker(\f_4) \lra 0.
\]
We claim that $\f_3$ is equivalent to the morphism represented by the matrix
\[
\left[
\ba{cccccc}
X & Y & Z & 0 & 0 & 0 \\
0 & 0 & 0 & X & Y & Z
\ea
\right]^{\T}.
\]
Firstly, we show that any matrix representing a morphism equivalent to $\f_3$
has three linearly independent entries on each column.
For this we use the fact that the only morphism
from the structure sheaf of a point to $\Coker(\f_4)$ is the zero-morphism
and we argue as in the proof that the vector space $H$ from \ref{2.1.4} has dimension $3$.
Thus $\f_3$ has one of the four canonical forms given in the proof of \ref{3.2.5}.
Three of these can be eliminated as in the proof of \ref{3.1.3}. The argument, we recall,
uses the fact that the map $5\O \to \Coker(\f_4)$ is injective on global sections
as well as the fact that the only morphism $\O_L(1) \to \Coker(\f_4)$
for any line $L \subset \P^2$ is the zero-morphism. Indeed, such a morphism must factor through
$\f_5$ because the composed morphism $\O_L(1) \to \Coker(\f_4) \to \G$ is zero.
This follows from the fact that both $\O_L(1)$ and $\G$ are semi-stable and $p(\O_L(1)) > p(\G)$.

Next we describe $\f_4$. Its matrix cannot be equivalent to a matrix having a zero-row.
Indeed, if this were the case, then $\Coker(\f_4)$ would be isomorphic to $\O \oplus \CC$, where
$\CC$ is a torsion sheaf with resolution
\[
0 \lra 2\O(-2) \lra 6\O(-1) \lra 4\O \lra \CC \lra 0.
\]
We have $P_{\CC}(t)=2t+4$ and $\CC$ maps injectively to $\G$
because $\CC \cap \O(-3)= \{ 0 \}$.
The semi-stability of $\G$ is violated.
We conclude that $\f_4$ has the form
\[
\left[
\ba{cc}
\xi & 0 \\
\star & \psi
\ea
\right],
\]
where $\xi$ is a morphism as in the proof of \ref{4.1.6} and $\psi$ is equivalent to the
morphism $\f_{12}$ also from {4.1.6}. We have exact sequences
\[
0 \lra \O(-2) \lra 3\O(-1) \stackrel{\xi}{\lra} 2\O \lra \O_L(1) \lra 0,
\]
\[
0 \lra \O(-2) \lra 3\O(-1) \stackrel{\psi}{\lra} 3\O \lra \O(1) \lra 0.
\]
Recall that the greatest common divisor of the maximal minors of $\xi$ is a linear form.
The line $L \subset \P^2$ is the zero-locus of this form.
From the snake lemma we obtain an extension
\[
0 \lra \O(1) \lra \Coker(\f_4) \lra \O_L(1) \lra 0,
\]
hence a resolution
\[
0 \lra \O \lra 2\O(1) \lra \Coker(\f_4) \lra 0.
\]
Note that $\f_5$ lifts to a morphism $\O(-3) \to 2\O(1)$, so we arrive at the exact sequence
\[
0 \lra \O(-3) \oplus \O \lra 2\O(1) \lra \G \lra \C_x \lra 0.
\]
From the horseshoe lemma we obtain a resolution
\[
0 \lra \O(-3) \lra \O(-3) \oplus 2\O(-2) \oplus \O \lra \O(-1) \oplus 2\O(1) \lra \G \lra 0.
\]
$\H^1(\G)$ vanishes, hence $\O(-3)$ can be cancelled to yield the dual of resolution \ref{4.1.4}.
\end{proof}


\subsection{Geometric description of the strata}

Let $X_0^s$ denote the subset of $X_0$ of isomorphism classes of stable sheaves.
Given $\f \in W_0$, we write its domain $\O(-2) \oplus 4\O(-2)$ and denote by $\f_{12}$
the restriction of $\f$ to the second component. Let $Y_0$ be the open subset of $X_0$
of stable-equivalence classes of sheaves $\F$ that occur as cokernels
\[
0 \lra \O(-2) \oplus 4\O(-2) \stackrel{\f}{\lra} 5\O(-1) \lra \F \lra 0
\]
in which the maximal minors of $\f_{12}$ have no common factor.

\begin{prop}
\label{4.3.1}
The sheaves in $Y_0$ have the form $\J_Z(3)$, where $Z \subset \P^2$
is a zero-dimensional scheme of length $10$ not contained in a cubic curve, contained in a quintic
curve $C$, and $\J_Z \subset \O_C$ is its ideal sheaf.

The generic sheaves in $X_0^s$ have the form $\O_C(3)(-P_1-\ldots -P_{10})$,
where $C \subset \P^2$ is a smooth quintic curve and $P_i$, $1 \le i \le 10$, are distinct points on $C$
not contained in a cubic curve.
\end{prop}

\begin{proof}
Consider the sheaf $\F =\Coker(\f)$, where the maximal minors of $\f_{12}$ have no common factor.
According to 4.5 and 4.6 \cite{modules-alternatives}, $\Coker(\f_{12}) \isom \I_Z(3)$,
where $Z \subset \P^2$ is a zero-dimensional scheme of length $10$, not contained in a cubic curve.
Conversely, any $\I_Z(3)$ is the cokernel of some morphism $\f_{12} \colon 4\O(-2) \to 5\O(-1)$
whose maximal minors have no common factor. It now follows, as at \ref{2.3.4}(i),
that $\F \isom \J_Z(3)$.

The claim about generic stable sheaves follows from the fact that any line bundle
on a smooth curve is stable.
\end{proof}

\begin{prop}
\label{4.3.2}
The sheaves $\F$ in $X_2$ are precisely the non-split extension sheaves of the form
\[
0 \lra \J_x(1) \lra \F \lra \C_z \lra 0,
\]
where $\J_x \subset \O_C$ is the ideal sheaf of a point $x$ on a quintic curve $C \subset \P^2$
and $\C_z$ is the structure sheaf of a point $z \in C$.
When $x=z$ we exclude the possibility $\F \isom \O_C(1)$.

The generic sheaf in $X_2$ has the form $\O_C(1)(P-Q)$, where $C \subset \P^2$ is a smooth
quintic curve and $P, Q$ are distinct points on $C$. In particular, the closure of $X_2$ contains $X_3$.
\end{prop}

\begin{proof}
To get the extension from the proposition we apply the snake lemma to a diagram similar to the
diagram from the proof of \ref{2.3.2}. Here
\[
\f = \left[
\ba{ccc}
u_1 & u_2 & 0 \\
\star & \star & v_1 \\
\star & \star & v_2
\ea
\right],
\]
$C$ is given by the equation $\det(\f)=0$,
$x$ is the point given by the equations  $v_1=0, v_2=0$ and
$z$ is the point given by the equations $u_1=0, u_2=0$.
To prove the converse we combine the resolutions
\[
0 \lra \O(-4) \lra \I_x(1) \lra \J_x(1) \lra 0
\]
and
\[
0 \lra \O(-4) \lra 2\O(-3) \lra \O(-2) \lra \C_z \lra 0
\]
into the resolution
\[
0 \lra \O(-4) \lra \O(-4) \oplus 2\O(-3) \lra \O(-2) \oplus \I_x(1) \lra \F \lra 0.
\]
If $x \neq z$, then $\Ext^1(\C_x, \I_x(1)) = 0$ and the arguments from the proof of \ref{2.3.2}
show that the map $\O(-4) \to \O(-4)$ in the above complex is non-zero.
Canceling $\O(-4)$ we get the exact sequence
\[
0 \lra 2\O(-3) \lra \O(-2) \oplus \I_x(1) \lra \F \lra 0
\]
from which we immediately obtain resolution 4.1.4.
A priori we have two possibilities: $h^0(\F)=2$ or $3$.
In the first case the map $\O(-4) \to \O(-4)$ is non-zero and we are done.
In the second case we can combine the resolutions
\[
0 \lra \O(-4) \oplus \O(-1) \lra 2\O \lra \J_x(1) \lra 0
\]
and
\[
0 \lra \O(-2) \lra 2\O(-1) \lra \O \lra \C_z \lra 0
\]
into the resolution
\[
0 \lra \O(-2) \lra \O(-4) \oplus 3\O(-1) \lra 3\O \lra \F \lra 0.
\]
We saw in the proof of \ref{4.1.6} how this resolution leads to the conclusion
that $\F$ be isomorphic to $\O_C(1)$ for a quintic curve $C \subset \P^2$.
This possibility is excluded by hypothesis.

If $C$ is a smooth quintic curve and $P$ converges to $Q$, then $\O_C(1)(P-Q)$ represents a point in
$X_2$ converging to the point in $X_3$ represented by $\O_C(1)$. This shows that
$X_3 \subset \overline{X}_2$.
\end{proof}

\begin{prop}
\label{4.3.3}
$\{ X_0, X_1, X_2, X_3 \}$ represents a stratification of $\M(5,0)$ by locally closed
irreducible subvarieties of codimensions $0, 1, 4, 6$.
\end{prop}

\begin{proof}
We saw above that $X_3$ lies in $\overline{X}_2$
and we know that $X_0$ is dense in $\M(5,0)$.
Thus we only need to show that $X_2$ is included in the closure of $X_1$.
For this we shall apply the method of
3.2.3 \cite{drezet-maican}.
Consider the open subset $X = \M(5,0) \setminus X_3$ of stable-equivalence classes
of sheaves satisfying the condition $h^0(\F(-1))=0$.
Using the Beilinson monad for $\F(-1)$ we see that
$X$ is parametrised by an open subset $M$ inside the space of monads of the form
\[
0 \lra 10 \O(-1) \stackrel{A}{\lra} 15 \O \stackrel{B}{\lra} 5\O(1) \lra 0.
\]
The automorphism of $\M(5,0)$ taking the stable-equivalence class of a sheaf $\F$ to the stable-equivalence class
of the dual sheaf $\F^{\D}$ leaves $X$ invariant.
Thus, in view of Serre duality, he have $h^1(\F(1))=h^0(\F^{\D}(-1))=0$ for all $\F$ in $X$.
This allows us to deduce that the map $\Phi$ defined by $\Phi(A,B)=B$ has surjective differential at every point
in $M$. As at 3.2.3 \cite{drezet-maican}, this leads to the conclusion that $X_2$ is included in the closure
of $X_1$ in $X$, hence $X_2$ is included in the closure of $X_1$ in $\M(5,0)$.
\end{proof}

\end{document}